\documentclass[12pt, oneside]{book} 
\usepackage{latexsym}
\usepackage{amssymb}
\usepackage[toc,page]{appendix}
\usepackage{enumerate}
\usepackage{chngcntr}
\usepackage{commath}
\usepackage{esint}
\usepackage{enumitem}
\usepackage{multicol}
\usepackage{comment}
\usepackage{amsfonts}
\usepackage[labelfont=bf]{caption}
\usepackage{amsmath,amsthm}
\usepackage{cases}
\usepackage[export]{adjustbox}[2011/08/13]
\usepackage{graphicx}
\usepackage{tabularx, cellspace}
\usepackage[paper=a4paper,left=30mm,right=20mm,top=25mm,bottom=30mm]{geometry}

\newenvironment{@abssec}[1]{%
\if@twocolumn

\section*{#1}%
\else

\vspace{.05in}\footnotesize
\parindent .2in
{\upshape\bfseries #1. }\ignorespaces
\fi}

{\if@twocolumn\else\par\vspace{.1in}\fi}

\newcommand\keywordsname{Key words}
\newcommand\AMSname{AMS subject classifications}
\newcommand\AMname{AMS subject classification}
\newcommand\restr[2]{{
\left.\kern-\nulldelimiterspace 
#1 
\vphantom{|} 
\right|_{#2} 
}}
\newtheorem{theorem}{Theorem}[chapter]
\newtheorem{lemma}[theorem]{Lemma}
\newtheorem{corollary}[theorem]{Corollary}
\newtheorem{proposition}[theorem]{Proposition}
\newtheorem{remark}[theorem]{Remark}

\newtheorem{example}[theorem]{Example}
\newtheorem{test}{Theorem}

\newtheorem{thm}{Theorem}

\newtheorem{lem}{Lemma}

\renewcommand{\theequation}{\arabic{chapter}.\arabic{equation}}
\newcommand\setbld[2]{\left\{ #1 \;\middle |\; #2\right\}}
\counterwithout{figure}{chapter}

\newcommand{\NN}{\mathbb{N}}

\newcommand{\RR}{\mathbb{R}}

\renewcommand{\SS}{\mathbb{S}}
\newcommand{\sss}{{\SS^{N-1}}}
\newcommand{\id}{{{\rm Id}}}
\newcommand{\mapto}{\mapsto}

\renewcommand{\theequation}{\arabic{chapter}.\arabic{equation}}

\def\XXint#1#2#3{{\setbox0=\hbox{$#1{#2#3}{\int}$}
\vcenter{\hbox{$#2#3$}}\kern-.5\wd0}}

\newcommand{\dist}{\mathop{\mathrm{dist}}} 
\newcommand{\link}{\mathop{\circ\kern-.35em -}}

\newcommand{\ol}{\overline}
\newcommand{\pa}{\partial}

\newcommand{\tr}{\mathop{\mathrm{tr}}}

\newcommand{\dv}{{\mathop{\mathrm{div}}}}

\newcommand{\gr}{{\nabla}}
\newcommand{\ali}{\infty}
\newcommand\wt[1]{\widetilde{#1}}

\newcommand{\al}{\alpha}
\newcommand{\be}{\beta}
\newcommand{\ga}{\gamma} 
\newcommand{\Ga}{\Gamma}

\newcommand{\De}{\Delta}

\newcommand{\eps}{\varepsilon}
 
\newcommand{\la}{\lambda}

\newcommand{\te}{\theta}
\newcommand{\Te}{\Theta}

\newcommand{\om}{\omega}
\newcommand{\Om}{\Omega}
\newcommand{\rn}{{{\mathbb{R}}^N}}
\newcommand{\sg}{\sigma}
\newcommand{\vol}{{\rm Vol}}
\newcommand{\per}{{\rm Per}}
\newcommand{\bc}{{\rm Bar}}

\newcommand{\A}{\mathcal{A}}
\newcommand{\cA}{\mathcal{A}}
\newcommand{\cB}{\mathcal{B}}
\newcommand{\C}{\mathcal{C}}
\newcommand{\cC}{\mathcal{C}}

\newcommand{\cE}{\mathcal{E}}
\newcommand{\cF}{\mathcal{F}}
\newcommand{\cG}{{\mathcal G}}
\newcommand{\cH}{{\mathcal H}}

\newcommand{\cJ}{{\mathcal J}}

\newcommand{\cO}{{\mathcal O}}

\newcommand{\cS}{{\mathcal S}}

\newcommand{\cU}{{\mathcal U}}

\newcommand{\PP}{\mathbb{P}}
\newcommand{\HH}{\mathbb{H}}
\newcommand{\YY}{\mathbb{Y}}

\newcommand{\dn}{{\pa_n}}

\newcommand{\cdottone}{{\boldsymbol{\cdot}}}
\newcommand{\ri}{-}
\newcommand{\ro}{+}

\newcommand{\hoi}{H_0^1}
\newcommand{\hi}{H^1}
\newcommand{\rnrn}{(\rn,\rn)}

\newcommand{\pato}{\restr{\pa_t}{t=0}}
\newcommand{\dato}{\restr{\frac{d}{dt}}{t=0}}
\newcommand{\pp}[1]{\left(#1\right)}
\newcommand{\Ome}{(\Om)}
\newcommand{\grt}{{{\gr}_\tau}}
\newcommand{\inv}{^{-1}}
\newcommand{\sm}{\setminus}
\newcommand{\ton}{\text{ on }}
\newcommand{\tor}{\text{ or }}
\newcommand{\tand}{\text{ and }}
\newcommand{\tin}{\text{ in }}
\newcommand{\tif}{\text{ if }}
\newcommand{\tas}{\text{ as }}
\newcommand{\tfor}{\text{ for }}
\newcommand{\tforall}{\text{ for all }}

\title{Analysis of two-phase shape optimization problems by means of shape derivatives}

\author{Lorenzo Cavallina\thanks{Research Center for Pure and Applied Mathematics,Graduate School of Information Sciences, Tohoku University, Sendai}}

\begin{document}

\begin{titlepage}
    \begin{center}
        \vspace*{1cm}
        
        \LARGE
        \textbf
        {Analysis of two-phase shape optimization problems by means of shape derivatives}
        
        \vspace{0.5cm}
        \vfill
        \large 
        A thesis submitted for the degree of\\
        Doctor of Philosophy  \\
        
        \vspace{1.0cm}
        by\\
      	\vspace{1.0cm}
        
        \LARGE
        
        
        {Lorenzo Cavallina}
        
        \vfill
        
        
        
        
        \large
        Division of Mathematics\\
		Graduate School of Information Sciences\\
		Tohoku University\\
        \vspace{0.8cm}
        
        July 2018
        
    \end{center}
\end{titlepage}

\tableofcontents
\vspace*{-0.5cm}

\pagestyle{plain}
\thispagestyle{plain}

\chapter*{Notations}\label{sec-preliminaries}
\addcontentsline{toc}{chapter}{Notations}
\subsection*{Euclidean space}
\begin{tabularx}{\linewidth}{@{}SlX@{}}
$\NN$ & the set of positive integers $\{1,2,3\dots\}$\\
$\RR$ & the set of real numbers \\
$\rn$ & the $N$-dimensional Euclidean space, $N\ge2$ \\
$a \cdot b$ & the inner product in $\rn$, $\sum_{i=1}^N a_i b_i$\\
$\norm{x}$ & the Euclidean norm, $\sqrt{x_1^2+\dots x_N^2}$, sometimes also used to denote a generic norm of some Banach space \\
$\id$ & the identity map $x\mapto x$\\
$B_r$ & the open ball with radius $r>0$ centered at the origin\\
$\ol A$ & the closure of the open set $A$\\
$\pa A$ & the boundary of the open set $A$, given by $\ol A \sm A$\\
$\int_\Omega f$ & the integral of $f$ over $\Omega$ with respect to the $N$-dimensional Lebesgue measure\\
$\int_{\partial\Omega} f$ & the (surface) integral of $f$ over $\pa\Om$ with respect to the $(N-1)$-dimensional Hausdorff measure
\end{tabularx}\\
\subsection*{Matrix notation}
\begin{tabularx}{\linewidth}{@{}SlX@{}}
$\RR^{N\times N}$ & the set of real square matrices\\
$I$ & the identity matrix\\
$\det A$ & the determinant of the square matrix $A$\\
$\tr A$ & the trace of the square matrix $A$\\
$A^T$ & the transpose of $A$, $(A^T)_{i,j}=A_{j,i}$\\
$A\inv$ & the inverse of an invertible square matrix $A$\\
$A^{-T}$ & the transpose of the inverse of $A$\\ 
\end{tabularx}\\
\subsection*{Differential operators} 
\begin{tabularx}{\linewidth}{@{}SlX@{}}
$\gr f$ & the gradient of the function $f$ with respect to the space variables $x_i$\\
$Dw$ & the Jacobian matrix of the vector field $w$, $(Dw)_{i,j}= \frac{\pa w_i}{\pa x_j}$\\
$\dv w$ & the divergence of the vector field $w$, given by $\tr D w$\\
$D^2 f$ & the Hessian matrix of the function $f$, given by $(D^2 f)_{i,j}=\frac{\pa^2 f}{\pa x_i \pa x_j}$\\
$\De f$ & the Laplace operator of the function $f$, given by $\tr D^2 f$\\
$\gr_\tau f$ & the tangential gradient of f, see Appendix A\\
$\dv_\tau w$ & the tangential divergence of w, see Appendix A\\
$\De_\tau f$ & the Laplace--Beltrami operator of f, see Appendix A\\
$\pa_s f$ & the partial derivative of $f$ with respect to the variable $s$
\end{tabularx}

\subsection*{Function spaces}
\begin{tabularx}{\linewidth}{@{}SlX@{}}
$L^p(\Om,\RR^M)$ & the space of p-summable functions $\Om\to\RR^M$, $1\le p\le\ali$, endowed with the usual norm $\norm{\cdottone}_p$\\
$L^p(\Om)$ & abbreviate notation for $L^p(\Om,\RR)$\\
$W^{k,p}(\Om,\RR^M)$ & the space of functions $\Om\to \RR^M$ whose partial derivatives up to the $k$-th order are p-summable, $1\le p\le \ali$, endowed with usual norm $\norm{\cdottone}_{k,p}$\\
$W^{k,p}(\Om)$ & abbreviate notation for $W^{k,p}(\Om,\RR)$\\
$H^1(\Om)$ & alternative notation for $W^{1,2}(\Om)$\\
$H_0^1(\Om)$ & the subset of $H^1(\Om)$ of functions with vanishing trace on $\pa\Om$\\
$\C^k(\Om)$ & the class of functions that are continuously differentiable $k$ times\\
$\C^{k,\ali}(\Om)$ & the space $\C^{k}(\Om)\cap W^{k,\ali}(\Om)$ endowed with the norm $\norm{\cdottone}_{k,\ali}$\\
$\C^{k+\al}(\Om)$ & the subclass of $\C^k(\Om)$ made of functions whose $k$-th partial derivatives are H\" older continuous with exponent $\al\in(0,1]$
\end{tabularx}

\chapter{Introduction and main results}
\label{ch intro}

Let $D\subset\Om$ be a pair of bounded domains in the $N$-dimensional Euclidean space $\rn$ ($N\ge2$). Moreover, assume that $\ol{D}\subset \Om$. In this way $\Om$ gets partitioned into two subsets: $D$ and $\Om\sm D$ (from now on they will be referred to as \emph{core} and \emph{shell} respectively). Take two (possibly distinct) positive constants $\sg_c$ and $\sg_s$ and set
\begin{equation}
\sg(x)=\sg_{D,\Om}(x):=\left\{
\begin{aligned}
\sg_c \quad& \text{ for } x\in D, & \text{{\em (core)}}\\
\sg_s \quad & \text { for } x\in \Om\setminus D & \text{{\em (shell)}}. 
\end{aligned}
\right.
\end{equation}
Consider the following boundary value problem 
\begin{equation}\label{pb u formal}
\left\{
\begin{aligned}
-\dv(\sg\gr u)&=1 &\quad \text{ in }\Om,\\
u&=0&\quad \text{ on }\pa\Om.
\end{aligned}
\right.
\end{equation}
We say that a function $u\in\hoi\Ome$ is a solution of \eqref{pb u formal} if it verifies the following weak formulation:
\begin{equation}\label{pb weak}
\int_{\Om} \sg \gr u \cdot \gr \psi = \int_\Om \psi\quad \tforall \psi\in\hoi\Ome.
\end{equation}
Since $\sg$ attains a different value at each \emph{phase} ($D$ and $\Om\sm D$), problems like \eqref{pb u formal} are usually called \emph{two-phase} problems (of course, the term \emph{multi-phase} is also used, when the phases are more than two).
In the sequel, 
the subscripts $c$ and $s$ will be used to denote the restriction of any function to the \emph{core} and the \emph{shell} respectively, moreover we will also employ the use of the notation $[f]:=f_c-f_s$ to refer to the jump of a function $f\in L^2(\Om)\cap H^1(D)\cap H^1(\Om\sm\ol{D})$ along the interface $\pa D$. 
\begin{figure}[h]
\centering
\includegraphics[width=0.5\linewidth,center]{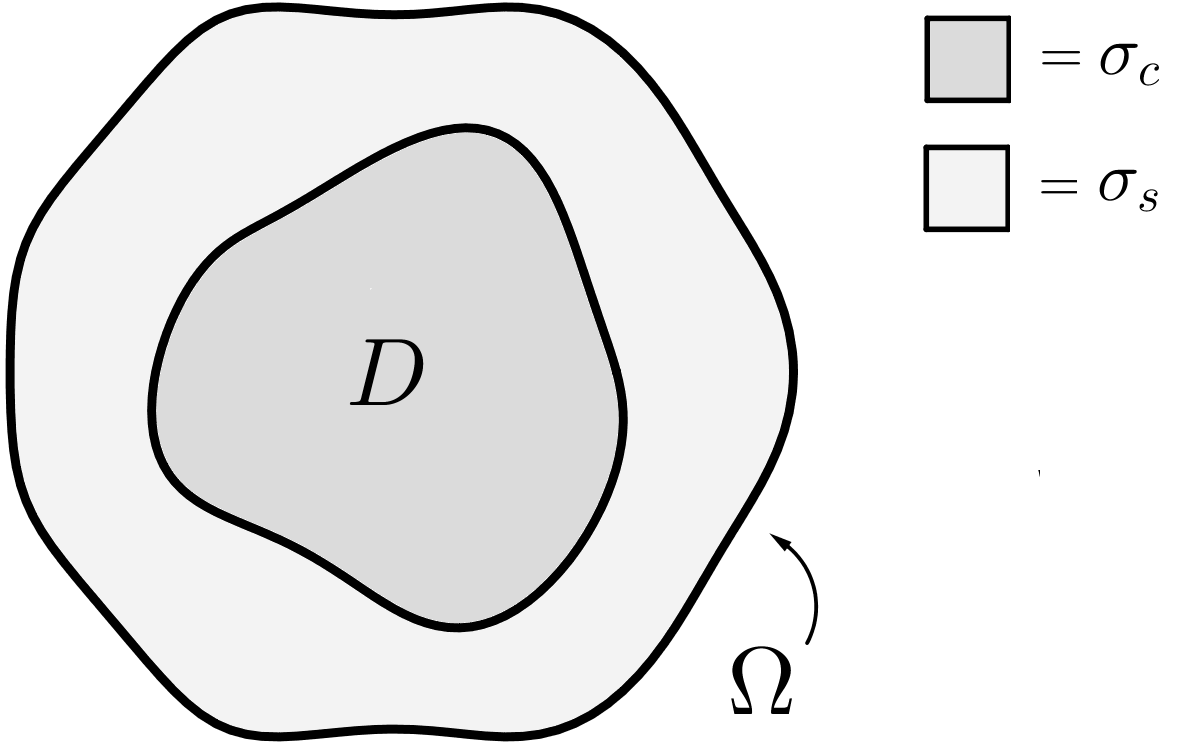}
\caption{Our problem setting
} 
\label{pbsett}
\end{figure}
When $\pa D$ and $\pa \Om$ are at least of class $\C^2$, then the solution $u$ of problem \eqref{pb weak} enjoys higher regularity, namely $u\in \hoi\Ome\cap H^2(D)\cap H^2(\Om\sm\ol{D})$ (see \cite[Theorem 1.1]{athastra96}). 
Under these regularity assumptions on $D$ and $\Om$, problem \eqref{pb weak} admits the following alternative formulation (see \cite{athastra96}):
\begin{equation}\label{pb u transmission}
\left\{
\begin{aligned}
-\sg \De u&=1 &\quad \text{ in }D\cup (\Om\sm\ol{D}),\\
[\sg \dn u]&=0&\quad\ton\pa D,\\
[u]&=0&\quad \ton \pa D,\\
u&=0&\quad \text{ on }\pa\Om.
\end{aligned}
\right.
\end{equation}
Here, the letter $n$ is used indistinctly to refer to both the outward unit normal to $\pa D$ and $\pa \Om$, and hence we will agree that, for smooth enough $f$, $\dn f=\gr f\cdot n$ stands the usual normal derivative (in the outward direction). 
The conditions 
\begin{equation}\label{tc}
[\sg\dn u]=[u]=0 \quad\ton \pa D
\end{equation}
are usually called \emph{transmission conditions} in the literature and therefore problems like \eqref{pb u transmission}, where the jump along the interface is prescribed, are usually referred to as \emph{transmission problems}.

When $(D,\Om)=(B_R,B_1)$, with $0<R<1$, then problem \eqref{pb u transmission} admits an explicit radial solution:
\begin{equation}\label{u}
u(x)=\begin{cases}\displaystyle
\frac{1-R^2}{2N\sg_s}+\frac{R^2-|x|^2}{2N\sg_c}&\quad |x|\in[0,R],\\
\vspace*{-6mm}\\
\displaystyle\frac{1-|x|^2}{2N\sg_s}&\quad|x|\in (R,1].
\end{cases}
\end{equation}

One of the main topics of this work is the study of the following functional
\begin{equation}\label{E}
E(D,\Om):=\int_\Om\sg|\gr u|^2=\int_\Om u,
\end{equation}
where $u$ is the solution of problem \eqref{pb u formal}.

Physically speaking, the function $u$, solution to \eqref{pb u formal}, plays the role of \emph{stress function} while its integral, $E(D,\Om)$, represents the \emph{torsional rigidity} of an infinitely long composite beam $\Om\times \RR$ made of two different materials, such that their distribution is the one given in Figure \ref{pbsett} for each cross section $\Om\times\{x_{N+1}\}$. The constants $\sg_c$ and $\sg_s$ are linked to the hardness of the materials in question (the smaller the constant, the harder the corresponding material, hence the higher the torsional rigidity $E(D,\Om)$, as one can see by \eqref{E} and \eqref{pb u formal}).

The one-phase case (i.e. when $D=\emptyset$) was studied by P\'olya by means of spherical rearrangement inequalities. In \cite{polya}, he proved that the ball maximizes the functional $E(\emptyset,\cdottone)$ among all open sets of a given volume (see Theorem \ref{thm polya}). 
Unfortunately, the methods employed by P\'olya do not generalize well to a two-phase setting. We decided to perform a local analysis of the functional $E(\cdottone,\cdottone)$ by means of shape derivatives. Inspired by the work of P\'olya, we aim to find the relationship between the radial symmetry of the configuration $(D,\Om)$ and local optimality for the functional $E$. The following theorem is one of our original results, concerning the first order shape derivative of $E$. From now on, let $(D_0,\Om_0)$ denote a pair of concentric balls $(B_R,B_1)$ with $0<R<1$. 
\begin{test}[\cite{cava2}]\label{thm1}
The pair $(D_0,\Omega_0)$ is a critical shape for the functional $E$ under the fixed volume constraint.
\end{test} 
Theorem \ref{thm1} can be improved by looking at second order shape derivatives. Exact computations are carried on with the aid of spherical harmonics at the end of Chapter \ref{ch 2-ph torsion}. We get the following symmetry breaking result.
\begin{test}[\cite{cava2}]\label{thm2}
The pair $(D_0,\Omega_0)$ is a local maximum for the functional $E$ under the fixed volume and barycenter constraint if $\sg_c\ge \sg_s$, otherwise it is a saddle shape.
\end{test}
Theorem \ref{thm2} shows a substantial difference between the one-phase maximization problem studied by P\'olya and our two-phase analogue. As a matter of fact, as we will show in Section \ref{sec 1-ph equivalence}, the one-phase functional $E(\phi,\cdottone)$ subject to the volume preserving constraint does not possess any critical point other than its global maximum. 

An obvious observation concerning the radially symmetric configuration $(D_0,\Om_0)$ is the following: the related stress function $u$ is itself radially symmetric and thus its normal derivative is constant on $\pa\Om_0$. It is well known that, when $D=\emptyset$ then this property characterizes the ball. In \cite{serrin} Serrin showed that if the stress function corresponding to $(\emptyset,\Om)$ has a normal derivative that is constant on the boundary $\pa\Om$, then $\Om$ must be a ball (see Theorem \ref{serrin thm}). The original proof by Serrin is based on an ingenious adaptation of Aleksandrov's reflection principle (see \cite{aleks}) nowadays referred to as \emph{method of moving planes}. This technique takes advantage of the invariance properties that characterize the Laplace operator and thus cannot be extended to our two-phase setting in any obvious way. It is not even clear at first glance whether an analogous characterization of the two-phase radial configuration $(D_0,\Om_0)$ holds true. For $\be\ge0$ and $\ga>0$ we consider the following overdetermined problem
\begin{equation}\label{pb 2-ph serrin}
\left\{
\begin{aligned}
{\rm div} (\sg \gr u)&=\be u - \ga & \text{ in }\Om,\\
u&=0 & \text{ on }\pa\Om,\\
\sg_s\pa_n u &= -d &\text{ on }\pa\Om, 
\end{aligned}
\right.
\end{equation}
where $d$ is a positive constant to be determined, depending on the geometry of the solution $(D,\Om)$. By means of a perturbation argument based on the implicit function theorem for Banach spaces, we manage to disprove the analogue of Serrin's result for the operator $-\dv\pp{\sg \gr\cdottone}$. 
\begin{test}[\cite{camasa}]\label{thm3}
For all domain $\Omega$ of class $\C^{2+\al}$ sufficiently close to $\Omega_0$ there exists a $\C^{2+\al}$-domain $D$ close to $D_0$ such that $(D,\Omega)$ is a solution to the overdetermined problem \eqref{pb 2-ph serrin}. In particular, there are infinitely many non radially symmetric solutions $(D,\Omega)$ of problem \eqref{pb 2-ph serrin} with $D$ and $\Omega$ of class $\C^{2+\al}$. 
\end{test}
As a final remark, notice that by a scaling argument, it is enough to prove Theorems \ref{thm1}--\ref{thm3} under the assumption that $\sg_s=1$. Therefore, in what follows {\bf we will always assume} {\boldmath{$\sg_s=1$}}.

As the title of this thesis suggests, shape derivatives will be our main tool. The concept of differentiating a shape functional with respect to a varying domain is actually really old. It dates back to the beginning of the $20$th century with the pioneering work of Hadamard \cite{hadamard}. It is virtually impossible to give an exhaustive list of all the contributions that have been made to this theory. We refer to the monographs \cite{SG, henrot} for some good introductory material on the classical theory of shape derivatives and shape optimization in general. Among others, we would like to refer to \cite{new, structure, simon} for their theoretical contributions and the related formalism. Moreover, one can not avoid mentioning works like \cite{conca} or \cite{sensitivity} where shape derivatives are used to investigate the local optimality of concentric balls for some two-phase eigenvalue problem (which is deeply related to the two-phase torsional rigidity functional $E$). As a final note, we might as well point the potential applications of this theory to the realm of numerical shape optimization (see for instance \cite{cz98} and \cite{kv obstacle}).

This thesis is organized as follows. In Chapter 2 we discuss the proofs of the classical results by P\'olya \cite{polya} and Serrin \cite{serrin} that take place in a one-phase setting. Chapter 3 provides the necessary theoretical background about shape derivatives. We used \cite{henrot} as the main reference here. Chapter 4 is devoted to the exposition of our first original result: the detailed analysis of the first and second order shape derivative of the functional $E$ and the subsequent proof of Theorems \ref{thm1}--\ref{thm2} (see \cite{cava, cava2}). Finally, the two-phase overdetermined problem of Serrin-type \eqref{pb 2-ph serrin} is analyzed in Chapter 5, where Theorem \ref{thm3} is proved by the implicit function theorem (see \cite{camasa}).

\chapter{Classical results in the one-phase setting}
\label{ch torsion+serrin}
\section{Optimal shape for the torsional rigidity}
\label{sec 1-ph torsion}
For all open sets $\Om\subset\rn$ of finite volume we denote by $\Om^*$ the ball centered at the origin whose volume agrees with $\Om$:
\begin{equation}\label{symmetrization}
\Om^*:=\setbld{x\in\rn}{\vol(B_1)|x|^N< \vol(\Om)}.
\end{equation}
In \cite{polya}, P\'olya gave a very elegant proof of the following result.
\begin{theorem}\label{thm polya}
For all open sets $\Om\subset\rn$ of finite volume, the following holds
\[
E(\emptyset,\Om)\le E(\emptyset,\Om^*)
\]
\end{theorem}
Actually the original proof by P\'olya employed the use of the following equivalent definition of the (one-phase) torsional rigidity of an open set $\Om$:
\begin{equation}\label{T}
T(\Om):=\max_{v\in \hoi\Ome\sm\{0\}}\frac{\pp{\int_\Om |v|}^2}{\int_\Om |\gr v|^2}.
\end{equation}
In order to prove the equivalence between the functionals $E(\emptyset,\cdottone)$ and $T(\cdottone)$, we will follow \cite{brasco} and introduce a third functional that will serve as a bridge between the two:
\begin{equation}\label{F}
\cF_\Om(v):= 2\int_\Om v-\int_\Om |\gr v|^2 \quad \tfor v\in \hoi\Ome.
\end{equation}
We will also need the following simple lemma. It follows immediately from Young's inequality for products and therefore the proof will be omitted.
\begin{lemma}\label{young brasco}
Let $A,B>0$, then we have
\begin{equation}\label{yb}
A t- B \frac{t^2}{2} \le \frac{A^2}{2B} \quad\tfor t\ge0,
\end{equation}
and equality in \eqref{yb} holds only for $t=A/B$.
\end{lemma}
\begin{lemma}\label{kekkyokuissho}
Let $\Om\subset \rn$ be an open set of finite volume. Then 
\[
E(\emptyset,\Om)=\max_{v\in\hoi\Ome} \cF_\Om(v)= T(\Om).
\]
\end{lemma}
\begin{proof}
First of all, let us prove that $E(\emptyset,\cdottone)=\max_{v\in \hoi\Ome}\cF_\Om(v)$. 
Since $\cF_\Om$ is a strictly concave functional, it has a unique maximizer, say $v_M\in \hoi\Ome$. Moreover, by computing its G\^ateaux derivative, we get
\[
0=\lim_{t\to0} \frac{\cF_\Om(v_M+t\psi)-\cF_\Om(v_M)}{t}= 2\int_\Om \psi -2\int_\Om \gr v_M\cdot \gr \psi \quad\tforall \psi\in\hoi\Ome.
\]
In other words, $v_M$ is a weak solution of \eqref{pb weak} for $D=\emptyset$ and $\sg_s=1$. This implies that 
\[
E(\emptyset,\Om)=\max_{v\in\hoi\Ome} \cF_\Om(v).
\]
Let now $v_0$ be a maximizer in \eqref{T} (which, without loss of generality, we will suppose non negative). Set 
\[
\la_0:=\frac{\int_\Om v_0}{\int_\Om |\gr v_0|^2}.
\]
It is not difficult to show that $w_0:=\la_0 v_0\in\hoi\Ome$ is a maximizer for $\cF_\Om$. Indeed, by Lemma \ref{young brasco} and definitions \eqref{T} and \eqref{F} we can write, for all $v\in\hoi\Ome$,
\[
\cF_\Om(v)=2\int_\Om v-\int_\Om |\gr v|^2\le 2 \max_{\la\ge0} \left\{ \la\int_\Om |v|-\frac{\la^2}{2}\int_\Om |\gr v|^2 \right\}=\frac{\pp{\int_\Om |v|}^2}{\int_\Om |\gr v|^2}\le T(\Om).
\]
Notice that equality holds in the chain of inequalities above if $v=\la_0 v_0$. In particular, 
$\cF_\Om(v)\le T(\Om)= \cF_\Om(\la_0 v_0)$ for all $v\in\hoi\Ome$ and therefore
\[
E(\emptyset,\Om)= \max_{v\in\hoi\Ome} \cF_\Om(v) = \cF_\Om(\la_0 v_0) = T(\Om),
\] 
which concludes the proof.
\end{proof}

The key to P\'olya's proof lies in spherical rearrangements of measurable functions and the related inequalities. 
Let $f$ be a nonnegative measurable function vanishing at infinity, in the sense that all its positive superlevel sets $\{f>t\}$ with $t>0$ have finite measure. We define its spherical decreasing rearrangement $f^*$ as the measurable function whose superlevel lets are the $*$-symmetrization of those of $f$ (see \eqref{symmetrization}):
\begin{equation}\label{levelsets}
\{f^*>t\}:=\{f>t\}^*\quad \tforall t>0.
\end{equation}
\vspace*{-0.5cm}
\begin{figure}[h]
\centering
\includegraphics[width=0.9\linewidth,center]{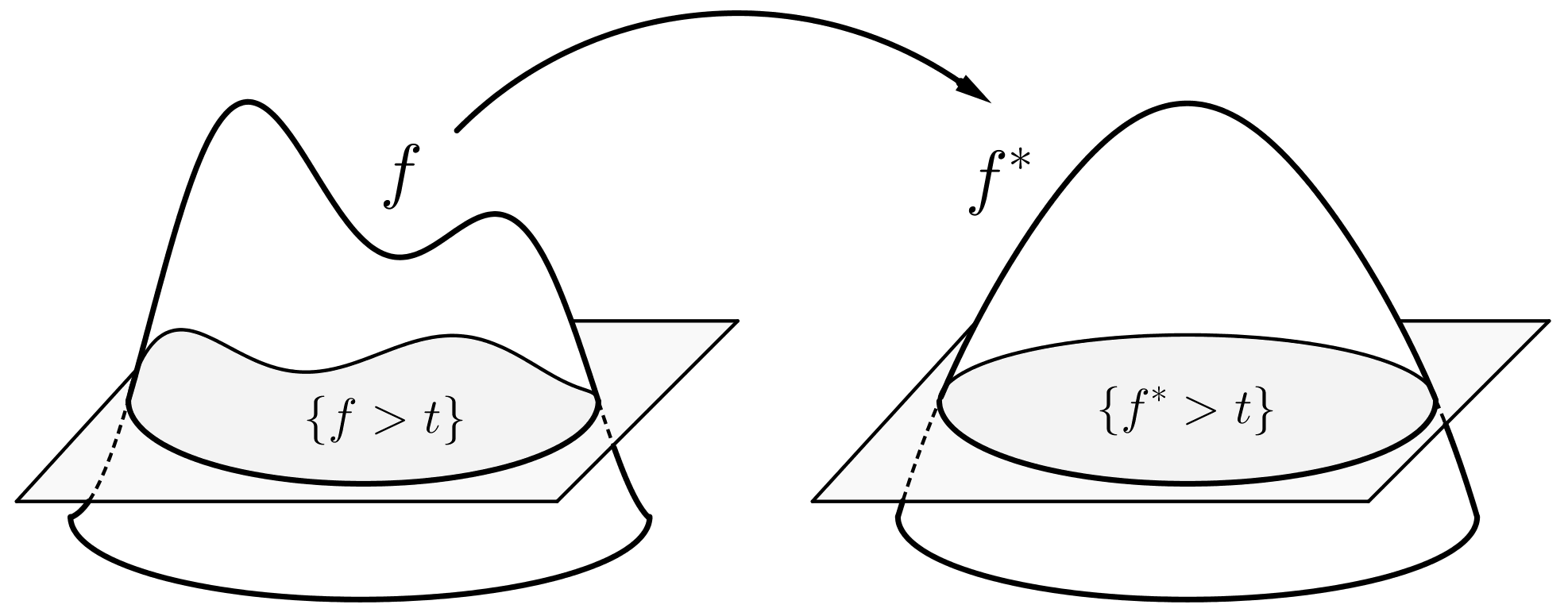}
\caption{Spherical decreasing rearrangement.
} 
\label{rearr}
\end{figure}

Such function $f^*$ is uniquely determined by the measure of the superlevel sets of $f$ and admits the following ``layer cake'' decomposition:
\[
f^*=\int_0^\ali \chi_{\{{f>t}\}^*} \,dt.
\]
By \eqref{levelsets} and Cavalieri's principle, it follows that $f$ and $f^*$ are equimeasurable, i.e. for every measurable function $g:[0,\ali)\to\RR$ the following holds
\[
\int_\rn g\circ f = \int_\rn g\circ f^*.
\] 
In particular, this implies that $L^p$-norms are preserved after spherical rearrangements, in the sense that, if $f\in L^p(\rn)$, $1\le p\le \ali$, is a nonnegative function vanishing at infinity, then 
\[
\norm{f}_p=\norm{f^*}_p.
\]
On the other hand, the $L^p$-norm of the gradient is not preserved by spherical rearrangements, as the following result shows. 
\begin{thm}[P\'olya-Szeg\H{o} inequality]\label{polyaszego}
Let $f\in W^{1,p}(\rn)$, $1\le p \le \ali$, be a nonnegative measurable function vanishing at infinity, then
\[
\norm{\gr f}_p \ge \norm{\gr f^*}_p.
\] 
\end{thm}
\begin{proof}[Proof of Theorem {\rm\ref{thm polya}}]
Once all the ingredients are ready, the proof just takes one line. Let $\Om$ be a measurable set of finite measure and set $v_0\in\hoi\Ome$ to be the maximizer in the definition of $T(\Om)$. We have
\[
T(\Om)=\frac{\pp{\int_\Om v_0}^2}{\int_\Om |\gr v_0|^2}\le \frac{\pp{\int_{\Om^*}v_0^*}^2}{\int_{\Om^*}|\gr v_0^*|^2}\le T(\Om^*),
\] 
where we used equimeasurability and Theorem \ref{polyaszego} in the first inequality.
\end{proof}
\section{Serrin's overdetermined problem}
\label{sec 1-ph serrin}
In this section we will deal with the original one-phase Serrin's overdetermined problem. We are looking for a bounded domain $\Om\subset \rn$ of class $\C^2$ such that the following overdetermined boundary value problem admits a solution for some constant $d$:
\begin{equation}\label{serrin op} 
\left\{
\begin{aligned}
-\De u&=1&\quad \tin \Om,\\
u&=0& \quad \ton \pa\Om,\\
\dn u&=-d & \quad\ton\pa\Om.
\end{aligned}
\right.
\end{equation}
In \cite{serrin}, Serrin proved the following theorem, characterizing the solutions to \eqref{serrin op}.
\begin{theorem}\label{serrin thm}
Let $\Om\subset\rn$ be a bounded domain of class $\C^2$. If the overdetermined problem \eqref{serrin op} admits a solution for some constant $d>0$, then $\Om$ is an open ball of radius $R=Nd$. 
\end{theorem}
If \eqref{serrin op} has a solution, then, $d$ must be positive by the Hopf lemma. Moreover, by the divergence theorem
\[
d=\frac{\vol(\Om)}{\per(\Om)},
\]
and thus if $\Om=B_R$, then $d=R/N$. On the other hand, the fact balls are the only domains that allow for a solution to problem \eqref{serrin op} in not obvious at all.
This has led many mathematicians to devise their own proofs: each of them shedding light on the problem from a different angle. In what follows, we will present the original proof by Serrin, nevertheless, the interested reader is encouraged to read the survey papers \cite{magna review} and \cite{nittr}.
Serrin's proof heavily relies on the invariance with respect to rigid motions of the Laplace operator and on the maximum principle. In particular, both the classical Hopf lemma and the following refined version for domains with corners (see \cite{serrin} for a proof) play a fundamental role in the proof of Theorem \ref{serrin thm}.
\setcounter{thm}{2}
\begin{lem}[Serrin's corner lemma, \cite{serrin}]\label{scl}
Let $\Om\subset\rn$ be a bounded domain of class $\C^2$. Fix a point $P\in\pa\Om$ and let $\theta$ be a direction orthogonal to $n(P)$ (the outward unit normal to $\pa\Om$ at $P$). Moreover, let $H_\theta$ be an open half-plane that is orthogonal to $\theta$ and such that $\Om\cap H_\theta\ne\emptyset$. Let $w\in \C^2(\ol \Om\cap \ol H_\theta)$ satisfy
\[
-\De w\ge 0 \quad \tand w\ge0\quad\tin \Om\cap H_\theta.
\]
If $w(P)=0$, then, for all directions $\ell$ in $P$ entering $\Om\cap H_\theta$, i.e. such that $\ell\cdot n<0$ at $P$, then 
\[
\text{ either } \;\pa_\ell w(P)>0\quad \tor \quad\pa_{\ell \ell } w(P)>0,
\]
unless $w\equiv 0$ in $\Om\cap H_\theta$.
\end{lem} 
\begin{proof}[Proof of Theorem {\rm \ref{serrin thm}}]
Serrin's proof is based on the following idea: a domain $\Om$ is a ball if and only if it is mirror-symmetric with respect to any fixed direction $\theta$. Suppose by contradiction that $\Om$ is not mirror-symmetric in the direction $\theta$ (which, up to a rotation, can be assumed to be the upward vertical direction). Take now a hyperplane $\pi$ perpendicular to $\theta$ that does not intersect $\Om$ (this can be done because $\Om$ is bounded). Now, move $\pi$ along the direction $\theta$ until it intersects $\Om$. Let $\cS$ denote the portion of $\Om$ that lies below the hyperplane and $\cS'$ its mirror-symmetric image with respect to it. If $\cS'\subset\Om$, then we can continue moving the hyperplane upwards. This motion will eventually stop, namely when (at least) one of the two following cases occur (see Figure \ref{mp}):
\begin{enumerate}[label=(\roman*)]
\item $\cS'$ becomes internally tangent to $\pa\Om$ at some point $P\in \pa\Om\sm\pi$ 
\item the hyperplane $\pi$ is orthogonal to $\pa\Om$ at some point $P\in\pa\Om\cap \pi$. 
\end{enumerate}
\begin{figure}[h]
\centering
\includegraphics[width=0.7\linewidth,center]{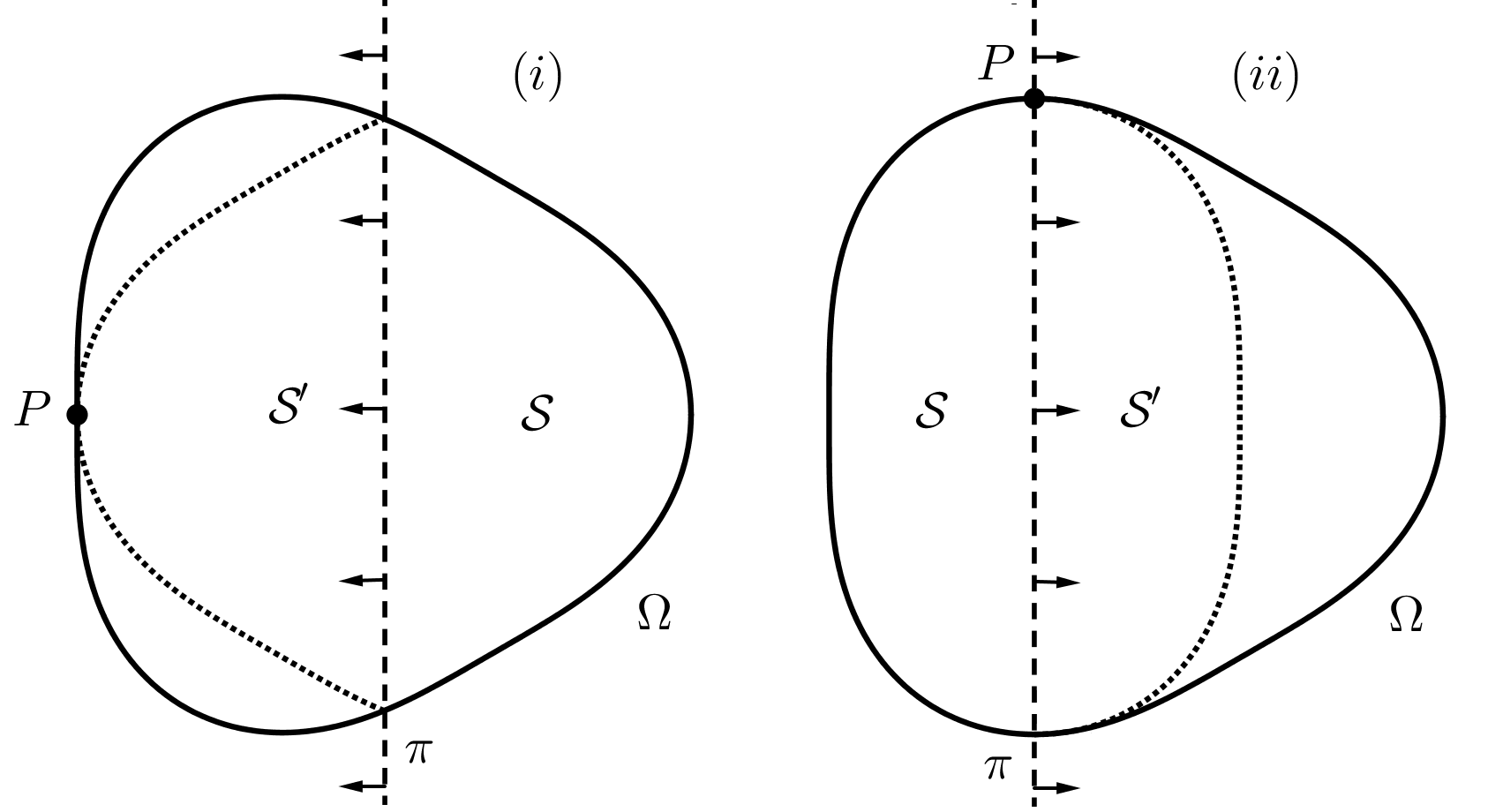}
\caption{The method of moving planes: case (i) on the left, case (ii) on the right.
} 
\label{mp}
\end{figure}
For all $x\in \rn$, let $x'$ denote the reflection of $x$ across the hyperplane $\pi$. We define the following auxiliary function on $\cS'$:
\[
u'(x) := u(x') \quad \tfor x\in \cS'.
\]
Consider now the function $w:=u-u'$ in $\cS'$. It is easy to see that $w$ verifies 
\[
\De w=0 \,\tin \cS',\quad w=0 \,\ton \pa \cS'\cap\pi,\quad w\ge 0\,\ton \pa\cS'\sm\pi,
\]
where we applied the maximum principle to $u$ in order to obtain the last inequality. 
A further application of the maximum principle yields either $w>0$ in $\cS'$ or $w\equiv 0$ in $\cS'$. The latter is excluded because we are supposing by contradiction that $\Om$ is not symmetric with respect to $\theta$.
Assume now that case (i) occurs, that is $\cS'$ is internally tangent to $\pa\Om$ at some point $P$ that does not belong to $\pi$. Then, by the Hopf lemma,
\[
\dn w (P)<0 ,
\]
but by construction 
\begin{equation}\label{gr w=0}
\dn w(P)= \dn u(P)-\dn u'(P) = -d+d=0.
\end{equation} 
In other words case (i) cannot occur if $\Om$ is not symmetric with respect to $\pi$.
Suppose now that case (ii) happens. In this case the Hopf lemma is not enough and we will resort to Lemma \ref{scl}. We are going to prove that, under these circumstances, the point $P$ is a second order zero for $w$, i.e. $w$ and all its first and second order derivatives computed at $P$ vanish. If this is the case, then by Lemma \ref{scl}, $w\equiv 0$ in $\cS'$, which is a contradiction. We will now show that $P$ is a second order zero for the function $w$. To this end, let us consider a coordinate system with the origin at $P$, the $x_N$ axis pointing in the direction $\theta$ and the $x_1$ axis in the direction of $n$. Locally there exists a $\C^2$-function $f:\RR^{N-1}\to \RR$ such that a portion of $\pa\Om$ is given by 
\[
\pp{f(x_2,\dots,x_N),x_2,\dots, x_N}\quad \tfor \norm{(x_2,\dots,x_N)} \text{ small}.
\]
Therefore, $u\equiv 0$ on $\pa\Om$ can be locally rewritten as 
\begin{equation}\label{u=0}
u(f,x_2,\dots,x_N)=0,
\end{equation}
and since the outward normal $n$ admits the local expression
\[
n(f,x_2,\dots,n_N)=\frac{\pp{1,-{\pa_{x_2}} f,\dots, -\pa_{x_N} f}}{\sqrt{1+\sum_{i=2}^N (\pa_{x_i} f)^2}},
\]
the overdetermined condition $\dn u=-d$ on $\pa\Om$ is locally expressed by 
\begin{equation}\label{dnuc}
\pa_{x_1} u- \sum_{i=2}^N \pa_{x_i} u \pa_{x_i} f = -d \bigg\{1+\sum_{i=2}^N \pp{\pa_{x_i} f}^2\bigg\}^{1/2}. 
\end{equation}
Differentiating \eqref{u=0} with respect to $x_i$, $i=2,\dots, N$, yields 
\begin{equation}\label{a1.6}
\pa_{x_1}u\pa_{x_i}f+\pa_{x_i} u =0.
\end{equation}
Evaluate now \eqref{a1.6} and \eqref{dnuc} at $P$. Since $\pa_{x_i}f(0)=0$ for $i=2,\dots,N$, we have
\begin{equation}\label{a.17}
\pa_{x_1}u(P)=-d,\quad \pa_{x_i} u(P)=0 \,\tfor i=2,\dots,N.
\end{equation}
This means that $\dn u=-d$ and $\grt u=0$ at $P$, in other words, all first derivatives of $u$ and $u'$ coincide at $P$, hence $\gr w(P)=0$. In order to show that also $D^2 w(P)=0$, notice that, in the new coordinates
\[
u'(x_1,\dots, x_{N-1},x_N)=u(x_1,\dots, x_{N-1},-x_N).
\]
In particular, by construction
\begin{equation}
\pa_{x_N x_N} u(P)=(-1)^2\pa_{x_N x_N} u'(P) \; \tand\; \pa_{x_ix_j} u(P) = \pa_{x_ix_j} u'(P), \, i,j=1,\dots,N-1. 
\end{equation}
We are now left to show that all mixed derivatives with respect to $x_i$ and $x_N$ ($i=1,\dots,N-1$) of $u$ and $u'$ coincide as well.
Differentiate \eqref{a1.6} with respect to $x_N$
\begin{equation}\label{a1.8}
{\pa _{x_i x_N}}u(P)=d\,\pa_{x_i x_N} f(0)=0\quad i=2,\dots,N-1,
\end{equation}
In the second equality above we used the assumption that the reflected cap $\cS'$ lies inside $\Om$ and, therefore, $\pa_{x_i x_N} f(0)=0$ for $i=2,\dots,N-1$. 
We now need to compute $\pa_{x_1 x_N} u$ at $P$. To this end, differentiate \eqref{dnuc} with respect to $x_1$ and use \eqref{gr w=0}. We get
\begin{equation}\label{a1.9}
\pa_{x_1 x_N} u(P)=0
\end{equation} 
as claimed.
We have proved that all the second order derivatives of $u$ and $u'$ coincide at $P$. As remarked before, this contradicts the assumption that $\Om$ is not mirror symmetric with respect to the hyperplane $\pi$, concluding the proof.
\end{proof}

\chapter{Shape derivatives}
\label{ch shape derivatives}

In this chapter we are going to introduce the concept of \emph{shape derivatives} and some of the basic techniques in order to compute them. The contents of this chapter are well known classical results: we will follow \cite{henrot} and \cite{SG} in our exposition. 

It is not unusual to encounter functions that depend on the ``shape'' of a domain $\om$: the volume of $\om$, its surface area, its barycenter or even the solution $u_\om$ of some boundary value problem on $\om$ etc... they are all \emph{shape functionals} and the machinery in this chapter will apply to them all. In what follows, we will study how to deduce optimality conditions for shape functionals. As one knows, in order to find the extremal points of a function $f:\rn\to\RR$, one could resort to studying the points where its gradient $\gr f$ vanishes. When the input variable of $f$ is not a point in the Euclidean space but a ``shape'' (for example an open set), then the above operation will lead to some overdetermined free boundary problem (we will discuss how this relates to the examples in Chapter \ref{ch torsion+serrin} in Section \ref{sec 1-ph equivalence}). Nevertheless, it is not clear at first glance how the concept of derivative could be extended to shape functionals. We will give two (equivalent) formulations of this in Section \ref{sec preliminaries}. The actual computation techniques will be discussed in Section \ref{sec hadam form}, where integral functionals (both on variable domains and on variable boundaries) will be of particular importance. Finally, in Section \ref{sec state functions} it will be discussed how to compute shape derivatives of functionals that take values in a Banach space, in particular, we will be interested in how to compute the shape derivative $u'$ of a functional of the form $\om\mapsto u_\om$, where $u_\om$ solves some boundary value problem on $\om$. We will show how to characterize $u'$, in turn, as a solution of a boundary value problem.

\section{Preliminaries to shape derivatives}
\label{sec preliminaries}

The classical notion of differentiability can be defined in the framework of normed vector spaces. Nevertheless, this is not enough for our purposes, as the set of ``shapes'' is not endowed with any obvious linear structure. In order to overcome this problem, one could opt for the following ``Fr\'echet-derivative'' approach. 
Let $J:\cO\to X$ be a shape functional, where $\cO$ is a family of subsets of $\rn$ and $X$ is a Banach space. One can then consider the application
\begin{equation}\label{frechet app}
\phi\mapsto\cJ(\phi):=J\big((\id+\phi)(\om)\big),\quad \text{ for some fixed }\om\in\cO,
\end{equation}
where $\phi$ ranges in a neighborhood of $0$ of some Banach space $\Theta$ of mappings from $\rn$ to itself. Of course, one should be careful about the choice of $\Theta$, and require that $(\id+\phi)(\om)\in\cO$ at least for $\phi\in\Theta$ small enough.

One could now examine the \emph{Fr\'echet differentiability} of the map $\cJ:\Theta\to X$ in a neighborhood of $0$.
We recall the definition of Fr\'echet differentiability.
Let $V$ and $W$ be Banach spaces (whose norms will be indistinctly denoted by $\norm{\cdottone}$) and let $U\subset V$ be an open subset of $V$. A function $f:U\to W$ is then said to be Fr\'echet differentiable at $x_0\in U$ if there exists a bounded linear operator $A:V\to W$ such that
\[
\lim_{x\to 0} \frac{\norm{f(x_0+x)-f(x_0)-Ax}}{\norm{x}}=0.
\]
It is easy to show that, when such an operator $A$ exists, then it is also unique. Therefore this bounded linear operator will be denoted by $f'(x_0)$ and referred to as the Fr\'echet derivative of $f$ at $x_0$ (the term ``differential" is also commonly used in this case). Moreover we will say that $f:U\to V$ is of class $\C^1$ in $U$ if $f':U\to L(V,W)$ is a continuous map from $U$ to $L(V,W)$, the space of bounded linear operators from $V$ to $W$. Analogously, if $f'$ happens to be Fr\'echet differentiable, say in $U$, then the map 
\[
f'':=(f')': U\to L(V,L(V,W))
\] 
is called the second derivative of $f$. To make it easier to work with, the space $L(V,L(V,W))$ is usually identified with the Banach space 
of all continuous bilinear maps from $V$ to $W$. We remark that Fr\'echet derivatives of higher order can be defined recursively in the natural way, although for our purposes it will be enough to work with derivatives up to the second order. 
This ``Fr\'echet-derivative'' approach will be very useful to prove theoretical results, such as regularity properties of shape functionals (see for instance Theorem \ref{thm diff state func 1}) and the structure theorem (Theorem \ref{struct thm} on page \pageref{struct thm}).
However, once the above-mentioned results are known, it is easier in practice to compute shape derivatives by means of a differentiation along a ``flow of transformations'' parametrized by a real variable $t$ as follows. 
As before, let $J:\cO\to X$ denote a shape functional. Consider the following flow of transformations $\Phi:[0,1)\to\Theta$, where the map $t\mapsto \Phi(t)$ is differentiable at $0$ and $\Phi(t) = t h+o(t)$ as $t\to 0$ for some $h\in\Theta$. We can now consider the derivative of the following map 
\begin{equation}\label{flow app}
t\mapsto j(t):= J(\om_t):=J\big((\id+\Phi)\om\big), \quad \text{for some fixed }\om\in\cO.
\end{equation}
We will write 
\[
j'(0)= J'(\om)(\Phi).
\]
Notice that the two approaches \eqref{frechet app} and \eqref{flow app} are equivalent in the following sense: when $\cJ$ is Fr\'echet differentiable, then 
\[
\cJ'(0)h=J'(\om)(\Phi),\quad \tif\; \Phi(t)=th+o(t) \text{ as }t\to0. 
\]



\section{Shape derivatives of integral functionals}
\label{sec hadam form}
That of integral functionals is quite vast subclass of shape functionals. In this subsection we will learn the basic formulas for computing the shape derivatives of functionals of the form 
\[
\om\mapto \int_\om f_\om \quad\tand\quad \om\mapto \int_{\pa\om} g_\om.
\]

\begin{proposition}[Hadamard formula]\label{hadam form}
Let $\Phi:[0,1)\to W^{1,\ali}\rnrn$, differentiable at $0$, with $\Phi(0)=0$ and $\pato \Phi=h$. Suppose that the map $[0,1)\ni t\mapsto f(t)\in L^1(\rn)$ is differentiable at $0$ with derivative $f'(0)$ and that $f(0)\in W^{1,1}(\rn)$.
If $\om$ is a bounded Lipschitz domain, then, the map $t\mapsto i(t)=\int_{\om_t}f(t)$ is differentiable at $t=0$ and we have
\begin{equation}\label{result hadam}
i'(0)=\int_\om f'(0)+\int_{\pa\om}f(0) h\cdot n.
\end{equation}
\end{proposition}
Formula \eqref{result hadam} is without doubts the natural result that one would expect. As a matter of fact, one can {\bf formally} verify it as follows:
By change of variables we have $i(t)=\int_{\om_t} f(t)= \int_{\om} f(t)\circ\pp{\id+\Phi(t)}J(t)$, where $J(t)=\det(I+D\Phi(t))$ is the Jacobian associated to the transformation $x\mapsto x+\Phi(t,x)$. Differentiation, followed by some easy manipulation and the application of the divergence theorem yield:
\begin{equation}\nonumber
\begin{aligned}
i'(0)=\int_\om f'(0)+\gr f(0)\cdot h + f(0) \,\dv h\\
= \int_{\om} f'(0) + \int_\om \dv\pp{f(0)h} = \int_\om f'(0) + \int_{\pa\om} f(0)h\cdot n.
\end{aligned}
\end{equation}
The rigorous proof of \ref{result hadam}, under the weak regularity assumptions of Proposition \ref{hadam form}, turns out to be quite delicate. We choose to postpone it, in order to first illustrate some applications. 
\begin{corollary}\label{generalized hadam form}
Let $\Phi\in \C^1\pp{[0,1),W^{1,\ali}\rnrn}$ and $f=f(t,x)\in\C^1\left([0,T),L^1(\rn)\right)\cap \C\left([0,T),W^{1,1}(\rn)\right)$.
Assume that $\om$ is a bounded open set with Lipschitz continuous boundary, then the function $[0,T)\ni t\mapsto i(t):=\int_{\om_t} f(t)$ is continuously differentiable on $[0,T)$ and we have
$$
i'(t_0)=\int_{\om_{t_0}}\pa_t f(t_0)+\int_{\pa \om_{t_0}} f(t_0) V(t_0)\cdot n_{t_0} \quad \tforall t_0\in[0,1),
$$
where $V(t,x):=\pa_t \Phi(t,(\id+\Phi(t))^{-1}(x))$ and $n_{t}$ is the outward unit normal to $\pa\om_{t}$.
\end{corollary}
\begin{proof}
One applies Proposition \ref{hadam form} to the following auxiliary function $\ol{f}$ and perturbation $\ol{\Phi}$:
\begin{equation}\label{proof generalized hadam form}
\ol{f}(t,x):=f(t+t_0,x),\quad \id+\ol{\Phi}(t):=\pp{\id+\Phi(t+t_0)}\circ\pp{\id+\Phi(t_0)}\inv.
\end{equation}
\end{proof}
\begin{remark}\emph{
When computing the second order derivative of integral functionals, we will need to know the expression of the first derivatives in a right neighborhood of $t=0$ and thus we cannot directly employ the use of the Hadamard formula, as stated in Proposition \ref{hadam form}. This is where Corollary \ref{generalized hadam form} comes in handy. 
}\end{remark}
When computing the shape derivative of a surface integral functional, usually the mean curvature comes out in the process. We give here an alternative definition of the (additive) mean curvature that is most natural in the framework of shape derivatives. Let $\om$ be a domain of class $\C^2$ and $n$ denote its outward unit normal. We set
\[
H:=\dv_\tau n,
\] 
where $\dv_\tau$ is the tangential divergence (defined in \eqref{tg div} in Appendix A).
Notice for example, that the (additive) mean curvature $H$ of a sphere $\pa B_R$ is positive and equals $(N-1)/R$ (it corresponds to the sum of the principal curvatures, computed with respect to the {\bf inward} normal $-n$).
The following result is an analogue of the Hadamard formula for surface integrals. Later, we will give a refined version, that relies on weaker regularity assumptions, Proposition \ref{hadam form 2}. 
\begin{corollary}[A first Hadamard formula for surface integrals]\label{surface hadam form}
Let $\Phi:[0,1)\mapto \C^{2,\ali}\rnrn$, differentiable at $t=0$, with $\Phi(0)=0$ and $\pato \Phi=h$. Suppose that $\om$ is a bounded domain of class $\C^3$. Consider a function $t\mapsto g(t)\in W^{1,1}(\rn)$ that is differentiable in a neighborhood of $0$ with derivative $g'(0)$ and such that $g(0)\in W^{2,1}(\rn)$. Then the map $t\mapsto j(t)=\int_{\pa \om_t} g(t)$ is differentiable at $0$ and we have
\[
j'(0)= \int_{\pa\om}g'(0) + \pp{\dn g(0)+H g(0)}h\cdot n. 
\]
\end{corollary}
\begin{proof}
Let $n$ (respectively $n_t$) denote an extension of the outward unit normal to $\pa \om$ (respectively $\pa\om_t$) of class $\C^2$ (respectively $\C^1$) on $\rn$. To fix ideas, we might put $n:=\gr d_{\om}$ in a neighborhood of $\pa\om$ and maybe multiply it by a smooth cut off function to be sure that the extension is smooth even far away from the boundary $\pa\om$, where $d_{\om}$ is the signed distance function to $\pa \om$ (see also \cite[Chapter 5]{SG}), defined as
\begin{equation}\label{dist func}
d_\om(x):=\begin{cases}
-\dist (x,\pa\om) \quad\tfor x\in\om,\\
\dist(x,\pa\om)\quad\tfor x\in \rn\sm\om.
\end{cases}
\end{equation}
Of course, the same can be done for $n_t$. 
Now, we just need to apply Proposition \ref{hadam form} to $j(t)=\int_{\om_t} \dv\pp{g(t)n_t}$. By hypothesis we have that $\dv\,n\in\C^1(\rn)$ and $\gr g(0)\in W^{1,1}\rnrn$ and thus $\dv\pp{g(0)n}\in W^{1,1}(\rn)$. Therefore, we just need to check that the map
\[
t\mapsto \dv\pp{g(t)n_t}=g(t)\dv(n_t)+\gr g(t)\cdot n_t \in L^1(\rn)
\]
is differentiable at $t=0$. By construction, $n_t$ is differentiable at $t=0$ (see also Proposition \ref{der of normal}) and so is the map $t\mapsto g(t)\in W^{1,1}(\rn)$ by hypothesis. Now, an application of Proposition \ref{hadam form} yields
\[
j'(0)=\int_{\pa\om} g'(0)+ \int_{\pa\om} g(0) (\pato n_t)\cdot n + \int_{\pa\om} \pp{\gr g(0)\cdot n+ g(0) \dv(n)}\,h\cdot n.
\] 
Since we chose $n_t=\gr d_{\om_t}$, then for $t\ge0$ small, $n_t$ is unitary in a neighborhood of $\pa\om$ and hence $\dato n_t$ is orthogonal to $n$. We conclude by recalling that in this case, $\dv(n)=\dv_\tau(n)= H$. 
\end{proof}

The following lemma is a key ingredient in the proof of Proposition \ref{hadam form}.

\begin{lemma}\label{5.2.6}
Let $g\in W^{1,1}(\rn)$ and $\Psi:[0,1)\to W^{1,\ali}\rnrn$ be continuous at $t=0$, $t\mapsto \Psi(t)\in L^\ali$ differentiable at $0$, with derivative $Z$. Then the map 
\[
t\mapsto G(t):=g\circ \Psi(t)\in L^1(\rn)
\]
is differentiable at $0$ and $G'(0)=\gr g\cdot Z$.
\end{lemma} 
\begin{proof}
First of all, we claim that, for every $f\in L^1(\rn)$
\begin{equation}\label{composition}
\lim_{t\to0} f\circ \Psi(t)=f \tin L^1(\rn).
\end{equation} 
We will prove it by an approximation argument. Fix $f\in L^1(\rn)$ and let $\{f_k\}_{k\in\NN}$ be a sequence of functions in $\C_0^\ali(\rn)$ converging to $f$ in $L^1(\rn)$. We get 
\begin{equation}\label{madafnohanashidattakoro}
\begin{aligned}
\norm{f\circ \Psi(t)-f}_1\le \norm{f\circ \Psi(t)-f_k\circ \Psi(t)}_1 + \norm{f_k\circ \Psi(t)-f_k}_1+\norm{f_k-f}_1\\
\le C \norm{f_k-f}_1+\norm{f_k\circ \Psi(t)-f_k}_1,
\end{aligned}
\end{equation}
where in the last inequality we used the fact that the Jacobian of $\Psi(t)$ is uniformly bounded. Since $f_k\in\C_0^\ali(\rn)$, then the last term in the above tends to $0$ for all $k\in\NN$. As a matter of fact, since $f_k\in\C_0^\ali(\rn)$, for some ball $B(f_k)$, whose radius (depending on the support of $f_k$ and on a uniform constant bounding the $L^\ali$-norm of $\Psi(t)$) is large enough,
\[
\norm{f_k\circ\Psi(t)-f_k}_1 = \int_ {B(f_k)} |f_k\circ \Psi - f_k|= \int_{B(f_k)} \big|\gr f_k(x)\cdot (\Psi(t,x)-x)+\eps_1(t,x)\big|\, dx,
\]
where $\norm{\eps_1(t,\cdot)}_\ali\to0$ as $t\to 0$. 
Therefore, 
\begin{equation}\label{himitsuheiki}
\norm{f_k\circ \Psi(t)-f_k}_1\le \vol(B(f_k)) \Big\{\norm{f_k}_{1,\ali}\norm{\Psi(t)-\id}_\ali +\norm{\eps_1(t)}_\ali\Big\}.
\end{equation}
We conclude by taking the limits with respect to $t\to0$ and then $k\to\ali$ in \eqref{madafnohanashidattakoro}.

Suppose now that $g\in\C_0^\ali$. For $y\in\rn$ we have
\[
g(x+y)-g(x)-\gr g(x)\cdot y = \int_0^1 \left\{ \gr g(x+sy)-\gr g(x)\right\}\cdot y\,ds.
\] 
We employ the use of the formula above with $y=\Psi(t,x)-x= t Z(x)+t \eps_2(t,x)$, where $\norm{\eps_2(t,\cdot)}_\ali\to0$ as $t\to0$, and integrate it with respect to $x$ on the whole $\rn$.
We put 
\[
\eta_t:= t\inv \norm{g \circ\Psi(t)-g-t g\cdot Z}_1.
\]
The following estimate holds:
\begin{equation}\label{etat}
\eta_t\le \norm{\gr g}_1\norm{\eps_2(t)}_\ali+ C\norm{e(t,g)}_1,
\end{equation}
where $C$ is a uniform majorant of $\norm{Z+\eps_2(t)}_\ali$ and 
\[
e(t,g)(x):=\int_0^1 \abs{\gr g\big((1-s)x+s\Psi(t,x)\big)-\gr g(x)} \, ds.
\]
By the change of variable $z=(1-s)x+s \Psi(t,x)$ we get the estimate
\[
\norm{e(t,g)}_1\le 2 \norm{\gr g}_{\ali}\norm{\Psi(t)}_{1,\ali}.
\]
Now suppose that $g\in W^{1,1}(\rn)$ and $\{g_k\}_{k\in\NN}$ is a sequence of functions in $\C_0^\ali(\rn)$ converging to $g$ in $W^{1,1}(\rn)$. Inequality \eqref{etat}, that holds for $g_k$, is actually true for $g$ too, by \eqref{composition}. Now, combining the previous estimates and $e(t,g)\le e(t,g-g_k)+e(t,g_k)$, we obtain
\[
\norm{e(t,g)}_1\le 2 \norm{g-g_k}_{1,1}\norm{\Psi(t)}_{1,\ali}+ \vol(B(g_k))\norm{g_k}_{2,\ali}\norm{\Psi(t)-\id}_\ali, 
\] 
where the last term is derived as \eqref{himitsuheiki}. Taking the limits for $t\to0$ and then $k\to\ali$ yields $\eta_t\to 0$, that is the conclusion of the lemma.
\end{proof}

\begin{proof}[Proof of Proposition {\rm\ref{hadam form}}]
By assumption we have 
\[
D\pp{\id+\Phi(t)}=I+ t Dh + t\, \eps_1(t) \quad \text{ almost everywhere in }\rn, 
\]
where $\eps_1(t)=\eps_1(t,\cdottone)\in L^\ali\rnrn$ and $\norm{\eps_1(t)}_\ali\to 0$ as $t\to0$. Now, recall that the map $A\mapsto \det A\in L^\ali(\rn)$ is differentiable in $L^\ali(\rn,\RR^{N\times N})$, and its derivative at the identity matrix $I$ is given by the trace function. Thus the following holds almost everywhere in $\rn$:
\begin{equation}\label{J exp}
J(t)=\det\pp{\id+\Phi(t)}= 1+t\,\dv h +t\, \eps_2(t),
\end{equation}
where $\eps_2(t)$ also tends to $0$ in the $L^\ali$-norm.
We set now $i(t):=\int_{\om_t}f(t)$ and decompose $\{i(t)-i(0)\}/t$ into the sum of three terms:
\begin{equation}\nonumber
\begin{gathered}
A(t):=\frac{1}{t} \int_\om \big\{f(t)-f(0)\big\}\circ\pp{\id+\Phi(t)}\,J(t), \\
B(t):= \frac{1}{t}\int_\om \big\{ f(0)\circ \pp{\id+\Phi(t)}-f(0)\big\} \,J(t),\quad
C(t):=\int_\om f(0) \frac{J(t)-J(0)}{t}.
\end{gathered}
\end{equation}
By \eqref{J exp} and the dominated convergence theorem, $C(t)$ converges to $\int_\om f(0) \,\dv h$ as $t\to0$. By a further change of variable we have
\[
A(t)= \int_{\om_t} \frac{f(t)-f(0)}{t}=\int_\rn \chi_{\om_t} \left\{\frac{f(t)-f(0)}{t}-f'(0)\right\}+\int_\rn \chi_{\om_t} f'(0),
\]
which converges to $\int_\om f'(0)$. Here we used the dominated convergence theorem and the fact that $t\mapsto f(t)\in L^1(\rn)$ is differentiable by assumption. Finally, $B(t)$ converges to $\int_\om \gr f(0)\cdot h$ by Lemma \ref{5.2.6} with $g=f(0)$ and $\Psi(t)=\id+\Phi(t)$. This concludes the proof of Proposition \ref{hadam form}.
\end{proof}

\begin{proposition}\label{der of normal}
Let $\om$ be a bounded open set of class $\C^2$ and $\Phi:[0,1)\to\Phi(t)\in\C^{1,\ali}\rnrn$ be differentiable at $t=0$ with $\Phi(0)=0$ and $\pato \Phi:=h$. Moreover, let $n$ denote an extension of class $\C^{1,\ali}\rnrn$ of the unit normal to $\pa\om$
. Then, 
\[
t\mapsto n_t:= w(t)/\norm{w(t)}, \quad \text{ where }w(t):=\pp{(I+D\Phi(t))^{-T} n}\circ \pp{\id + \Phi(t)}^{-1}, 
\]
is an extension of $n$ to $\pa\om_t$ that is differentiable at $t=0$ when seen as a map $[0,1)\to \C^{0,\ali}\rnrn$. 
Moreover, for all extensions of the form $t\mapsto \wt n_t\in \C^{0,\ali}\rnrn$, differentiable at $t=0$ and such that $n_0\in\C^{1,\ali}\rnrn$, the following holds:
\[
\pato \wt n_t = -\grt(h\cdot n)-(D \wt n_0 \cdot n)\, h\cdot n \quad \ton \pa\om. 
\]
\end{proposition} 
\begin{proof}
Let $n\in\C^{1,\ali}\rnrn$ be an extension of the outward unit normal to $\pa\om$. The function $t\mapsto n_t =w(t)/\norm{w(t)}\in\C^{0,\ali}\rnrn$ is differentiable by composition. Moreover, its restriction to $\pa\om_t$ coincides with the unit normal. To show this, fix $x_0\in\pa\om$ and consider a smooth path $s\mapsto x(s)\in\pa\om$ with $x(0)=x_0$ and $x'(0)=p$: we have $p\cdot n(x_0)=0$. Now, since $x(s)+\Phi(t,x(s))\in\pa\om_t$, taking the derivative with respect to $s$ at $s=0$ yields that $q:=(I+D\Phi(t,x_0))p$ is a tangential vector to $\pa\om_t$ at the point $x_0+\Phi(t,x_0)$. It is then immediate to see that $w(t,x_0+\Phi(t,x_0))= (I+D\Phi(t,x_0))^{-T}n(x_0)$ is orthogonal to $q$.

Let us now differentiate the expression $(I+D\Phi(t))^{-T}\circ(\id+\Phi(t))\inv w(t)=n\circ (\id+\Phi(t))\inv$ with respect to $t$ to obtain 
\[
(Dh)^T n+ w'(0)= - D n h \implies w'(0)=-\gr(h\cdot n)+((Dn)^T-Dn)h.
\]
Recalling the definition of $n_t=w(t)/\norm{w(t)}$ we get
\[
\pato n_t= w'(0)- \pp{w'(0)\cdot n}n
\]
In order to carry on our computations we need to choose an extension $n$: let it be defined as $\gr d_\om$ in a neighborhood of $\pa\om$. By construction we have that $Dn=D^2 d_\om$ is symmetric and hence in this case:
\begin{equation}\label{thisoneabove}
\pato n_t= -\grt (h\cdot n). 
\end{equation}
Take now an extension $\wt n_t$ as in the statement of the proposition. As, for all $x\in\pa\om$, $(\wt{n}_t-n_t)(x+\Phi(t,x))=0$, we get
\[
\restr{\frac{\pa(\wt{n}_t-n_t)}{\pa_t}}{t=0}+D(\wt{n}_0-n)h=0.
\]
However, (see \eqref{tg jacobian}) 
\[
D_\tau (\wt{n}_0-n)=0 \implies D(\wt{n}_0-n)h=D(\wt n_0-n)n (h\cdot n) = (D\wt n_0\, n) \, h\cdot n,
\]
where in the last equality we used that $n$ is unitary in a neighborhood of $\pa\om$ and hence $Dn\, n=0$.
By recalling \eqref{thisoneabove} one gets 
\[
\pato \wt n_t = \pato n_t - D(\wt n_0-n) h = -\grt(h\cdot n)-(D\wt n_0 \,n)\, h\cdot n.
\]
\end{proof}

In order to handle surface integrals on variable domains, we will introduce the following change of variable formula. 
Let $\om$ be a bounded open set of class $\C^1$, $\phi\in\C^{1,\ali}\rnrn$ and $g\in L^1(\pa\om_\phi)$. Then $g\circ(\id+\phi)\in L^1(\pa\om)$ and the following holds
\[
\int_{\pa\om_\phi} g =\int_{\pa\om} g\circ (\id+\phi) J_\tau (\phi),
\]
where the term $J_\tau(\phi)$, defined as
\begin{equation}\label{tg Jacobian}
J_\tau(\phi)= \det(I+D\phi)\,\big\Vert{(I+D\phi)^{-T}n}\big\Vert,
\end{equation}
is called \emph{tangential Jacobian} associated to the transformation $\id+\phi$.
\begin{lemma}\label{J tau diff}
Let $\om$ be a bounded open set of class $\C^1$. The application $\phi\mapsto J_\tau (\phi)\in\C(\pa\om)$ is of class $\C^\ali$ in a neighborhood of $0\in\C^{1,\ali}\rnrn$. Moreover we have 
\[
J'_\tau(0)\phi=\dv_\tau \phi. 
\] 
Furthermore, if $t\mapsto \Phi(t)\in\C^{1,\ali}\rnrn$ is differentiable at $0$, with derivative $h$, then $t\mapsto J_\tau(\Phi(t))\in\C(\pa\om)$ is differentiable at $0$ and we have
\[
\pato J_\tau(\Phi(t))=\dv_\tau h.
\]
\end{lemma}
\begin{proof}
The application $\phi\mapsto J_\tau(\phi)= \det(I+D\phi)\,\big\Vert (I+D\phi)^{-T} n\big\Vert\in \C(\pa\om)$ is of class $\C^\ali$ by composition of applications of class $\C^\ali$.
Let us then compute the Fr\'echet derivative of $\phi\mapsto J_\tau(\phi)$ as a G\^ateaux derivative, namely, $\dato J_\tau(t\phi)$. 
We know that $\dato \det(I+tD\phi)=\dv(\phi)$. Moreover
\[
\dato \big\Vert (I+tD\phi)^{-T} n \big\Vert = \frac{n\cdot (-D\phi)^T n}{\norm{n}}= - n\cdot (D\phi)n.
\]
The first claim of the lemma follows then from definition \eqref{tg div} and the second is obvious, by composition.
\end{proof}

The last ingredient to prove Proposition \ref{hadam form 2} is the following improvement of Lemma \ref{5.2.6}.

\begin{lemma}\label{5.2.7}
Let $t\mapsto G(t)\in L^1(\rn)$ be differentiable at $t=0$ with $G(0)\in W^{1,1}(\rn)$. Then, if $t\mapsto \Phi(t)\in W^{1,\ali}(\rnrn)$ is differentiable at $t=0$ with $\Phi(0)=0$, $\pato \Phi=h$, then the function $t\mapsto g(t):=G(t)\circ \pp{\id+\Phi(t)}\inv\in L^1(\rn)$ is differentiable at $t=0$ and we have
$g'(0)=G'(0)-\gr g(0)\cdot h$.
\end{lemma}
\begin{proof}
For ease of notation, let $\psi_t$ denote $\pp{\id+\Phi(t)}\inv$. We write
$\{g(t)-g(0)\}/t=A(t)+B(t)+C(t)$, where
\[
A(t)=\left\{ \frac{G(t)-G(0)}{t}-G'(0)\right\}\circ \psi_t, \, B(t)=G'(0)\circ \psi_t,\, C(t)=\left\{ G(0)\circ\psi_t -G(0)\right\}/t.
\]
By change of variable, the $L^1$-norm of $A(t)$ can be estimated by that of $\big(G(t)-G(0)\big)/t-G'(0)$, which tends to $0$ by assumption.
The term $B(t)$ tends to $G'(0)$ because of \eqref{himitsuheiki} and finally, $C(t)$ tends to $-\gr g(0)\cdot h$ by Lemma \ref{5.2.6}.
\end{proof}

\begin{proposition}[Hadamard formula for surface integrals]\label{hadam form 2}
Let $\om$ be a bounded open set of class $\C^2$ and $t\mapsto \Phi(t)\in\C^{1,\ali}\rnrn$ be differentiable at $0$ with $\Phi(0)=0$ and $\pato\Phi=h$. Suppose that $t\mapsto g(t)\circ \pp{\id+\Phi(t)}\in W^{1,1}(\om)$ is differentiable at $0$, with $g(0)\in W^{2,1}(\om)$.
Then the map $t\mapsto \int_{\pa\om_t} g(t)$ is differentiable at $t=0$, $t\mapsto \restr{g(t)}{U}\in W^{1,1}(U)$ is differentiable at $t=0$ for all open sets $U\subset\ol{U}\subset \om$; the shape derivative $g'(0)$ is then a well defined element of $W^{1,1}(\om)$ and the following expression for the derivative of $j$ holds true:
\[
j'(0)=\int_{\pa\om}g'(0)+\pp{\dn g(0)+Hg(0)}h\cdot n.
\] 
\end{proposition}
\begin{proof}
Let $G(t):=g(t)\circ\pp{\id+\Phi(t)}$. Since, by change of variables, $j(t)=\int_{\pa\om}G(t)J_\tau(\Phi(t))$, the differentiability of $j$ comes from Lemma \ref{J tau diff}. One has 
\[j'(0)=\int_{\pa\om} G'(0)+G(0)\dv_\tau h= \int_{\pa\om} G'(0)+g(0) \dv_\tau h.\]
The differentiability of $t\mapsto \restr{g(t)}{U}\in W^{1,1}(U)$ can be shown as follows. Take a bump function $\eta\in\C^\ali(\rn)\cap\C_0^\ali(\om)$ with $\eta\equiv 1$ in a neighborhood of $\ol U$. By Lemma \ref{5.2.7} applied to $(\eta G(t))\circ \pp{\id+\Phi(t)}\inv$ and $(\eta \gr G(t))\circ \pp{\id+\Phi(t)}\inv$ we get that the map $t\mapsto \restr{g(t)}{U}\in W^{1,1}(U)$ is differentiable at $t=0$ for all open sets $U$ compactly contained in $\om$. Moreover, $g'(0)=G'(0)-\gr g(0)\cdot h\in W^{1,1}(\om)$. Therefore we may write
\[
j'(0)= \int_{\pa\om} g'(0)+\gr g(0)\cdot h +g(0) \dv_\tau h
\]
and conclude by reorganizing the integral above by means of the decomposition formula of tangential divergence \eqref{decomp tg div} and tangential Stokes theorem (Lemma \ref{tg stokes}).
\end{proof}

\section{Structure theorem and examples}
\label{sec examples}

In this section we will introduce the structure theorem for general shape functionals. Loosely speaking, it says that, under some mild regularity assumptions, shape derivatives are ``concentrated at the boundary''. In particular, first order shape derivatives can be written as a linear form that depends only on the normal component of the perturbation on the boundary. Second order derivatives are a bit more complicated, being the sum of a bilinear form and a linear one. At the end of the subsection we provide some geometrical examples. 

We will employ the use of the following notation: for $k\ge0$ integer, set 
\[
\cO_k:=\setbld{\om\subset\rn}{ \om \text{is a bounded open set of class }\C^k}, \quad \cO_k^\ell := \underbrace{\cO_k\times\dots\times \cO_k}_{\ell \text{ times }}.
\]
\begin{thm}[Structure theorem, \cite{structure}]\label{struct thm}
For integer $k,\ell\ge 1$, let $\cO\subset\cO_k^\ell$ be admissible, $X$ be a Banach space and $J:\cO\to X$ be a shape functional. Consider a fixed element $\boldsymbol{\om}\in\cO$ and define the functional $\cJ=J(\om+\cdottone): \Te_k\to X$, where $\Te_k$ is a sufficiently small neighborhood of $0\in\cC^{k,\ali}(\rn,\rn)$
. Moreover, let $\Ga:=\bigcup_{i=1}^\ell \pa \om_i$ and let $n$ denote the outward unit normal vector to each $\pa\om_i$.
\begin{enumerate}[label=(\roman*)]
\item Assume that $\boldsymbol{\om}\in\cO_{k+1}^\ell$ and that the functional $\cJ$ be differentiable at $0\in\Te_k$. Then there exists a continuous linear map $l_1:\cC^k(\Ga)\to X$ such that 
\begin{equation}
\forall \te\in\cC^k(\rn,\rn), \quad \cJ'(0)\te=l_1(\te\cdot n).
\end{equation} 
\item Assume that $\boldsymbol{\om}\in\cO_{k+2}^\ell$ and that the functional $\cJ$ be twice differentiable at $0\in \Te_k$. Then there exists a continuous bilinear symmetric map 
\[l_2:\cC^k(\Ga)\times\cC^k(\Ga)\to X \;\text{ such that}\] 
\begin{equation}\label{second structure}
\forall \te_1,\te_2\in\cC^{k+1}(\rn,\rn),\quad \cJ''(0)(\te_1,\te_2)=l_2(\te_1\cdot n,\te_2 \cdot n)+l_1(Z_{\te_1,\te_2}),
\end{equation}
where $Z_{\te_1,\te_2}=(\te_1)_\tau\cdot D_\tau n (\te_2)_\tau+n\cdot D_\tau \te_1 (\te_2)_\tau + n\cdot D_\tau \te_2 (\te_1)_\tau$.
\item Suppose that $\cJ$ is twice differentiable at $0\in\Te_k$ and that $l_1$ admits a continuous extension to $\cC^{k-1}(\Ga)\to X$. Then, if $\boldsymbol{\om}\in\cO_{k+1}^\ell$ only, \eqref{second structure} holds true for all $\te_1,\te_2\in\cC^{k}(\rn,\rn)$ instead.
\end{enumerate}
\end{thm}

\begin{corollary}\label{struc coroll}
Let $\boldsymbol{\om}$ and $J$ be as in Theorem {\rm \ref{struct thm} on page \pageref{struct thm}}, moreover let $k=1$. Define $j(t):=J\left((\id+\Phi(t))\boldsymbol{\om}\right)$ for $\Phi\in\cA$ and $t\ge 0$ small. 
\begin{enumerate}[label=(\roman*)]
\item Under the hypothesis of (i) of Theorem {\rm \ref{struct thm} on page \pageref{struct thm}}, we have
\[j'(0)=l_1(h\cdot n).\]
\item Under the hypothesis of (ii) of Theorem {\rm \ref{struct thm} on page \pageref{struct thm}}, for $\Phi$ of class $\cC^2\left([0,1),\cC^2(\rn,\rn)\right)$, then 
\begin{equation}\label{second structure t}
j''(0)= l_2(h\cdot n, h\cdot n)+l_1(Z). 
\end{equation}
Here we have set 
\[
Z:=(V'+(Dh)h)\cdot n + ((D_\tau n)h_\tau)\cdot h_\tau-2h_\tau \cdot \gr_\tau (h\cdot n),
\]
where $V(t,\id+\Phi(t)):=\pa_t \Phi(t)$ and $V'=\restr{\pa_t V(t,\cdottone)}{t=0}$. 
\item Under the hypothesis of (iii) of Theorem {\rm \ref{struct thm} on page \pageref{struct thm}}, then \eqref{second structure t} holds true for all $\Phi\in\cA$.
\end{enumerate}
\end{corollary}

\begin{remark}\emph{\label{rmk hadamard is enough}
Notice that for Hadamard perturbations (i.e. of the form $\Phi(t,x)=t h(x)$ with $h_\tau =0$ on $\Ga$), the term $Z$ appearing in {\rm \eqref{second structure t}} vanishes. As a matter of fact we have $Z=\pp{V'+(Dh)h}\cdot n$, because $h_\tau$ by assumption. Now, as $V(t,x)=h\circ \pp{\id+t h (x)}^{-1}$, we have $V'=-(Dh)h$ and hence $Z=0$ as claimed. In other words, if $\Phi$ is an Hadamard perturbation, then 
the second order shape derivative of $J$ coincides with the 
bilinear form 
$l_2^J$, that is
$J''({\boldsymbol{\om}})(\Phi)=l_2^J(h\cdot n,h\cdot n)$. This remark will be very useful when actually computing second order shape derivatives in Section {\rm \ref{sec second order shape der}}.
}\end{remark}

In what follows we will carry out the explicit calculations of the linear form $l_1$ and bilinear form $l_2$ from Theorem \ref{struct thm} on page \pageref{struct thm} for the following three geometrical shape functionals: volume, barycenter and surface area. 

\begin{example}[Computation of $l_1$]\label{ex l1}
For $\om\in\cO_2$, set $\vol(\om):=\int_\om 1$, $\bc(\om):=\int_\om x$ and $\per(\om):=\int_{\pa\om}1$. For $\xi\in\C^1(\pa\om)$ we have 
\[
l_1^\vol(\xi)=\int_{\pa\om}\xi, \qquad l_1^\bc(\xi)= \int_{\pa\om} x\xi, \qquad l_1^\per(\xi)=\int_{\pa\om}H \xi.
\]
\end{example}
\begin{proof}
The expressions of $l_1^\vol$ and $l_1^\per$ are derived from a direct application of Proposition \ref{hadam form} and Proposition \ref{surface hadam form} respectively. Finally, the computation of $l_1^\bc$ is done component-wise, that is, by applying the Hadamard formula to real valued functional $\om\mapsto \int_\om x_i$ for all $i=1,\dots,N$.
\end{proof}

\begin{example}[Computation of $l_2$]\label{ex l2}
We employ the same notation as in Example {\rm\ref{ex l1}}. For all $\xi\in\C^1(\pa\om)$ the following holds:
\begin{equation}\nonumber
\begin{aligned}
l_2^\vol(\xi,\xi)=\int_{\pa\om}H \xi^2, \qquad 
l_2^\bc(\xi,\xi)=\int_{\pa\om}(n+xH)\xi^2,\\
l_2^\per(\xi,\xi)=\int_{\pa\om}\abs{\grt\xi}^2+\xi^2\pp{H^2-\tr ((D_\tau n)^T D_\tau n)}.
\end{aligned}
\end{equation}
\end{example}
\begin{proof}
As stated in Remark \ref{rmk hadamard is enough}, in order to compute the various bilinear forms $l_2$, it will be enough to compute the shape derivative twice with respect to an Hadamard perturbation. Take now an arbitrary $\xi\in\C^1(\pa\om) $ and an extension $h\in\C^{1,\ali}\rnrn$ that satisfies $h=\xi n$ on $\pa\om$. Put $\Phi(t):=t h$. For ease of exposition, we will first perform our computation for a generic integral functional of the form $i(t)=\int_{\om_t}f(t)$. If $f$ is sufficiently smooth, then by Corollary \ref{generalized hadam form} we have 
\[
i'(t)=\int_{\om_t} f'(t) + \int_{\om_t} \dv\pp{f(t)\, h\circ \pp{\id+th)}\inv}.
\] 
By a further application of Proposition \ref{hadam form} and the divergence theorem we get
\begin{equation}
\dato \int_{\om_t}\dv\pp{f(t)\, h\circ \pp{\id+th}\inv} = \int_{\pa\om} \pp{f'(0)+\dv\pp{f(0)h}}h\cdot n-f(0)(Dh h)\cdot n,
\end{equation} 
where we used the fact that $\pato(\id+th)\inv=-h$ and thus $\pato\pp{h\circ\pp{\id+th}\inv}=-Dhh$.
Recall that $h=\xi n$ and $\dv h-(Dh n)n=\dv_\tau h = H h\cdot n= H\xi$ on $\pa\om$. We get
\begin{equation}\label{a part}
\dato \int_{\om_t}\dv\pp{f(t)\, h\circ \pp{\id+th}\inv}=\int_{\pa\om}f'(0)\xi + \pp{H f(0)+\dn f(0)}\xi^2.
\end{equation}
Therefore, for a functional of the form $t\mapsto \int_{\om_t}f$ (with $f$ independent of $t$) the second order shape derivative consists only of the term in \eqref{a part} and hence
\[
\restr{\frac{d^2}{dt^2}}{t=0}\int_{\om_t}f= \int_{\pa\om}\pp{H f+\dn f}\xi^2.
\]
Now, for $f\equiv 1$ we obtain the bilinear form $l_2^\vol$ and for $f=x_i$ ($i=1,\dots,N$) we get 
\[
\dato \int_{\pa\om_t} x_i= \int_{\pa\om} \pp{H x_i+\gr x_i\cdot n}\xi^2 =\int_{\pa\om} \pp{H x_i+n_i}\xi^2,
\]
which yields the desired expression for $l_2^\bc$. 
As far as the functional $\per$ is concerned, we set $f(t):=\dv\,n_t$, where $n_t$ is a unitary extension of the outward normal to $\pa\om_t$. Hence $\int_{\om_t}\pa_t f=\int_{\pa\om_t}n_t\cdot \pa_t n_t=0$ and the second order shape derivative of $\per(\om_t)$ is given by the term in \eqref{a part} only. In the following we will choose $n=\gr d_{\om}$, where $d_\om$ is the signed distance function, defined in \eqref{dist func}. We get the following (recall the expression for the shape derivative of the unit normal given in Proposition \ref{der of normal}):
\begin{equation}\label{lper almost}
l_2^\per(\xi,\xi)=-\int_{\pa\om} \dv(\grt \xi)\xi+\int_{\pa\om}\pp{H \dv\,n+\dn (\dv\,n )}\xi^2.
\end{equation}
Now, the first integral can be handled as follows using Proposition \ref{tg int by parts}:
\[
-\int_{\pa\om}\dv(\grt \xi)\xi=-\int_{\pa\om}\dv_\tau(\grt \xi)\xi=\int_{\pa\om}\abs{\grt \xi}^2. 
\]
The remaining part of \eqref{lper almost} is simplified by noticing that 
$\dv\,n=\dv_\tau\,n=H$ and that $\dn(\dv\,n)=-\tr((D_\tau n)^T D_\tau n)$. 
To prove the latter, notice that
\[
0= {\De(|\gr d_{\om}|^2)}/{2}= \gr(\De d_{\om})\cdot \gr d_{\om}+ \tr\pp{(D^2 d_{\om})^2}=\dn(\dv\,n)+\tr\pp{(D_\tau n)^T D_\tau n}.
\] 
\end{proof}

\section{State functions and their derivatives}
\label{sec state functions}

Notice that not all integral functionals are like those in Example \ref{ex l1}. Usually, the integrand in those shape functionals depends on the domain indirectly, by means of the solution to some boundary value problem, usually referred to as \emph{ state function} (see for instance the two-phase torsional rigidity functional $E$, defined by \eqref{E}, whose state function $u$ is the solution of \eqref{pb u formal}). In order to compute the shape derivative of such an integral functional, one must first compute the shape derivative of state function (cf. the first term in \eqref{result hadam}). 
The aim of this subsection is twofold: we will first prove some (quite general) differentiability results for state functions and then show how the shape derivative of a state function can in turn be characterized as the solution to some boundary value problem. 

We will give now the definitions of \emph{shape derivative} and \emph{material derivative} of a state function. 
Consider a flow of transformation $\Phi:[0,1)\to \Theta$, where $\Theta$ is a suitable Banach space of mappings from $\rn$ to itself. Fix a sufficiently smooth domain $\om$ and consider a smoothly varying family of functions $u_t$ on $\om_t$: to fix ideas, $u_t$ will be solution to some boundary value problem on $\om_t$ (whose parameters may depend on $t$ indirectly). 
Notice that, for $x\in\om$, then $x\in\om_t$ if $t\ge0$ is small enough. Computing the partial derivative of $u_t$ with respect to $t$ at a fixed point $x\in\om$ yields the so called \emph{shape derivative} of $u_t$; we will write
\begin{equation}\label{def shape der of state function}
u'_{t_0}(x):=\restr{\pa_t}{t=t_0} u_t(x) \quad \tfor t_0\in[0,1).
\end{equation}
On the other hand, differentiating along the trajectories $x\mapsto x+\Phi(t,x)$ gives rise to the \emph{material derivative} of $u_t$:
\begin{equation}\label{def mat der of state function}
\dot u_{t_0}(x):= \restr{\frac{d}{dt}}{t=t_0} u_t(x+\Phi(t,x)) \quad \tfor t_0\in[0,1).
\end{equation}
In what follows, we will also introduce the following auxiliary function $v_t:=u_t\circ \pp{\id+\Phi(t)}$. Notice that, under the notation introduced above, we have $v'_t=\dot u_t$. From now on, for the sake of brevity, we will omit the $t$ subscript in the case $t=0$. 
\begin{remark}\label{u'=u'}\emph{
Notice that the choice of the name and notation in \eqref{def shape der of state function} is not at all a coincidence. Indeed, for fixed $x\in\om$, $u'(x)$ is the shape derivative (intended with the usual meaning) of the functional 
$
t\mapsto u_t(x)\in\RR$.
Of course a Fr\'echet derivative formulation like \eqref{frechet app} is also possible. Moreover, notice that instead of the point-wise definition in \eqref{def shape der of state function} one could define $u'$ ``globally'', as the shape derivative of a shape functional with values in some Banach space $t\mapto u_t\in X$ (so that Theorem \ref{struct thm} of page \pageref{struct thm} can be applied). In this case, notice that, since the domain $\om_t$ changes with $t$, one should first fix a common domain (for instance, extend $u_t$ to the whole $\rn$) in order to properly define $u'$ in this sense. 
}\end{remark}
Although, the {\bf shape} derivative of states functions are an essential constituent in the computation of the shape derivative of integral functionals (see Proposition \ref{hadam form}), it will be easier to prove existence and smoothness result for material derivatives first and then recover the results for shape derivatives by composition. 
In order to show the differentiability of the auxiliary function $v_t$, we will employ the use of the following version of the implicit function theorem, for the proof of which we refer to \cite[Theorem 2.7.2, pp. 34--36]{Nams2001}.
\begin{thm}[Implicit function theorem, \cite{Nams2001}]
\label{ift}
Suppose that $X$, $Y$ and $Z$ are three Banach spaces, $U$ is an open subset of $X\times Y$, $(x_0,y_0)\in U$, and $\Psi:U\to Z$ is a Fr\'echet differentiable mapping such that $\Psi(x_0, y_0)=0$. Assume that the partial derivative $\pa_x\Psi(x_0,y_0)$ of $\Psi$ with respect to $x$ at $(x_0,y_0)$, i.e. the map $\Psi'(x_0,y_0)(\cdottone, 0): X\to Z$, is a bounded invertible linear transformation from $X$ to $Z$. 
Then there exists an open neighborhood $U_0$ of $y_0$ in $Y$ and a unique Fr\'echet differentiable function $f:U_0\to X$ such that $f(y_0)=x_0$, $(f(y),y)\in U$ and $\Psi(f(y),y)=0$ for all $y\in U_0$. 
\end{thm}




For $\phi\in W^{1,\ali}(\rn,\rn)$, we set $\sg_\phi:=\sg_c \chi_{D_\phi}+\chi_{{\Om_\phi}}$ and let $u_\phi\in H_0^1({\Om_\phi})$ denote the weak solution to the following boundary value problem for $\be\ge0$, $\ga>0$:
\begin{equation}\label{u_theta eq cl}
\left\{
\begin{aligned}
\dv(\sg_\phi \gr u_\phi) &=\be u_\phi - \gamma &\text{ in }{\Om_\phi},\\
u_\phi&=0 &\text{ on }\pa{\Om_\phi}.
\end{aligned}
\right.
\end{equation}
The function obtained by extending $u_\phi$ with zero on the rest of $\rn$ will be denoted by the same symbol, $u_\phi$. Moreover, for $\phi\in W^{1,\ali}(\rn,\rn)$ sufficiently small, the map $\id+\phi:\rn\to\rn$ is a bi-Lipschitz homeomorphism and therefore the function
\begin{equation}\label{v_phi}
v_\phi:=u_\phi \circ (\id+\phi)
\end{equation}
is well defined and belongs to $H_0^1(\Om)$.

\begin{theorem}\label{thm diff state func 1}
Let $(D,\Om)$ be a pair of domains of class $\C^1$ with $\overline{D}\subset \Om$. 
\begin{enumerate}[label=(\roman*)]
\item The map $\phi\mapsto v_\phi\in H_0^1(\Om)$ defined by \eqref{v_phi} is of class $\cC^\ali$ in a neighborhood of $0\in W^{1,\ali}(\rn,\rn)$.
\item If $D$ and $\Om$ are of class $\C^2$, then the map $\phi \mapsto v_\phi\in \hoi(\Omega)\cap H^2(D)\cap H^2(\Om\setminus\ol{D})$ is of class $\C^\ali$ in a neighborhood of $0\in W^{2,\ali}\rnrn$.
\item If $D$ and $\Om$ are of class $\C^{2+\al}$, then the map $\phi\mapsto v_\phi\in \cB$ is of class $\C^\ali$ in a neighborhood of $0\in \C^{2+\al}(\rn,\rn)$, where
$$
\mathcal{B}:= \hoi({\Om})\cap \C(\ol \Om)\cap \C^{2+\al}(\overline{\Om}\setminus D)\cap \C^{2+\al}(\overline{D}).
$$
\end{enumerate}
\end{theorem}
\begin{proof}
\begin{enumerate}[label=(\roman*)]
\item We will now prove that $W^{1,\ali}\ni \phi\mapsto v_\phi \in H_0^1(\Om)$
is Fr\'echet differentiable infinitely many times in a neighborhood of $0$. First notice that $v_\phi\in H_0^1(\Om)$ is characterized by
\begin{equation}\label{v_te eq wk}
\int_{\Om} A_\phi\gr v_\phi\cdot \gr \psi = \int_\Om (\ga-\be v_\phi) J_\phi\psi \quad\text{ for all }\psi\in H_0^1(\Om),
\end{equation}
where $J_\phi$ is the Jacobian associated to the map $\id+\phi$ and 
\begin{equation}\label{A li seme}
A_\phi:=\sg
J_\phi \left(I+ D\phi\right)^{-1} ({I}+D\phi^T)^{-1}.
\end{equation}
This can be proved by a change of variable in the weak formulation of $u_\phi$. 
Let us now consider the following operator
\begin{equation}\label{operator F}
F: H_0^1(\Om) \times W^{1,\ali}(\rn, \rn)\ni (v,\phi) \mapsto -\dv \left(A_\phi\gr v \right)+(\be v -\ga)J_\phi \in H^{-1}(\Om).
\end{equation}
By \eqref{v_te eq wk}, we have $F(v_\phi,\phi)=0$. We are going to apply Theorem \ref{ift} of page \pageref{ift} to the map $F$. First, we claim that $F$ is differentiable infinitely many times in a neighborhood of $(u,0)$ (here $u$ is the solution of \eqref{u_theta eq cl} when $\phi\equiv0$ and thus it coincides with $v_0$). As a matter of fact, the map $W^{1,\ali}(\rn,\rn) \ni \phi\mapsto J_\phi=\det( { I}+D\phi)\in L^\ali(\rn)$ is differentiable infinitely many times because also $\phi\mapsto {I}+D\phi\in L^\ali(\rn,\RR^{N\times N})$ is, and the application $\det$ is a polynomial and is continuous with respect to the $L^\ali$ norm. Similarly, the map $\phi\mapsto (I+D\phi)^{-1}$ can be expressed as a Neumann series as $\pp{I+D\phi}\inv=\sum_{k\ge0}(-1)^k (D\phi)^k$ and thus it is $\C^\ali$ in a neighborhood of $0\in W^{1,\ali}(\rn,\rn)$. Therefore $ W^{1,\ali}(\rn,\rn)\ni \phi\mapsto A_\phi\in L^\ali(\rn,\RR^{N\times N})$ is also of class $\C^\ali$. Thus, the map $L^\ali(\rn,\RR^{N\times N})\to H^{-1}(\Om)$ defined by $(A,v) \mapsto -\dv(A \gr v)$ is also of class $\C^\ali$ because both bilinear and continuous. We can conclude that the full map $(v,\phi)\mapsto F(v,\phi)$ is of class $\C^\ali$. Now, its partial Fr\'echet derivative with respect to the variable $v$: $\pa_v F(u,0): H_0^1(\Om)\to H^{-1}(\Om)$ is given by $v\mapsto -\dv(\sg
\gr v)+\be v$ and, since $\be\ge 0$, it is an isomorphism (see \cite[Theorem 1.1]{ers2007}). Therefore, by Theorem \ref{ift} of page \pageref{ift} there exists a $\C^\ali$ branch $\phi\mapsto v(\phi)\in H_0^1(\Om)$ defined for sufficiently small $\phi\in W^{1,\ali}(\rn,\rn)$ such that $F(v(\phi),\phi)=0$. Uniqueness for problem \eqref{v_te eq wk} yields that $v(\phi)=v_\phi$ (and thus the smoothness of the map $\phi\mapsto v_\phi$). 
\item Define the Banach space $X:= \hoi\Ome\cap H^2(D)\cap H^2 (\Om\sm\ol D)$, endowed with the norm $\norm{\cdottone}_X:=\norm{\cdottone}_{\hoi\Ome}+\norm{\cdottone}_{H^2(D)}+\norm{\cdottone}_{H^2(\Om\sm\ol{D})}$ and consider the function 
\begin{equation}\label{wtF}
\begin{aligned}
\wt F: X\times W^{2,\ali}\rnrn \to L^2(D)\times L^2(\Om\sm\ol{D})\times H^{1/2}(\pa D),\\
(v,\phi)\longmapsto\pp{ F(\restr{v}{D},\phi),\, F(\restr{v}{\Om\sm\ol D}, \phi)\,, \big[ (A_\phi \gr v)\cdot n\big]}.
\end{aligned}
\end{equation}
Here, by a slight abuse of notation, $F$ is used to denote the restriction of \eqref{operator F} to $H^2(D)\times W^{1,\ali}\rnrn$ and $H^2(\Om\sm\ol D)\times W^{1,\ali}\rnrn$ respectively. The differentiability of the map $\wt F$ follows from the same arguments used to prove that of $F$. Its partial Fr\'echet derivative $\pa_v \wt F(u,0)(\cdot,0): X\to L^2(D)\times L^2(\Om\sm\ol D)\times H^{1/2}(\pa D)$ is given by 
\[
v\mapsto \pp{-\sg_c \De v+\be v, \, -\De v+\be v,\,[\sg\dn v]}.
\]
The invertibility of $\pa_v\wt F$ amounts to the well posedness of the following transmission problem with data $f\in L^2(D)$, $g\in L^2(\Om\sm\ol D)$ and $h\in H^{1/2}(\pa D)$: 
\begin{equation}
\left\{
\begin{aligned}
-\sg_c \De v + \be &= f& \quad \tin D,\\
- \De v + \be& = g &\quad \tin \Om\sm\ol D,\\
[v]&=0&\quad \ton \pa D,\\
[\sg\dn v]&= h& \quad\ton\pa D,\\
v&=0& \quad\ton\pa\Om.
\end{aligned}
\right.
\end{equation}
The proof of the well posedness of the problem above is based on the extension of the simpler result in the case $h=0$, made possible by means of an auxiliary function $\wt u$. We refer to \cite[Theorem 1.1 and Remark 1.3]{athastra96} for a proof and \cite[Remark 2.1]{cz98} for an explicit construction of $\wt u$.
\item This time one considers the following restriction of the map defined in \eqref{wtF}:
\[
\wt F: 
\mathcal{B}
\times \C^{2+\al}\rnrn \to \C^\al (\ol D)\times \C^\al (\ol \Om \sm D) \times \C^{1+\al}(\pa D).
\]
The proof runs exactly as before, this time employing the use of the sharp Schauder-like estimates given in \cite[Theorem 2.2 and Theorem 2.3]{xb13}. 
\end{enumerate}
\end{proof}

\begin{lemma}\label{5.3.3}
Let $\Psi: W^{1,\ali}\rnrn\to W^{1,\ali}\rnrn$ be continuous at $0$ with $\Psi(0)=\id$ and, for $1\le p<\ali$ let $W^{1,\ali}\rnrn\ni \phi \mapto \pp{g(\phi),\Psi(\phi)}\in L^p(\rn)\times L^\ali\rnrn$ be differentiable at $0$ with $g(0)\in W^{1,p}(\rn)$ and let $g'(0): W^{1,\ali}\rnrn\to W^{1,p}(\rn)$ be continuous. Then the application
\[
\mathcal{G}: W^{1,\ali}\rnrn \to L^p(\rn),\quad\phi\mapto g(\phi)\circ\Psi(\phi)
\]
is differentiable at $0$ and
\[
\mathcal{G}'(0)\phi = g'(0)\phi+\gr g(0)\cdot \Psi'(0)\phi
\]
holds true for all $\phi\in W^{1,\ali}\rnrn$.
\end{lemma}
\begin{proof}
We will show that 
\[
\norm{g(\phi)\circ \Psi(\phi)-g(0)-\gr g(0)\cdot \Psi'(0)\phi-g'(0)\phi}_p=o\big(\norm{\phi}_{1,\ali}\big)\quad \tas \phi\to 0.
\]
We decompose it into four terms
\begin{equation}\nonumber
\begin{gathered}
A(\phi):=\big\{g(\phi)-g(0)-g'(0)\phi\big\}\circ\Psi(\phi),\quad B(\phi):=\gr g(0)\cdot \big\{\Psi(\phi)-\Psi(0)-\Psi'(0)\phi \big\},\\
C(\phi):=g(0)\circ \Psi(\phi)-g(0)-\gr g(0)\cdot \big\{\Psi(\phi)-\Psi(0)\big\},\quad D(\phi):=\pp{g'(0)\phi}\circ \Psi(\phi)-g'(0)\phi.
\end{gathered}
\end{equation}
By change of variables for $A(\phi)$, we have
\begin{equation}\nonumber
\begin{aligned}
\norm{A(\phi)}_p\le& \norm{g(\phi)-g(0)-g'(0)\phi}_p \norm{\Psi(\phi)}_{1,\ali}=o\big(\norm{\phi}_{1,\ali}\big),\\
\norm{B(\phi)}_p\le& \norm{g}_{1,p}\norm{\Psi(\phi)-\Psi(0)-\Psi'(0)\phi}_\ali=o\big(\norm{\phi}_{1,\ali}\big).
\end{aligned}
\end{equation}
The estimate for $C(\phi)$ runs along the same lines as the proof of Lemma \ref{5.2.6}. Put $g:=g(0)$ and $\psi:=\Psi(\phi)-\Psi(0)$. We have
\begin{equation}\label{5.33}
\norm{g\circ(\id+\psi)-g-\gr g\cdot \psi}_p= \norm{\int_0^1 \big\{\gr g (\id+s\psi)-\gr g\big\}\cdot \psi\,ds}_p\le \norm{\psi}_\ali \norm{e(g)}_p,
\end{equation}
where, $e(g)=\int_0^1 |\gr g(\id+s\psi)-\gr g|\,ds$. One then approximates $g$ in the $W^{1,p}$-norm by means of a sequence $\{g_k\}_k\subset \C_0^\ali(\rn)$ and, as in the proof of Lemma \ref{5.2.6}, we have
\[
\norm{e(g)}_p\le \norm{e(g-g_k)}_p+\norm{e(g_k)}_p\le 2 \norm{g-g_k}_{1,p} \pp{1+\norm{\psi}_{1,\ali}}+C_k \norm{\psi}_\ali,
\]
where $C_k$ is a positive constant that depends only on $g_k$. Since, by assumption, $\psi$ tends to $0$ in $L^\ali$ and remains bounded in $W^{1,\ali}$ as $\norm{\phi}_{1,\ali}$ tends to $0$, by passing to the limit as $\norm{\phi}_{1,\ali}\to 0$ and $k\to \ali$ respectively, we get that $\norm{e(g)}_p\to 0$ as $\norm{\phi}_{1,\ali}\to 0$. Therefore, by \eqref{5.33}, $C(\phi)=o\big(\norm{\phi}_{1,\ali}\big)$, because $\norm{\psi}_\ali =O \big(\norm{\phi}_{1,\ali}\big)$.

Lastly, for $D(\phi)$ we estimate the increment by means of the gradient of $g'(0)\phi$, as we did when proving \eqref{composition}. We have
\[
\norm{D(\phi)}_p = O\big(\norm{g'(0)\phi}_{1,p} \norm{\Psi(\phi)-\id}_{\ali} \big)= O\big(\norm{\phi}_{1,\ali}^2 \big).
\] 
\end{proof}
\begin{corollary}\label{corol 3.16}
Let $f\in W^{1,p}(\rn)$, $1\le p <\ali$. Then the application 
\[
\mathcal{F}: W^{1,\ali}\rnrn\to L^p(\rn),\quad \phi\mapto f\circ\pp{\id+\phi}
\]
is of class $\C^1$ in a neighborhood $\mathcal{U}$ of $0\in W^{1,\ali}\rnrn$ and 
\begin{equation}\label{F'(phi0)phi}
\mathcal{F}'(\phi_0) \phi = \big\{ \gr f\circ \pp{\id+\phi_0} \big\}\phi
\end{equation}
holds true for all $\phi_0\in\mathcal{U}$ and $\phi\in W^{1,\ali}\rnrn$.
\end{corollary}
\begin{proof}
Apply Lemma \ref{5.3.3} to $g(\phi):=f\circ\pp{\id+\phi_0}$ and $\Psi(\phi):=\pp{\id+\phi_0}\inv\circ\pp{\id+\phi_0+\phi}$ (cf. \eqref{proof generalized hadam form}). Finally, notice that the map $\phi_0\mapsto\cF'(\phi_0)$ given by \eqref{F'(phi0)phi} is continuous (that is because property \eqref{composition} actually holds for all $L^p$ with $1\le p<\ali$ as one can see by following its proof once again).
\end{proof}

\begin{theorem}\label{thm diff state func 2}
If $(D,\Om)$ is a pair of bounded open sets of class $\C^1$ with $\ol{D}\subset \Om$, then
the map $\phi \mapsto u_\phi \in L^2(\rn)$ is differentiable at $0\in W^{1,\ali}(\rn,\rn)$. Moreover, for all $\phi\in W^{1,\ali}\rnrn$ we have
\[
u' \phi = v'\phi-\gr u \cdot \phi,
\]
where $u'$ and $v'$ denote the Fr\'echet derivatives of $\phi\mapsto u_\phi$ and $\phi\mapto v_\phi$ respectively computed at $0$. 
\end{theorem}
\begin{proof}
This is an immediate consequence of Lemma \ref{5.3.3} with $g(\phi):=v_\phi$ and $\Psi(\phi):=\pp{\id+\phi}\inv$.
\end{proof}
\begin{remark}\label{rmk 3.18}\emph{
The reader 
might wonder what is the regularity that the map $\phi\mapsto u_\phi$ enjoys {\bf in a neighborhood} of $0$ (and not only at $0$, as discussed in Theorem \ref{thm diff state func 2}). To this end, notice that $\phi\mapto \Psi(\phi):=\pp{\id+\phi}\inv\in L^\ali(\rn)$ is differentiable in a neighborhood $\cU$ of $0\in W^{1,\ali}\rnrn$ and for $\phi_0\in\cU$ and $\phi\in W^{1,\ali}\rnrn$: 
\[
\Psi'(\phi_0)\phi= - \left\{\pp{I+D\phi_0}\inv\circ \Psi(\phi_0)\right\}\, \left\{\phi\circ \Psi(\phi_0)\right\}.
\]
Therefore $\Psi$ is of class $\C^1$ seen as a map of $W^{1,\ali}\rnrn\cap\C^{1,\ali}\rnrn\to L^{\ali}\rnrn$ but not of $W^{1,\ali}\rnrn\to L^\ali\rnrn$ (this is because \eqref{composition} extends to $L^\ali$ only for smooth $f$). 
}\end{remark}

Sometimes, the differentiability of $\phi\mapto u_\phi$ in $L^2$ is not enough. Especially when dealing with energy-like functionals like \eqref{E}, it will be useful to control also the differentiability of the gradient of $u_\phi$ in the $L^2$-norm. Since we are working with two-phase problems, finding the right formalism to discuss the differentiability of $\phi\mapto u_\phi$ in more regular spaces can be a bit tricky, but nevertheless possible, as shown in the following theorem.

\begin{theorem}\label{diff u in H1}
Let $(D,\Om)$ be a pair of domains of class $\C^2$ with $\ol D\subset \Om$.
The restrictions of $u_\phi$ to the core $D_\phi$ and the shell $\Om_\phi\sm\ol D_\phi$ admit extensions $u_\phi^c$, $u_\phi^s\in H^1(\rn)$ respectively, such that the maps $\phi\mapto u_\phi^c\in H^1(\rn)$ and $\phi\mapto u_\phi^s\in H^1(\rn)$ are of class $\C^1$ in a neighborhood of $0\in C^{2,\ali}\rnrn$. 
\end{theorem} 
\begin{proof}
We will prove differentiability for the map $\phi\mapto u_\phi^c$, since $u_\phi^s$ is completely analogous. First, we claim that the map
\begin{equation}\label{differenziabilitas di F}
\cF: H^2(\rn)\times C^{1,\ali}\rnrn\to H^1(\rn),\quad (g,\phi)\mapsto g\circ\pp{\id+\phi}
\end{equation}
is of class $\C^1$ in $H^2(\rn)\times \left\{\norm{\phi}_{1,\ali}<1\right\}$. By Corollary \ref{corol 3.16}, we know that, for fixed $g\in H^1(\rn)$, the map $\cF(g,\cdottone): \C^{1,\ali}\rnrn\to L^2(\rn)$ is of class $\C^1$ for $\norm{\phi}_{1,\ali}<1$. The claimed differentiability of \eqref{differenziabilitas di F} amounts to showing that the map
\[
\gr\pp{ \cF (g,\cdottone)}: \C^{1,\ali}\rnrn\to L^2\rnrn,\quad \phi\mapsto \pp{\id+D\phi}^T \gr g\circ \pp{\id+\phi}
\]
is of class $\C^1$ for $\phi$ small. Indeed, as $\C^{1,\ali}\rnrn\ni\phi\mapsto D\phi\in C^{0,\ali}\rnrn$ is of class $\C^\ali$ and $\gr g\in L^2\rnrn$, the smoothness of $\gr \cF$ follows from a further application of Corollary \ref{corol 3.16}.
By linearity with respect to $g$, we get the differentiability of $\cF:H^1(\rn)\times \C^{1,\ali}\rnrn\to L^2(\rn)$. 
Now, consider the map $\phi\mapto \cF\pp{P_c(\restr{v_\phi}{D}),\pp{\id+\phi}\inv-\id}$, where $P_c$ is a continuous linear extension operator from $H^2(D)$ to $H^2(\rn)$ (see \cite{sobolev spaces}). By Theorem \ref{thm diff state func 1} (ii) we have that $\phi\mapsto \restr{v_\phi}{D}\in H^2(D)$ is of class $\C^1$ in a neighborhood of $0\in \C^{2,\ali}\rnrn$. Moreover, by similar reasonings to those carried on in Remark \ref{rmk 3.18}, it can be shown that the map $\phi\mapsto \pp{\id+\phi}\inv\in \C^{1,\ali}\rnrn$ is of class $\C^1$ in a neighborhood of $0\in \C^{2,\ali}\rnrn$. By composition we conclude that the map
\[
\phi\mapsto u_\phi^c:= P_c(\restr{v_\phi}{D})\circ \pp{\id+\phi}\inv \in H^1(\rn)
\]
is an extension of $u_\phi$ to $H^1(\rn)$ that is of class $\C^1$ in a neighborhood of $0\in C^{2,\ali}\rnrn$.
\end{proof}

For $t\mapto\Phi(t)\in \C^{2,\ali}\rnrn$ smooth, $\Phi(0)=0$ and $t\ge0$ small, let $u_t$ denote the solution to \eqref{u_theta eq cl} corresponding to $\sg_t= \sg_c\chi_{D_t}+\chi_{\Om_t\setminus D_t}$. By Theorem \ref{thm diff state func 2} we know that $u_t$ admits a shape derivative, which we will call $u'$. The explicit computation of $u'$ is the content of the following theorem. 

\begin{theorem}\label{u' general thm}
Assume that $(D,\Om)$ is a pair of domains of class $C^{2}$ satisfying $\overline{D}\subset \Om$. Let $\Phi:\C^1\pp{[0,1),\C^{2,\ali}\rnrn}$ satisfy $\Phi(0)=\id$ and $\dato \Phi(t)= h$. Then, the shape derivative $u'\in L^2\Ome\cap \hi(D)\cap \hi(\Om\setminus \overline{D})$ is a weak solution of the following problem: 
\begin{equation}\label{u' eq cl}
\left\{
\begin{aligned}
\sg \De u' &= \be u' &\text{ in } D\cup (\Om\setminus \overline{D}),\\
[\sg \dn u']&= (\sg_c-1)\dv_\tau (\gr_\tau u\, h\cdot n)& \text{ on }\pa D,\\
[u']&=-[\dn u]\, h\cdot n& \text{ on }\pa D,\\
u'&=-\dn u\, h\cdot n & \text{ on }\pa \Om, 
\end{aligned}
\right.
\end{equation}
namely $u'+\gr u \cdot h$ belongs to $\hoi\Ome$ and
\begin{equation}\label{u' eq wk H01}
\int_\Om \sg \gr u'\cdot \gr \psi+\int_{\pa D} (\sg_c-1)(\gr_\tau u\cdot \gr_\tau \psi)h\cdot n = -\be\int_\Om u' \psi
\end{equation}
for all $\psi\in\C_0^\ali\Ome$.
\end{theorem}
\begin{proof}
By Theorem \ref{diff u in H1} we know that $u'$ is well defined. Moreover, by definition $u_t\circ \pp{\id+\Phi(t)}=v_t$. By differentiating we get $u'+\gr u \cdot h=\dot{u}$, which belongs to $\hoi\Ome$ by Theorem \ref{thm diff state func 1}.
Now, take an arbitrary function $\psi\in \C_0^\ali\Ome$. Notice that, for $t>0$ small enough, $\psi$ belongs to $\C_0^\ali(\Om_t)$ as well. Now integrate \eqref{u_theta eq cl} against the test function $\psi$:
\[
\sg_c\int_{D_t} \gr u_t \cdot\gr \psi + \int_{\Om_t\setminus D_t} \gr u_t \cdot \gr\psi = \int_{\Om_t} (\ga - \be u_t)\psi. 
\]
Computing the derivative with respect to $t$ of the above by employing the use of the Hadamard formula, Proposition \ref{hadam form} (again, the hypothesis are fulfilled by Theorem \ref{diff u in H1}), yields
\[
\int_\Om \sg\gr u'\cdot \gr \psi+ \int_{\pa D}[\sg \gr u]\cdot \gr \psi (h\cdot n)=-\beta\int_\Om u'\psi, 
\] 
which is equivalent to the weak formulation given in the statement of the theorem, since $[\sg \gr u]= [\sg \dn u] n + [\sg \gr_\tau u]=(\sg_c-1)\gr_\tau u$ by the transmission condition \eqref{tc}.
We remark that problem \eqref{u' eq wk H01} has a unique solution $u'$ such that $u'+\gr u\cdot h$ belongs to $\hoi\Ome$. Indeed suppose that $u_1$ and $u_2$ are such two solutions, then we claim that $u_3:=u_1-u_2$ is constantly $0$ in $\Om$. As a matter of fact, $u_3\in \hoi\Ome$ and satisfies
\[ 
\int_\Om \sg\gr u_3\cdot \gr \psi=-\be \int_\Om u_3 \psi \quad \text{ for all }\psi\in\C_0^\ali\Ome.
\]
Since $\be\ge0$, then $u_3\equiv 0$ in $\Om$ as claimed.

Now we show that, if $D$ and $\Om$ are smoother ($C^{2+\al}$ is enough), then $u'$ satisfies \eqref{u' eq cl} in the strong sense.
By restricting to test functions $\psi$ in $\C_0^\ali(D)$ and $\C_0^\ali(\Om\setminus \overline{D})$ we get 
\begin{equation}\label{De u'}
\sg \De u' (x) = \be u'(x)\quad \text{ for all } x\in D\cup (\Om\setminus \overline{D}).
\end{equation} 
An integration by parts with \eqref{De u'} at hand gives
\[
\int_{\pa D} [\sg \dn u'] \psi =- \int_{\pa D} (\sg_c-1) (\gr_\tau u \cdot\grt \psi) h\cdot n. 
\]
Now, by an application of the tangential version of integration by parts (see Proposition \ref{tg int by parts} in the Appendix) we get 
\begin{equation}\label{dududu}
- \int_{\pa D} (\sg_c-1) (\gr_\tau u \cdot\grt \psi) h\cdot n = (\sg_c-1)\int_{\pa D}\dv_\tau(\grt u\, h\cdot n)\psi.
\end{equation}
Notice that the right hand side of \eqref{dududu} is meaningful because $u\in\C^{2+\al}(\ol D)$ if $\pa D$ and $\pa\Om$ are of class $\C^{2+\al}$ (see \cite[Theorem 2.2 and Theorem 2.3]{xb13}). 
This implies the second condition of \eqref{u' eq cl} by the arbitrariness of $\psi\in\C_0^\ali(\Om)$. The third and fourth conditions of \eqref{u' eq cl} follow easily from the fact that $u'=\dot{u}-\gr u \cdot h$. Indeed we have $[u']=[\dot{u}]- [\dn u] h\cdot n - [\grt u]\cdot \grt \psi = -[\dn u]h\cdot n$ on $\pa D$ and $u'=\dot{u}-\gr u \cdot h = -\dn u\, h\cdot n$ on $\pa \Om$ because of the boundary condition satisfied by $u_t$.
\end{proof}
\section{Optimal shapes and overdetermined problems}
\label{sec 1-ph equivalence}
In this section we will explain how to use shape derivatives in order to investigate the relationship between the two problems discussed in Chapter \ref{ch torsion+serrin}, namely the maximization of the one-phase torsional rigidity and Serrin's overdetermined problem. 

Let $\Om\subset \rn$ be a bounded domain of class $\C^2$ and $t\mapsto \Phi(t)\in \C^{2,\ali}\rnrn$ be differentiable at $t=0$ with $\Phi(0)=0$ and $\dato \Phi =h$. Moreover, suppose that the perturbation $\Phi$ leaves the volume of $\Om$ unaltered, that is 
\begin{equation}\label{phi li ante ala e suli}
\vol(\Om_t)=\vol(\Om) \tforall t\ge0 \text{ small}.
\end{equation} 
Lastly suppose that $\Om$ is a critical point for the functional $E(\emptyset, \cdottone)$ under the fixed volume constraint, i.e.
\begin{equation}\label{E' li ala}
E'(\emptyset,\Om)(\Phi)=0 \quad\tforall \Phi \text{ satisfying }\eqref{phi li ante ala e suli}.
\end{equation}
In other words, if $u_t$ represents the solution of 
\begin{equation}\label{torsio problemo}
-\De u_t = 1 \,\tin \Om_t, \quad u_t=0\,\ton\pa\Om_t,
\end{equation}
and $j(t):=\int_{\Om_t} |\gr u_t|^2$, then we can rewrite \eqref{E' li ala} by means of the Hadamard formula (Proposition \ref{hadam form}) as follows: 
\[
j'(0)= 2\int_{\Om} \gr u'\cdot \gr u + \int_{\pa\Om} |\dn u|^2 \,h\cdot n,
\]
where $u$ denotes the solution of \eqref{torsio problemo}.
By Theorem \ref{u' general thm}, we know that $\int_\Om \gr u'\cdot \gr u=0$ and thus 
\[
\int_{\pa\Om} |\dn u|^2\,h\cdot n =0. 
\]
Now, if we compute the derivative with respect to $t$ at $t=0$ of \eqref{phi li ante ala e suli} (see also Example \ref{ex l1}) we obtain $\int_{\pa\Om} h\cdot n=0$.
By the arbitrariness of $\Phi$ (see also Proposition \ref{extension of pert} in the next chapter) and Lemma \ref{almost fundamental theorem of} below, we get that the solution $u$ of \eqref{torsio problemo} on $\Om$ must verify the following overdetermined condition 
\[
|\dn u|^2 \equiv constant \;\ton\pa\Om.
\] 
By the Hopf lemma, we conclude that $\dn u$ must be constant: thus $u$ is a solution of Serrin's overdetermined problem \eqref{serrin op} and $\Om$ must be a ball by Theorem \ref{serrin thm}. 
\begin{lemma}\label{almost fundamental theorem of}
Let $\Om$ be a bounded open set and $f\in L^2(\pa\Om)$. If 
\begin{equation}\label{afto}
\int_{\pa\Om} fg =0 \;\tforall g\in L^2(\pa\Om) \;\text{such that } \int_{\pa\Om} g=0,
\end{equation}
then $f$ is constant (almost everywhere) on $\pa\Om$.
If $\pa\Om$ is of class $\C^k$, then the condition \eqref{afto} can be restricted to the subclass of functions $g\in\C^k(\pa\Om)\subset L^2(\pa\Om)$. 
\end{lemma}
\begin{proof}
Let $\ol f$ denote the mean value of $f$, i.e. $\ol f=\fint_{\pa\Om} f= \pp{\int_{\pa\Om} f}/\per(\Om)$. 
Choose $g:=f-\ol f$ in \eqref{afto}. We have
\[
0=\fint_{\pa\Om} f(f-\ol f) = \fint_{\pa\Om} f^2 \,-\ol{f}^2.
\]
On the other hand,
\[
0\le \fint_{\pa\Om} \pp{f-\ol f}^2= \int_{\pa \Om} f^2\,- \ol{f}^2,
\]
with equality holding if and only if $f\equiv \ol f$ almost everywhere in $\pa\Om$.
The final claim of the lemma follows by a density argument.
\end{proof}
\begin{remark}\emph{
We have actually proved a slightly stronger version of Theorem \ref{thm polya} for $\pa\Om$ of class $\C^2$. Indeed balls are not only the {\bf unique} $\C^2$-maximizers for $E(\emptyset,\cdottone)$ under volume constraint, but more generally the only critical shape of class $\C^2$. In particular, no other maximizers or saddle shapes of class $\C^2$ exists for the one-phase functional $E(\emptyset,\cdottone)$ (compare this with Theorem \ref{thm2}).
}\end{remark}

\section{When the structure theorem does not apply}
In Chapter \ref{ch shape derivatives} we gave differentiability results under pretty weak regularity assumptions (both for integral functionals in Section \ref{sec hadam form} and state functions in Section \ref{sec state functions}). Nevertheless, when actually computing those derivatives, we imposed higher regularity in order to write shape derivatives by means of surface integrals. This aim of this section is to show how the same computations can be carried out without imposing any ``extra'' regularity.   

Suppose that $(D,\Om)$ is a pair of bounded domains of class $\C^1$ with $\ol D\subset \Om$. For $\phi\in W^{1,\ali}\rnrn$, let $u_\phi$ be the solution of \eqref{u_theta eq cl} and $v_\phi$ be the function defined by \eqref{v_phi}. Then, consider the map 
\begin{equation}\label{cE}
\phi\mapto \cE(\phi):= \int_{\Om_\phi} \sg_\phi |\gr u_\phi|^2=\int_\Om A(\phi) \gr v_\phi\cdot \gr v_\phi,
\end{equation}
where we have set $A(\phi):=\sg \pp{I+D\phi}^{-T}\pp{I+D\phi}\inv J_\phi$.
By composition we obtain that $\cE(\cdottone)$ is actually of class $\C^\ali$ in a neighborhood of $0\in W^{1,\ali}\rnrn$ (see also Theorem \ref{thm diff state func 1}, (i)). On the other hand both the domains and the perturbation field lack are not regular enough to apply the structure theorem (Theorem \ref{struct thm} on page \pageref{struct thm}). One can wonder how we can write the shape derivatives of $\cE$ then. By differentiating the integral over $\Om$ in \eqref{cE} we get for all $\zeta\in W^{1,\ali}\rnrn$: 
\[
\cE'(\phi)\zeta = \dato \cE(\phi+t\zeta) = \int_\Om A'(\phi)\zeta\, \gr v_\phi \cdot \gr v_\phi + 2\int_\Om A(\phi) \gr v'(\phi)\zeta \cdot \gr v_\phi,
\]
where $v'(\phi)\zeta$ denotes the Fr\'echet differential of the map $\phi\mapto v_\phi$  applied to $\zeta$ (which is well defined by Theorem \ref{thm diff state func 1}). We use the notation $U_\phi:= \pp{I+D\phi}\inv$ and the following identities from matrix calculus:
\[
U'_\cdottone (\phi)\zeta= -U_\phi\inv  D\zeta U_\phi\inv,\quad J'_\cdottone (\phi)\zeta = J_\phi \tr(U_\phi D\zeta). 
\]
We have
\[
A'(\phi)\zeta = -\sg J_\phi \left\{ U_\phi^T D\zeta^T U_\phi^{-T} U_\phi +U_\phi^T U_\phi\inv D\zeta U_\phi\inv+ U_\phi^T U_\phi \tr(U_\phi D\zeta)  \right\}.
\]
One could go on and compute higher order derivatives in a similar fashion. We will give the result concerning the second Fr\'echet derivative of $\cE(\cdottone)$. For $\phi\in W^{1,\ali}\rnrn$ small and arbitrary $\xi,\zeta\in W^{1,\ali}\rnrn$ we have
\begin{equation}\nonumber
\begin{aligned}
\cE''(\phi)(\xi,\zeta)=\dato \cE'(\phi+t \xi)\zeta=\int_\Om A''(\phi)(\xi,\zeta) \gr v_\phi\cdot \gr v_\phi
+ 2\int_\Om A'(\phi)\zeta\, \gr v'(\phi)\xi\cdot \gr v_\phi\\ + 2\int_\Om A'(\phi)\xi \,\gr v'(\phi)\zeta \cdot \gr v_\phi+2\int_\Om A(\phi) \gr v''(\phi)(\xi,\zeta)\cdot \gr v_\phi+2\int_\Om A(\phi) \gr v'(\phi)\xi\cdot \gr v'(\phi)\zeta.
\end{aligned}
\end{equation}
Where, 
\begin{equation}\nonumber 
\begin{aligned}
&A''(\phi)(\xi,\zeta)=\dato A'(\phi+t\xi)\zeta =\\
&\sg J_\phi\bigg\{ 
- D\xi^T D\zeta^T U_\phi^{-T} U_\phi
-U_\phi^{-T}D\zeta^T D\xi^T U_\phi 
-U_\phi^T D\xi D\zeta U_\phi\inv
-U_\phi^T U_\phi\inv D\zeta D\xi \\
&-U_\phi^{-T} D\zeta^T U_\phi^{-T} U_\phi\tr(U_\phi D\xi)
-U_\phi^T U_\phi\inv D\zeta U_\phi\inv \tr(U_\phi D\xi)
+U_\phi^{-T}D\xi^T U_\phi^{-T} U_\phi \tr(U_\phi D\zeta)\\
&+U_\phi^T U_\phi\inv D\xi U_\phi \inv \tr(U\phi D\zeta)
-U_\phi^T U_\phi \tr(U_\phi D\xi)\tr(U_\phi D\zeta) 
+U_\phi^T U_\phi \tr(U_\phi\inv D\zeta U_\phi\inv D\xi)\\
&+U_\phi^{-T}D\xi^T U_\phi ^{-T} U_\phi\inv D\zeta U_\phi\inv 
+U_\phi^{-T}D\zeta^T U_\phi^{-T} U_\phi\inv D\xi U_\phi\inv 
\bigg\}.
\end{aligned}
\end{equation} 
Notice that the expression for $\cE''(\phi)(\xi,\zeta)$ given above is a symmetric bilinear form. Further derivatives of order $k\ge 3$ can be computed inductively in the same way, although the computations will become longer at any step. Finally, notice that, independently of $k$, no second order derivatives with respect to the space variables will ever appear in the process (this confirms the fact that $\phi\in W^{1,\ali}\rnrn$ is enough regularity for the functional $\phi\mapto \cE(\phi)$ to be of class $\C^\ali$).   

\chapter{Two-phase torsional rigidity}\label{2-ph torsion}
\label{ch 2-ph torsion}

In this chapter, we will study the functional $E$ defined by \eqref{E}. In particular, we will analyze the link between optimality and radial symmetry. The results contained in this chapter are original and taken from \cite{cava} and \cite{cava2}.

\section{Perturbations verifying some geometrical constraints}
\label{sec geom constr}

Let us introduce the most general class of perturbations that we will be working with in this chapter. Since we are going to compute shape derivatives of the functional $E$ up to the second order, we want enough regularity for the structure theorem (Theorem \ref{struct thm} on page \pageref{struct thm}) to apply. We define
\[
\A:=\setbld{\Phi\in \C^2\pp{[0,1),\C^{2,\ali}\rnrn}}{\Phi(0)=0}.
\]
Moreover, for all bounded open sets $\om$ of class $\C^2$, we set
\[
\A_{\vol(\om)}:=\setbld{\Phi\in\A}{\vol(\om_t)=\vol(\om)}, \quad \A_{\bc(\om)}:=\setbld{\Phi\in\A}{\bc(\om_t)=\bc(\om)}.
\]
For all $\Phi\in\cA_{\vol(\om)}$, Example \ref{ex l1} and Example \ref{ex l2} yield the following two conditions: 
\begin{align}
&\int_{\partial \omega} h\cdot n =0, \quad\quad\quad&\text{($1^{\rm st}$ order volume preserving)} \label{1st}\\
&\int_{\partial \omega} H (h\cdot n)^2+ \int_{\partial \omega}Z=0. &\text{($2^{\rm nd}$ order volume preserving)}\label{2nd}
\end{align} 

If $\Phi\in\A_{\bc(\om)}$, then, by Example \ref{ex l1}:
\begin{equation}\label{bar 1st}
\int_{\pa\om} x_i\, (h\cdot n) =0 \quad \text{ for all }i=1,\dots,N.
\end{equation}

We will consider the following class of perturbations:
$$
\A^*:=\A_{\vol(D)}\cap \A_{\vol(\Om)}\cap \A_{\bc (\Om)}.
$$

\begin{proposition}\label{extension of pert}
Take $h\in\C^{1,\ali}(\rn, \rn)$. Suppose that $h$ satisfies \eqref{1st} for $\om=D,\Om$ and \eqref{bar 1st} for $\om=\Om$. Then there exists a perturbation $\widetilde{\Phi}\in\A^*$ such that $\widetilde{\Phi}(t)=th+o(t)$ as $t\to 0$. 
\end{proposition}
\begin{proof}
We will give an explicit construction of $\widetilde{\Phi}$. First, we put $D_t:=(\id+th)D$, $\Om_t:= (\id+th)\Om$. Now, we define the following auxiliary perturbations:
\begin{equation}\nonumber
\begin{aligned}
\widetilde{\Phi}_-:= \sqrt[\leftroot{-1}\uproot{5}\scriptstyle N]{\frac{\vol(D)}{\vol(D_t)}}\left(\id+th\right)-\id,\quad
\widetilde{\Phi}_+:=\sqrt[\leftroot{-1}\uproot{5}\scriptstyle N]{\frac{\vol(\Om)}{\vol(\Om_t)}}\left(\id+th - \bc(\Om_t)\right) - \id.
\end{aligned}
\end{equation}
By definition we have $\widetilde{\Phi}_-\in\A_{\vol(D)}$ and $\widetilde{\Phi}_+\in\A_{\vol(\Om)}\cap \A_{\bc (\Om)}$. We will now ``blend them together'' by means of a bump function. Let $\eps_0>0$ be a sufficiently small constant, such that $D+B_{2\eps_0}\subset \Om$. Take now a smooth bump function $\eta:\rn\mapsto [0,1]$ that is constantly equal to $1$ in $D+B_{\eps_0}$ and vanishes outside $D+B_{2\eps_0}$ and put: 
$$
\widetilde{\Phi}:=\eta\widetilde{\Phi}_-+\left(1-\eta\right)\widetilde{\Phi}_+ 
.$$
By construction, $\widetilde{\Phi}\in\A^*$. Moreover, a simple calculation with \eqref{1st} and \eqref{bar 1st} at hand ensures that $\pato\widetilde{\Phi}=h$ as claimed. 
\end{proof} 

Since we are working with a shape functional that takes a pair of domains $(D,\Om)$ as input, for all $\Phi\in\A$, in what follows it will be useful for our purposes to separate its contributions on $\pa D$ and $\pa\Om$. For a fixed pair $(D,\Om)$ take some small $\eps_0$ such that $D+B_{2\eps_0}\subset \Om$ as done previously in the proof of Proposition \ref{extension of pert} and define 
\[
\A_-:=\setbld{\Phi\in\A}{\Phi(t,x)=0 \text{ if } x\notin D+B_{2\eps_0} }, \; 
\A_+:=\setbld{\Phi\in\A}{\Phi(t,x)=0\text{ if }x\in \ol{D+B_{\eps_0}}}.
\]
Notice that for every $\Phi\in\A$ there exist some $\Phi_\pm\in\A_\pm$ such that $\Phi=\Phi_-+\Phi_+$ and although such decomposition is not unique, the values of $\Phi_\pm$ are uniquely determined (and actually equal to $\Phi$) on $\ol{D+B_{\eps_0}}$ and $\rn\setminus (D+B_{2\eps_0})$ respectively. In accordance with the notation for $\Phi$ we will write 
\begin{equation}\label{hpm}
\Phi_\pm= t h_\pm + o(t) \quad \text{ as } t\to 0. 
\end{equation}
In a similar manner we put 
\[
\A_\pm^*:=\A^*\cap \A_\pm.
\]
\section{First order shape derivatives}

\subsection{Computation of $E'$ and proof of Theorem \ref{thm1}}
\label{ssec E'}

\begin{theorem}\label{E'=}
Let $(D,\Om)$ be a pair of domains of class $\C^2$ satisfying $\overline{D}\subset \Om$. The first order shape derivative of the functional $E$ computed at $(D,\Om)$ with respect to an arbitrary perturbation $\Phi\in\A$ is given by 
\[
E'(D,\Om)(\Phi)=l_1^E(h\cdot n)=(1-\sg_c) \int_{\pa D} \pp{\sg_c \abs{\dn u}^2+\abs{\gr_\tau u}^2} h\cdot n + \int_{\pa \Om} \abs{\dn u}^2 h\cdot n. 
\]
\end{theorem}
\begin{proof}
For a fixed perturbation $\Phi\in\A$, we will apply the Hadamard formula, Proposition \ref{hadam form}, to the integral functional 
\begin{equation}\label{e(t) li seme}
e(t):=E(D_t,\Om_t)=\int_{\Om_t} \sg_t \abs{\gr u_t}^2.
\end{equation} 
Notice that the integrand in \eqref{e(t) li seme} does not actually satisfy the assumptions of Proposition \ref{hadam form}. Therefore we will need to split the integrals into two parts, namely $D_t$ and $\Om_t$ and then apply the Hadamard formula to both. This yields
\begin{equation}
\begin{aligned}
E'(D,\Om)(\Phi)=e'(0)= \dato \sg_c\int_{D_t} \abs{\gr u_t}^2+\dato \int_{\Om_t\setminus \overline{D_t}}\abs{\gr u_t}^2 =\\
2\int_\Om \sg \gr u\cdot \gr u' + \int_{\pa D} [\sg \abs{\gr u}^2] h\cdot n + \int_{\pa \Om} \abs{\dn u}^2 h\cdot n.
\end{aligned}
\end{equation}
We now get rid of the volume integral $\int_\Om \sg\gr u\cdot \gr u'$ in the above. To this end, notice that, by a density argument, the weak formulation \eqref{u' eq wk H01} holds true even when we choose $u$ as a test function. Now, as $\be=0$ in this case, we obtain:
\begin{equation}\label{almost e'(0)}
e'(0)=2(1-\sg_c)\int_{\pa D} |\grt u|^2 h\cdot n + \int_{\pa D} [\sg |\gr u|^2] h\cdot n + \int_{\pa \Om} \abs{\dn u}^2 h\cdot n.
\end{equation} 
We can split the normal and tangential parts of the gradient of $u$ in the integral over $\pa D$ above: 
\[
l_1^E(h\cdot n)=e'(0)= \int_{\pa D} \sg_c \dn u [\dn u] h\cdot n+ (1-\sg_c)\int_{\pa D}\abs{\gr_\tau u}^2 h\cdot n + \int_{\pa \Om} \abs{\dn u}^2 h\cdot n. 
\]
Finally, we can rewrite the jump part by means of the transmission condition \eqref{tc} and rearrange the terms as in the statement of the theorem.
\end{proof}
\begin{remark}\emph{
If $(D_0,\Om_0)$ are concentric balls, then the corresponding solution $u$ is radially symmetric. This means that $\gr_\tau u$ vanishes on $\pa D_0$, while $\dn u$ is constant on both $\pa D_0$ and $\pa \Om_0$. Hence, $l_1^E(D_0,\Om_0)=0$ for all $\Phi\in\A$ that satisfy the first order volume preserving condition \eqref{1st} on both $\pa D_0$ and $\pa \Om_0$, and, in particular, for all $\Phi\in\A^*$. This proves Theorem \ref{thm1}. 
}\end{remark}
\begin{remark}\emph{
Just as done in Section {\rm \ref{sec 1-ph equivalence}}, the condition $E'(D,\Om)(\Phi)=0$ for all $\Phi\in\A^*$ can be restated as an overdetermined problem, as follows:
\begin{equation}
\left\{
\begin{aligned}
-{\rm div} (\sg \gr u)&=1 & \text{ in }\Om,\\
u&=0 & \text{ on }\pa\Om,\\
\sg_c\abs{\dn u}^2+\abs{\grt u}^2&=c_1 &\text{ on }\pa D,\\
\pa_n u &= c_2 &\text{ on }\pa\Om,
\end{aligned}
\right.
\end{equation}
where the overdetermined condition on $\pa D$ has to be intended in the sense of traces taken from the inside of $D$ and $c_1$, $c_2$ are real constants determined by the data of the problem.
}\end{remark}

\subsection{Explicit computation of $u'$ for concentric balls}

As we know from the abstract structure theorem (Theorem \ref{struct thm} on page \pageref{struct thm}), the shape derivative of the state function $u'$ too depends on $h\cdot n$ in a linear fashion (although this statement is also a direct consequence of the explicit calculations in Theorem \ref{u' general thm}). For arbitrary $\Phi\in\A$, with $\Phi=\Phi_-+\Phi_+$, the first order shape derivative $u'$ of the state function $u$ with respect to $\Phi$, can be decomposed as $u'=u'_-+u'_+$, where $u'_\pm$ are the shape derivatives of $u$ with respect to the perturbation $\Phi_\pm$. In the special case when $D$ and $\Om$ are concentric balls (which will be denoted by $D_0:=B_R$ and $\Om_0:=B_1$), the functions $u'_\pm$ are solutions to the following problems and can be computed explicitly by separation of variables. 
\noindent\begin{minipage}{.5\linewidth}
\begin{equation}\label{u'in}
\begin{cases}\De u'_{-}=0 \quad\mbox{ in } D_0\cup (\Om_0\setminus\ol{D_0}),\\
[\sg\, \pa_n u'_{-}]=0 \quad\mbox{ on }\pa D_0,\\
[u'_{-}]=\frac{1-\sg_c}{\sg_c}\frac{R}{N} h_{-}\cdot n\quad\mbox{ on }\pa D_0,\\
u'_{-}=0 \quad \mbox{ on }\pa\Om_0.
\end{cases}
\end{equation}
\end{minipage}%
\begin{minipage}{.5\linewidth}
\begin{equation}\label{u'out}
\begin{cases}
\De u'_{+}=0 \quad \mbox{ in } D_0\cup (\Om_0\setminus\ol{D_0}),\\
[\sg\, \pa_n u'_{+}]=0 \quad\mbox{ on }\pa D_0,\\
[u'_{+}]=0\quad \mbox{ on }\pa D_0,\\
u'_{+}= \frac{1}{N}\, h_{+}\cdot n \quad \mbox{ on }\pa\Om_0.
\end{cases}
\end{equation}
\end{minipage}

\vspace*{5mm}
\begin{proposition}\label{prop u' kepeken sphar}
Let $\Phi\in\A$ and assume it to be decomposed as $\Phi=\Phi_-+\Phi_+$. With the same notation as {\rm\eqref{hpm}}, suppose that for some real constants $\al_{k,i}^\pm$, the following expansions hold for all $\theta\in\SS^{N-1}$ (see Appendix \ref{app sphar} for the notation concerning spherical harmonics and their fundamental properties):
\begin{equation}\label{h_in h_out exp}
(h_-\cdot n)(R\theta)=\sum_{k=1}^\ali\sum_{i=1}^{d_k}\al_{k,i}^- Y_{k,i}(\theta), \quad 
(h_+\cdot n)(\theta)=\sum_{k=1}^\ali\sum_{i=1}^{d_k}\al_{k,i}^+ Y_{k,i}(\theta).
\end{equation} 
Then, the functions $u'_\pm$ admit the following explicit expression for $\theta\in\SS^{N-1}$:
\begin{equation}\label{u'pm li seme}
u'_\pm(r\theta)=\begin{cases}\displaystyle
\sum_{k=1}^\ali\sum_{i=1}^{d_k}\al_{k,i}^\pm B_k^\pm r^k Y_{k,i}(\theta) \quad &\tfor r\in[0,R],\\
\displaystyle
\sum_{k=1}^\ali\sum_{i=1}^{d_k}\al_{k,i}^\pm\pp{C_k^\pm r^{2-N-k}+D_k^\pm r^k}Y_{k,i}(\theta)\quad&\tfor r\in(R,1],
\end{cases}
\end{equation}
where the constants $B_k^\pm$, $C_k^\pm$ and $D_k^\pm$ are defined as follows
\begin{equation*}
\begin{aligned}
&B_k^- = \frac{1-\sg_c}{\sg_c}R^{-k+1}\left( (N-2+k)R^{2-N-2k} +k\right)/F, \quad\quad 
C_k^- = -D_k^- = (\sg_c-1)k R^{-k+1}/F, \\
&B_k^+ = (N-2+2k)R^{2-N-2k}/F, \quad C_k^+ = (1-\sg_c)k/F,\quad D_k^+ = (N-2+k+k\sg_c)R^{2-N-2k}/F,
\end{aligned}
\end{equation*}
and the common denominator $F= N(N-2+k+k\sg_c)R^{2-N-2k}+k N (1-\sg_c)>0$.
\end{proposition} 
\begin{proof}
We will compute here the expression for $u'_+$ only, as the case of $u'_-$ is completely analogous (we refer to \cite[Section 4]{cava} for the details). 
Let us pick an arbitrary $k\in\{1,2,\dots\}$ and $i\in\{1,\dots, d_k\}$. We will use the method of separation of variables to find the solution of problem \eqref{u'out} in the particular case when $h_+\cdot n=Y_{k,i}$ on $\pa\Om_0$ and then the general case will be recovered by linearity.
We will be searching for solutions to \eqref{u'out} of the form $u'_+=u'_+(r,\theta)=f(r)g(\theta)$ (where $r:=\abs{x}$ and $\theta:=x/\abs{x}$ for $x\ne 0$).
Using the well known decomposition formula for the Laplace operator into its radial and angular components (see Proposition \ref{decomp lapl}), the equation $\De u'_+=0$ in $D_0 \cup (\Omega_0\setminus \overline{D_0})$ can be rewritten as
$$
0 = \pa_{rr}f(r)g(\theta)+\frac{N-1}{r}\pa_r f(r)g(\theta)+\frac{1}{r^2}f(r)\De_{\tau}g(\theta) \;\text{for }r\in(0,R)\cup (R,1),\, \theta\in\SS^{N-1}.
$$
Take $g=Y_{k,i}$.
Under this assumption, we get the following equation for $f$:
\begin{equation}\label{f}
\pa_{rr}f+\frac{N-1}{r}\pa_r f-\frac{\lambda_k}{r^2}f=0 \quad\text{in } (0,R)\cup (R,1).
\end{equation}
Since we know that $\la_k=k(k+N-2)$, it can be easily checked that, on each interval $(0,R)$ and $(R,1)$, any solution to the above consists of a linear combination of the following two independent solutions:
\begin{equation}\label{xieta}
f_{sing}(r):= r^{2-N-k}\ \quad \text{ and } \quad f_{reg}(r):= r^k. 
\end{equation}
Since equation \eqref{f} is defined for $r\in (0,R)\cup (R,1)$, we have that the following holds for some real constants $A_k^+$, $B_k^+$, $C_k^+$ and $D_k^+$;
$$
f(r)= \begin{cases}
A_k^+r^{2-N-k}+B_k^+r^k \quad&\text{for } r\in(0,R),\\
C_k^+r^{2-N-k}+D_k^+ r^k \quad&\text{for } r\in(R,1).
\end{cases}
$$
Moreover, since ${2-N-k}$ is negative, $A_k^+$ must vanish, otherwise a singularity would occur at $r=0$.
The other three constants can be obtained by the interface and boundary conditions of problem \eqref{u'out} recalling that we are assuming $h_+\cdot n = Y_{k,i}$ on $\pa\Om_0$.
We get the following system:
\[
\begin{cases}
C_k^+ R^{2-N-k}+ D_k^+R^k-B_k^+R^k= 0,\\
\sg_c kB_k^+ R^{k-1}= {(2-N-k)} C_k^+ R^{{1-N-k}} + k D_k^+ R^{k-1},\\
C_k^++D_k^+=1/N.
\end{cases}
\]
By solving it we obtain the coefficients of the series representation \eqref{u'pm li seme} of $u'_+$.
\end{proof}

\section{Second order shape derivatives}
\label{sec second order shape der}

In this section we will carry out the computation of the second order shape derivative of the shape functional $E$ at the radially symmetric configuration $(D_0,\Om_0)$. 
\subsection{Computation of $E''$}
\label{ssec comp E''}
The computation of $E''(D_0,\Om_0)(\Phi)=l_2^E(h\cdot n, h\cdot n)+l_1^E(h\cdot n)$ for $\Phi\in\A^*$ will require two steps. First, we will compute the bilinear form $l_2^E$ by means of Hadamard perturbations as done in Example \ref{ex l2} and finally we will take care of the term containing $Z$ using the second order volume preserving condition \eqref{2nd}.
\begin{proposition}\label{prop l2e}
Let $\Phi\in\A$. Then, the bilinear form $l_2^E$ admits the following explicit expression:
\begin{equation}\nonumber
\begin{aligned}
&l_2^E(h\cdot n, h\cdot n)=2\int_{\pa \Om_0} \dn u\,\dn u' \,(h\cdot n) + 2\int_{\pa \Om_0} \pa_n u \,\pa_{nn} u (h\cdot n)^2 + \int_{\pa \Om_0} |\dn u|^2 H (h\cdot n)^2 \\
& +2\int_{\pa D_0} \left[ \sg \dn u \,\dn u'\right] (h\cdot n) + 2\int_{\pa D_0} \sg_c\pa_n u\,[\pa_{nn} u] (h\cdot n)^2 + \int_{\pa D_0} \left[ \sg |\dn u|^2\right] H (h\cdot n)^2.
\end{aligned}
\end{equation}
\end{proposition}
\begin{proof}
We will proceed along the same lines of Example \ref{ex l2}. As stated in Remark \ref{rmk hadamard is enough}, we know that $E''(\Phi)=l_2^E(h\cdot n,h\cdot n)$ in the special case that $\Phi\in\A$ is an Hadamard perturbation. 
Therefore, for all $\Phi\in\A$ of the form $\Phi=\id+th$ with $h_\tau \equiv 0$ on $\pa D_0\cup\pa\Om_0$, by employing the explicit form of the first order shape derivative given by \eqref{almost e'(0)} and reasoning as in the proof of Corollary \ref{generalized hadam form}, we can write
\begin{equation}\label{first E''}
\begin{aligned}
l_2^E=
\dato\pp{2(1-\sg_c) \int_{\pa D_t} |\grt u_t|^2 \xi_t 
+ \int_{\pa D_t} \left[\sg |\gr u_t|^2\right]\xi_t 
+ \int_{\pa \Om_t} \sg |\gr u_t|^2\xi_t},
\end{aligned}
\end{equation}
here we have put $\xi_t=h_t \cdot n_t$, where $h_t=h \circ \left(\id+\Phi(t)\right)^{-1}$ and $n_t$ denotes the outward unit normal to both $\pa D_t$ and $\pa \Om_t$. Let us examine with \eqref{first E''} term by term. 
First of all, we claim that 
\[\dato\int_{\pa D_t}|\grt u_t|^2\xi_t=0.\] 
By definition of tangential gradient \eqref{tg grad} and Proposition \ref{der of normal} we see that $|\grt u_t|^2$ is differentiable at $t=0$, and the same goes for $\xi_t$. We will now apply Proposition \ref{hadam form 2} with $g(t)=|\grt u_t|^2\xi_t$. At a glance it might look like we do not have enough regularity to apply Proposition \ref{hadam form 2} since we do not have control over the gradient of $u_t$ in the right Sobolev space, nevertheless, this is just one of the ``artificial'' regularity that comes from the composition $\id=\pp{\id+\Phi(t)}\circ \pp{\id+\Phi(t)}\inv$. Indeed notice that 
\[
\gr u_t\circ\pp{\id+\Phi(t)}=\pp{I+D\Phi(t)}^T \gr v_t
\]
and conclude by Theorem \ref{thm diff state func 1}.
Now, since the term $|\grt u_t|$ appears squared in $g(t)=|\grt u_t|^2\xi_t$, then $g(0)=g'(0)=0$ on $\pa D_0$ (recall that for $t=0$, $u$ is a radial function, and thus $\grt u =0$). Thus $\dato \int_{\pa D_t}|\grt u_t|^2\xi_t=0$ as claimed. 
Now, \eqref{first E''} can be rewritten in the following compact way:
\begin{equation}\label{(A) and (B)}
l_2^E(h\cdot n,h\cdot n)= \underbrace{\restr{\frac{d}{dt}}{t=0} \int_{\pa D_t} \!\!\! f(t)\, h_t\cdot n_t^1}_{(A)} + \underbrace{\restr{\frac{d}{dt}}{t=0} \int_{\pa (\Om_t\setminus \ol{D_t})}\!\!\! f(t)\, h_t\cdot n_t^2}_{(B)},
\end{equation}
where $f(t):=\sg_t |\gr u_t|^2$, and $n_t^1$ (respectively $n_t^2$) denotes the unit normal vector to $\pa D_t$ (respectively $\pa (\Om_t\setminus \ol{D_t})$) pointing in the outward direction with respect to the domain $D_t$ (respectively $\Om_t\setminus \ol{D_t}$).
We first deal with the term $(A)$ of \eqref{(A) and (B)}. We get 
$$
(A)= \restr{\frac{d}{dt}}{t=0} \int_{D_t} \dv\left( f(t)\, h\circ \left(\id+\Phi(t)\right)^{-1}\right),
$$
The divergence theorem, followed by an application of the Hadamard formula (Proposition \ref{hadam form}), yields 
$$
(A)= \int_{D_0} \restr{\frac{\pa}{\pa t}}{t=0} \dv\left( f(t) \, h\circ\left(\id+ \Phi(t)\right)^{-1}\right) + \int_{\pa D_0} \dv\left( f(0) h\right) \, h\cdot n=(A1)+(A2).
$$
We have
\begin{equation}\nonumber
\begin{aligned}
(A1)= 
\int_{\pa D_0} f'(0) h \cdot n - \int_{\pa D_0} f(0) \left(D h \, h\right)\cdot n,
\quad (A2)= \int_{\pa D_0} \left( \gr f(0)\cdot h + f(0)\dv h\right) h\cdot n.
\end{aligned} 
\end{equation}
Moreover, as $h=(h\cdot n)n$ on $\pa D_0$ by hypothesis, we get
\begin{equation}\label{(A) ii kanji}
(A)=\int_{\pa D_0} f'(0) h\cdot n+ \int_{\pa D_0} \pa_n f(0) (h\cdot n)^2
+ \int_{\pa D_0} f(0)\left(\dv h - n\cdot (Dh \, h)\right)\, h\cdot n.
\end{equation}

Now, by the definition of tangential divergence \eqref{tg div} and \eqref{decomp tg div} (recall that by assumption $h_\tau=0$ on $\pa D_0$) we get:
$
\dv h - n\cdot (Dh \, n)= \dv_\tau h =\dv_\tau \big( (h\cdot n) n\big)= H \, h\cdot n
$.

Recalling the definition of $f(t)$, we can rewrite \eqref{(A) ii kanji} as follows
$$
(A)= 2\int_{\pa D_0}\sg \gr u \cdot \gr u' \, (h\cdot n) + 2\int_{\pa D_0} \sg \pa_n u (\pa_{nn} u ) (h\cdot n)^2+\int_{\pa D_0}\sg |\gr u|^2 H (h\cdot n)^2.
$$
The term labeled $(B)$ in \eqref{(A) and (B)} can be computed analogously. The claim of Proposition \ref{prop l2e} is finally obtained by combining the two terms $(A)$ and $(B)$ and recalling that $\grt u=0$ on $\pa D_0\cup \pa\Om_0$.
\end{proof}

The following theorem is an immediate consequence of Proposition \ref{prop l2e} and the combination of Theorem \ref{E'=} and \eqref{2nd}. 
\begin{theorem}\label{thm E''=}
For all $\Phi\in\A^*$, the following holds: 
\begin{equation}\nonumber
\begin{aligned}
E''(\Phi)=& +2\int_{\pa D_0} \left[ \sg \dn u\, \dn u'\right] (h\cdot n) + 2\int_{\pa D_0} \sg_c\pa_n u\, [\pa_{nn} u] (h\cdot n)^2\\
&+2\int_{\pa \Om_0} \dn u \,\dn u' \,(h\cdot n) + 2\int_{\pa \Om_0} \pa_n u \,\pa_{nn} u (h\cdot n)^2.
\end{aligned}
\end{equation}
\end{theorem}
\begin{remark}\emph{
Theorem {\rm\ref{thm E''=}} actually holds true for all $\Phi\in\A$ that satisfy just the second order volume preserving condition \eqref{2nd} for $\omega= D_0$, $\Om_0$. In particular, we have not used the barycenter preserving condition yet.
}\end{remark}

\subsection{Analysis of the non-resonant part}
\label{ssec anal non res}
Since we know that $u'$ depends linearly on $h\cdot n$ (see for example Theorem \ref{struct thm} on page \pageref{struct thm} or also Theorem \ref{u' general thm}), Theorem \ref{thm E''=} tells us that $E''(D_0,\Om_0)(\Phi)$ is a quadratic form in $h\cdot n$ for all $\Phi\in\A^*$. In particular, $E''(D_0,\Om_0)(\Phi_-+\Phi_+)= E''(D_0,\Om_0)(\Phi_-)+E''(D_0,\Om_0)(\Phi_+)$ for $\Phi_\pm\in\A_\pm^*$ is {\bf not true} in general (although it can happen, even in non trivial cases).

\begin{figure}[h]
\centering
\includegraphics[width=0.9\linewidth,center]{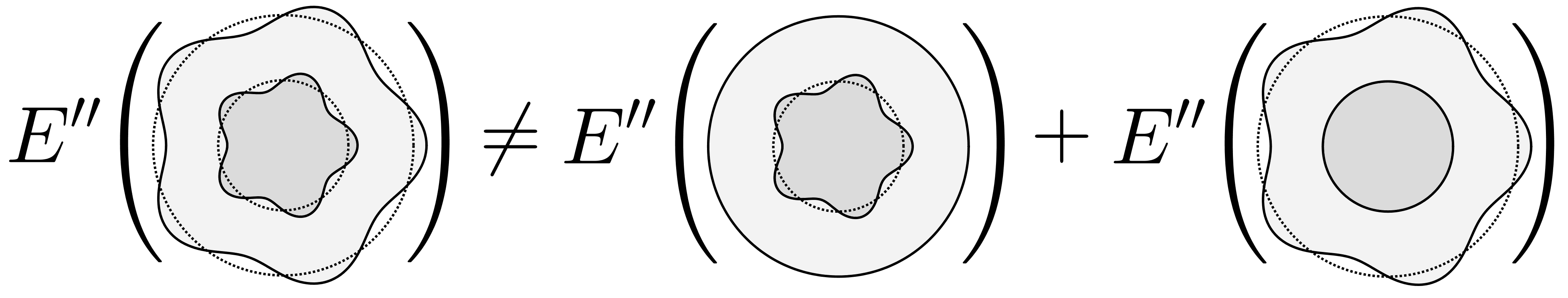}
\caption{$E''$ is nonlinear
} 
\label{nonlinear}
\end{figure}

In what follows we will assume that the expansion \eqref{h_in h_out exp} holds true for $h_\pm$. Combining the result of Theorem \ref{thm E''=} and the explicit expressions for $u$ and $u'=u'_-+u'_+$ (given by \eqref{u} and Proposition \ref{prop u' kepeken sphar} respectively) yields the following. 
\begin{equation}\label{E''=in+out+res}
E''(D_0,\Om_0)(\Phi)=\sum_{{k=1}}^\ali\sum_{{i=1}}^{d_k}\left\{ \left(\al_{k,i}^\ri\right)^2 E''_\ri (k) +\left(\al_{k,i}^\ro\right)^2 E''_\ro (k) + \al_{k,i}^\ri\, \al_{k,i}^\ro \, E''_{\rm res}(k)\right\},
\end{equation}
where
\begin{equation}\label{in out res li seme}
\begin{aligned}
E''_\ri(k)&= \frac{2 R^N}{N} \left( \frac{1-\sg_c}{\sg_c} \right) \left(F - k\left( k(1-\sg_c)+(N-2+k)(1-\sg_c)R^{2-N-2k}\right) \right)\bigg/F,\\
E''_\ro(k) &= \frac{2}{N}\left( F-k\left( (-N+2-k)(1-\sg_c)+(N-2+k+k\sg_c)R^{2-N-2k} \right)\right)\bigg/ F,\\
E''_{\rm res}(k)&= \frac{4 (\sg_c-1) R^{1-k}}{N}\left((N-2)k+2k^2\right)\bigg/ F,
\end{aligned}
\end{equation}
and $F$ is the term defined at the end of the statement of Proposition \ref{prop u' kepeken sphar}. 
The term $E''_{\rm res}$ will be referred to as the \emph{resonant part} of $E''$. As we can see from \eqref{E''=in+out+res}, the resonant part $E''_{\rm res}$ arises when the perturbations $h_-$ and $h_+$ both have a non-zero component corresponding to the same spherical harmonic $Y_{k,i}$. 

In this subsection we will consider only the coefficients $k\in\{1,2,\dots\}, i\in\{1,\dots d_k\}$ such that $\al_{k,i}^\ri\, \al_{k,i}^\ro =0$ (in other words we will consider only the non-resonant part of $E''$). 
Under this assumption the contributions of $E''_\ri(k)$ and $E''_\ro(k)$ can be analyzed separately. We have the following result.

\begin{lemma}\label{lemma decreasing}
Consider the functions $\NN\ni k\mapsto E''_\pm$. The following holds.
\begin{enumerate}[label=(\roman*)] 
\item The function $k\mapsto E''_-(k)$ is strictly decreasing for $\sg_c\ne1$ and constantly zero otherwise.
\item The function $k\mapsto E''_+(k)$ is strictly decreasing for all $\sg_c>0$.
\end{enumerate}
\end{lemma}
\begin{proof}
In the following, we will replace the integer parameter $k$ with a real variable $x$ and study the function $x\mapsto E''_\pm(x)$ in $(0,\infty)$.
The calculations are going to be pretty long, although elementary. For the sake of readability we will adopt the following notation:
\begin{equation}\label{notation pona}
L:=R^{-1}>1,\quad \lambda:=\log(L)>0,\quad M:=N-2\ge0;\quad P=P(x):=L^{2x+M}>1.
\end{equation}
\begin{enumerate}[label=(\roman*)]
\item First we will prove the result about $E''_-$. Rearranging the terms in \eqref{in out res li seme} yields:
\[
E_-''(x)=\frac{2R^N}{N}\pp{\frac{1-\sg_c}{\sg_c}} - \frac{2 R^N (1-\sg_c)^2}{N\sg_c}\,\frac{x^2+(Mx+x^2)P}{F}.
\]
We will show that $x\mapsto j(x):=\pp{x^2+(Mx+x^2)P}/F$ is strictly increasing in $(0,\ali)$. To this end we compute the derivative 
\[
\frac{d}{dx} j(x)= \frac{MP(MP+2Px+2x)+x^2(P+1)^2+\sg_c x^2P (P-1/P-4x\la-2M\la)}{F^2}.
\]
The denominator in the above is positive and we claim that also the numerator is. To this end it suffices to show that the quantity multiplied by $\sg_c x^2P$ in the numerator, namely $P-1/P-4x\la-2M\la$, is positive for $x\in (0,\infty)$.
\[
\frac{d}{dx}\left(P-\frac{1}{P}-4x\la-2M\la\right)= 2{\la}{\left(P+\frac{1}{P}-2\right)}>0 \quad\text{for }x>0,
\]
where we used the fact that $L>1$ and that $P\mapsto P+P^{-1}-2$ is a non-negative function vanishing only at $P=1$ (notice that, by definition $P>1$ for positive $x$).
We now claim that 
$$
\restr{\left(P-\frac{1}{P}-4x\la-2M\la\right)}{x=0}= L^M-\frac{1}{L^M}-2M \la\ge 0.
$$
This can be proven by an analogous reasoning: treating $M$ as a real variable and differentiating with respect to it yield
$$
\frac{d}{dM}\left( L^M-\frac{1}{L^M}-2M \la \right) = \la\left(L^M+\frac{1}{L^M}-2\right)\ge 0
$$
(notice that the equality holds only when $M=0$), moreover,
$$
\restr{\left( L^M-\frac{1}{L^M}-2M \la \right)}{M=0}=0,
$$
which proves the claim.
\item Differentiating the expression for $E''_\ro(x)$ in \eqref{in out res li seme} by $x$ yields the following
$$
\frac{d}{dx}E''_\ro(x)= \frac{2\left(a(x)+\sg_c b(x)+\sg_c^2 x^2 c(x)\right)}{F^2},
$$
where we have set
\begin{equation}\nonumber
\begin{aligned}
a(x)&:=x^2 P^{-1}+ M(2x+M) - (x+M)^2 P - 2\lambda (2x^3+3 Mx^2+M^2 x),\\
b(x)&:=-2x^2 P^{-1} - M(2x+M) -2(M x + x^2) P + 2\lambda M (Mx + 2x^2),\\
c(x)&:= P^{-1} - P + 2\lambda (M+2x).
\end{aligned}
\end{equation}
In order to prove claim (ii) of the lemma, it will be sufficient to show that $a(x)<0$, $b(x)<0$ and $c(x)<0$ for all $x> 0$.

We have
$$
\restr{a(x)}{M=0}= x^2\underbrace{(L^{-2x}-L^{2x})}_{<0}-4\lambda x^3 <0.
$$
Treating now $M$ as a real variable and differentiating yields:
$$
\frac{d}{dM}a(x)= -\lambda x^2 P^{-1}+ 2(x+M)\underbrace{(1-L^M)}_{<0} - \lambda (x+M)^2L^M- 2 \lambda(3x^2+2Mx)<0.
$$
This implies that $a(x)<0$ for all $x> 0$ and all $M\ge 0$. 

As far as $b(x)$ is concerned, we will decompose it further, as follows
$$
b(x)=-2x^2 P^{-1}- M(2x+M) + 2x \,\widetilde{b}(x), 
$$
where $\widetilde{b}(x):={ - (M+x)P + \lambda M(M+2x) }$.
We have $\widetilde{b}(0)=M(-L^M+\lambda M)$.
The quantity $-L^M+\lambda M$ is negative for all $M\ge 0$ because it takes the value $-1$ for $M=0$ and is a decreasing function of $M$. As a matter of fact, we have
$$
\frac{d}{dM}(-L^M+\lambda M)= -\lambda L^M + \lambda = \lambda (-L^M+1) <0.
$$
Hence $\widetilde{b}(0)<0$. We claim that $\widetilde{b}(x)$ is also decreasing in $x$, because
$$
\frac{d}{dx}\widetilde{b}(x)= -P -2\lambda (M+x)P + 2\lambda M = -P +2\lambda M (-P+1) - 2\lambda x P <0.
$$
We conclude that $\widetilde{b}(x)$ (and therefore also $b(x)$) is negative for $x\ge 0$.

Finally, we show that $c(x)<0$ for $x>0$. 
We have $c(0)=L^{-M}-L^M+2\lambda M$. We claim that this quantity is non-positive for all $M\ge 0$. Indeed
$$
\restr{c(0)}{M=0}=0, \quad\text{ and }\quad \frac{d}{dM}c(0)= -\lambda L^{-M} (L^M-1)^2<0. 
$$
Moreover, since
$$
\frac{d}{dx}c(x)= -2\lambda P^{-1}-2\lambda P +4 \lambda = -2\lambda (P-1)^2<0,
$$
we conclude that also $c(x)<0$ for $x>0$. 
This implies that the function $x\mapsto E''_\ro(x)$ is strictly decreasing in $(0,\infty)$, as claimed.
\end{enumerate}
\end{proof}

Moreover, by a simple calculation we can check that 
\[
E''_\pm(1)=2(1-\sg_c)/F(1)\quad \text{ and }\lim_{k\to\ali}E''_\pm (k)=-\ali.
\]
Now, by combining this observation with Lemma \ref{lemma decreasing}, we get the behavior of $E''_\ri$ and $E''_\ro$ (see also Figure \ref{grafici}).
\begin{proposition}[Behavior of $E''_\pm$]\label{behavior of E''_pm}
\begin{enumerate}[label=(\roman*)] Let $\sg_c>0$.
\item If $\sg_c>1$, then $E''_\pm(k)$ is negative for all integer $k\ge1$.
\item If $\sg_c=1$, then the two functions $E''_-$ and $E''_+$ behave differently from one another. Namely, $E''_-(k)=0$ for all integer $k\ge 1$. On the other hand, $E''_+(k)>0$ for all integer $k\ge2$, while $E''_+(1)=0$.
\item If $0<\sg_c<1$, then $E''_\pm$ are sign changing. Namely $E''_\pm(1)>0$, while $E''_\pm(k)<0$ for large enough $k\in\NN$. 
\end{enumerate}
\end{proposition} 
\begin{figure}[h]
\centering
\includegraphics[width=1.1\linewidth,center]{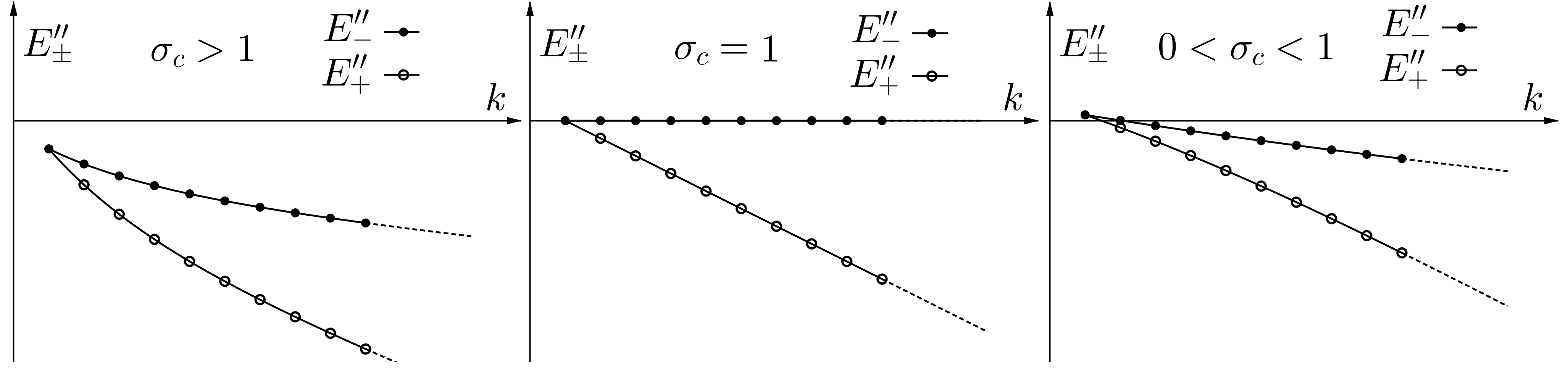}
\caption{The graphs of $E''_\pm$ for all possible values of $\sg_c$. Adapted from \cite{cava2}.
} 
\label{grafici}
\end{figure}

\subsection{Analysis of the resonance effects: proof of Theorem \ref{thm2}}
\label{ssec anal res}

Part $(iii)$ of Proposition \ref{behavior of E''_pm} tells us that $E''_\pm$ changes sign for $0<\sg_c<1$. This means that, by applying Proposition \ref{extension of pert}, we can actually construct perturbations $\Phi_1^\pm,\Phi_2^\pm\in\A_\pm^*$ such that $E''(D_0,\Om_0)(\Phi_1^\pm)>0$ and $E''(D_0,\Om_0)(\Phi_2^\pm)<0$. In other words, we have shown that $(D_0,\Om_0)$ is a \emph{saddle shape} for the functional $E$ under the volume preserving constraint (indeed, the barycenter preserving condition does not play any role in this). 

On the other hand, Proposition \ref{behavior of E''_pm} suggests that the radial configuration $(D_0,\Om_0)$ might be a local maximum for $E$ under the aforementioned constraints. This is actually the case. In order to show it, we will need the following lemma, that takes care of the resonance effects that arise when $\sg_c>1$. 
\begin{lemma}\label{lemma res then <=0}
Suppose that $\boldsymbol{\sg_c>1}$. For any $k\in\{1,2,\dots\}$ and $i\in\{1,\dots, d_k\}$ that satisfy $\al_{k,i}^\ri \al_{k,i}^\ro\ne 0$, we get: 
$$
\left(\al_{k,i}^\ri\right)^2E''_\ri(k)+\left(\al_{k,i}^\ro\right)^2E''_\ro(k)+\al_{k,i}^\ri \al_{k,i}^\ro E''_{\rm res}(k) \le0, 
$$
where equality holds if and only if $k=1$ (see Figure {\rm \ref{pentagoni}}, case V).
\end{lemma}
\begin{proof}
Since, by hypothesis, $\al_{k,i}^\ro\ne0$, we can put $t:=\al_{k,i}^\ri/\al_{k,i}^\ro$. For $k$ fixed, we study the following quadratic polynomial in $t$:
$$
Q(t):=E''_\ri(k) t^2 + E''_{\rm res} (k) t +E''_\ro(k).
$$
It can be checked that the discriminant of $Q$ is 
$$
\De= \underbrace{\frac{-16 (\sg_c-1)(k-1) R^N}{\sg_c N^2 F^2}}_{\le0} \underbrace{\left( \sg_c k (R^{2-N-2k}-1)+ (N-2+k)R^{2-N-2k}+k \right)}_{>0}\cdot G,
$$
where we have set $G:=(\sg_c-1)k (N-1+k) (R^{2-N-2k}-1)+(N-2+2k)R^{2-N-2k}$. 
We see immediately that $G>0$, as $\sg_c>1$ by hypothesis. We will distinguish two cases. If $k>1$, then $\De<0$ and therefore the quadratic polynomial $Q(t)$ has no real roots. Since $Q(0)=E''_\ro(k)<0$ (see Proposition \ref{behavior of E''_pm} and Figure \ref{grafici}), then $Q$ must be strictly negative for all other values of $t$ as well. If $k=1$, then $\De=0$, which means that $Q(t)$ has one double root (which actually corresponds to $t=1$). We conclude as before. 
\end{proof}
\begin{figure}[h]\label{pentagoni}
\includegraphics[width=1.0\linewidth, center]{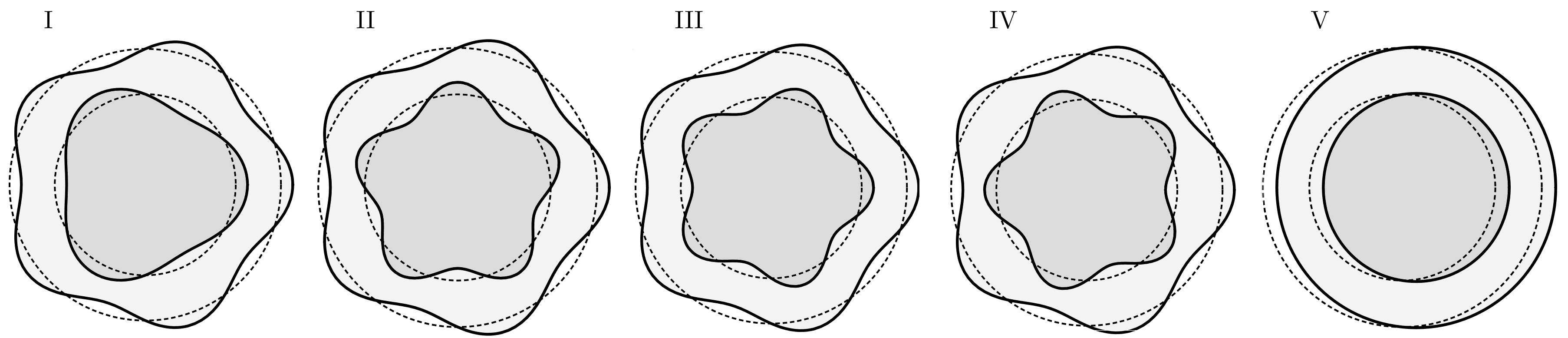}
\caption{How $(D_t,\Om_t)$ looks like for simple perturbations corresponding to
$(h_\ri\cdot n)(R\cdottone)= \al Y_{k,i}(\cdottone)$ and $h_\ro \cdot n = \be Y_{m,j}$,
for the following values of $k,i,m,j$ and $\al, \be$:\newline
I: $k=3, m=5$. II: $k=m=5$, $i\ne j$. III: $k=m=5$, $i=j$, $\al \be >0$. IV: $k=m=5$, $i=j$, $\al\be <0$. V: $k=m=1$, $i=j$, $\al=\be\ne0$.
Notice that resonance effects appear in cases III, IV and V only. Moreover, as shown in Lemma \ref{lemma res then <=0}, V is the only case when $E''(\Phi)=0$ for $\sg_c\ne 1$. Reprinted from \cite{cava2}.
}
\label{pentagoni}
\end{figure}

We notice that, for all $\Phi\in\A^*$, by \eqref{bar 1st} (see Remark \ref{lastrmk}) the coefficients $\al_{1,i}^\ro$ that appear in the expansion \eqref{h_in h_out exp} must vanish for $i=1,\dots,N$ (in particular, we are able to avoid the case V of Figure \ref{pentagoni} by considering $\Phi\in\A^*$).
Combining this observation with the one at the beginning of this subsection, yields the main result of this chapter. 
\begin{theorem}
If $\sg_c>1$, then $E''(D_0,\Om_0)(\Phi)<0$ for all $\Phi\in\A^{*}$. In other words the configuration $(D_0,\Om_0)$ is a local maximum for the functional $E$ under the volume-preserving and barycenter-preserving constraint. 
If $0<\sg_c<1$, then there exist two perturbation fields $\Phi_1,\Phi_2\in\A^*$ such that $E''(D_0,\Om_0)(\Phi_1)>0$ and $E''(D_0,\Om_0)(\Phi_2)<0$. In other words, the configuration $(D_0,\Om_0)$ is a saddle shape for the functional $E$ under the volume and barycenter-preserving constraint. Notice that for $\sg_c=1$ we recover a local version of P\'olya's result Theorem {\rm\ref{thm polya}}, namely $E''(D_0,\Om_0)(\Phi_+)<0$ for all $\Phi_+\in\A_+^*$. 
\end{theorem}

Finally, we would like to give some remarks on the results of our computations in the case $k=1$. 
It corresponds to studying a pair of possibly distinct (volume preserving perturbations that, up to the second order, are indistinguishable from) translations acting on $\pa D_0$ and $\pa \Om_0$ simultaneously. We know that the functional $E$ is invariant under rigid motions, i.e. $E(D,\Om)=E\big(T(D),T(\Om)\big)$ for all rigid motions $T:\rn\to\rn$. In turn this implies that for fixed $x_0\in\rn$ and $t\ge0$:
\begin{equation}\nonumber
E(D_0+tx_0, \Om_0)=E(D_0,\Om_0-tx_0)=E(D_0,\Om_0+tx_0).
\end{equation}
Therefore by differentiating twice with respect to $t$, we get $E''_-(1)=E''_+(1)$ (see also Figure \ref{grafici} on page \pageref{grafici}), as we obtained by direct computation right after the proof of Lemma \ref{lemma decreasing}. Take now two orthogonal directions, say $e_1$ and $e_2$. We have
\begin{equation}\label{onlyone}
E(D_0+t e_1, \Om_0+te_2)=E(D_0+t(e_1-e_2),\Om_0)=E(D_0+\sqrt{2}t e_1, \Om_0),
\end{equation}
and thus, 
\begin{equation}
\begin{aligned}
\restr{\frac{d^2}{dt^2}}{t=0}{E\Big(D_0+t e_1, \Om_0+te_2\Big)}= 2\,\restr{\frac{d^2}{dt^2}}{t=0} {E\Big(D_0+te_1,\Om_0\Big)}=\\
\restr{\frac{d^2}{dt^2}}{t=0} {E\Big(D_0+te_1,\Om_0\Big)}+\restr{\frac{d^2}{dt^2}}{t=0} {E\Big(D_0,\Om_0+t e_2\Big)}, 
\end{aligned}
\end{equation}
i.e. second order shape derivatives ``behave linearly'' in this case.
On the other hand, if, unlike \eqref{onlyone}, we apply the same translation to both $D_0$ and $\Om_0$, then the value of $E$ does not get altered (recall that $E''_{\rm res}(1)=-2E''_-(1)$, see for example \eqref{in out res li seme}). Hopefully, this observations might serve as an intuitive explanation of the geometrical meaning of the resonant part $E''_{\rm res}$ and the inevitability thereof. 

\begin{figure}[h]
\centering
\includegraphics[width=0.7\linewidth,center]{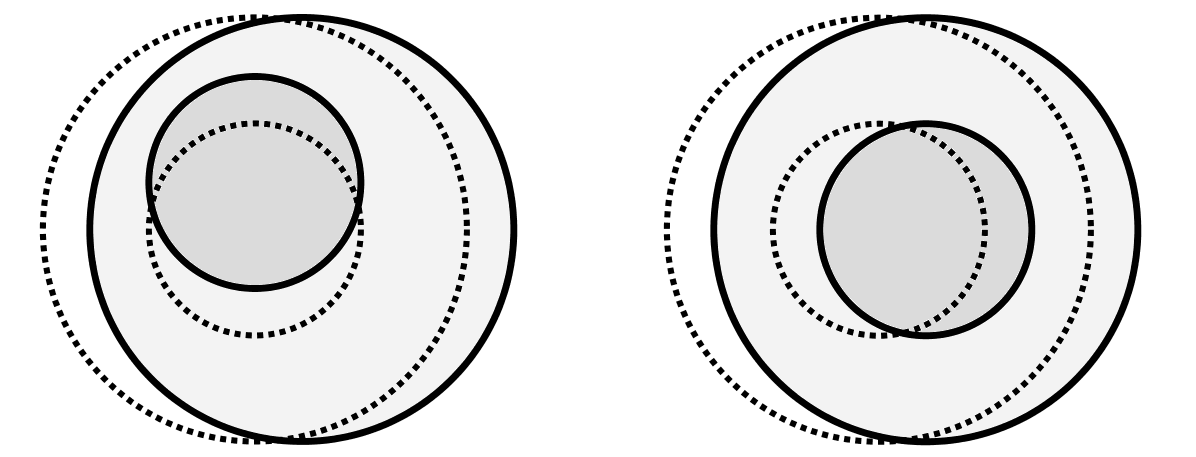}
\caption{Example of non-resonance (left) and resonance (right) due to the combined effect of two translations.
} 
\label{due}
\end{figure}

\chapter{A two-phase overdetermined problem of Serrin-type}\label{2-ph serrin}
\label{ch 2-ph serrin}

In this chapter, we will obtain the proof of Theorem \ref{thm3} by a perturbation argument. This is one of a series of results about two-phase overdetermined problems that were obtained in \cite{camasa}.
Let $D$, $\Omega\subset\rn$ be two bounded domains of class $C^{2+\al}$ with $\overline{D}\subset \Omega$. 
We look for a pair $(D,\Om)$ for which the overdetermined problem \eqref{pb 2-ph serrin} has a solution for some positive constant $d$. As remarked in Chapter \ref{ch intro}, it is sufficient to examine \eqref{pb 2-ph serrin} with $\sg_s=1$ in the form
\begin{eqnarray}
&&\dv(\sigma \nabla u)=\beta u - \gamma<0\quad\mbox{ in }\ \Omega, \label{pbcava eq}
\\
&&u=0 \ \quad\qquad\qquad\qquad\qquad \mbox{ on } \partial\Omega, \label{pbcava dirichlet}
\\ 
&&\partial_n u=-d \ \quad\qquad\qquad\qquad \mbox{ on } \pa\Omega,\label{pbcava neumann}
\end{eqnarray}
where $\be\ge 0$, $\ga>0$, and $\sg=\sg_c \chi_D+\chi_{\Om\setminus D}$.
By the divergence theorem, the constant $d$ is related to the other data of the problem by the formula:
\begin{equation}
\label{whatisd}
d=\frac1{\per(\Om)}\left\{\ga\,\vol\Ome-\beta\,\int_{\Om} u\right\},
\end{equation}
where, the functionals $\vol(\cdottone)$ and $\per(\cdottone)$ have been defined in Example \ref{ex l1}.
\par
It is obvious that, for all values of $\sg_c>0$, the pair $(B_R, B_1)$ is a solution to the overdetermined problem \eqref{pbcava eq}--\eqref{pbcava neumann} for some $d$.
We will look for other solution pairs of \eqref{pbcava eq}--\eqref{pbcava neumann} near $(B_R, B_1)$ by a perturbation argument which is based on Theorem \ref{ift}, page \pageref{ift}.

\section{Preliminaries}
We introduce the functional setting for the proof of Theorem \ref{thm3}. As done in Chapter 4, we set $D_0:=B_R$ and $\Om_0:=B_1$. For $\alpha\in(0,1)$, let
$\phi\in \C^{2+\al}(\rn, \rn)$ satisfy that $\mathrm{Id}+\phi$ is a diffeomorphism from $\rn$ to $\rn$, and
\begin{equation}\label{phi f g}
\phi=f\,n \ \mbox{ on } \ \pa D_0, \qquad \phi=g\,n \ \mbox{ on } \ \pa\Om_0,
\end{equation}
where $f$ and $g$ are given functions of class $C^{2+\al}$ on $\pa D_0$ and $\pa\Om_0$, respectively, and $n$ indistinctly denotes the outward unit normal to both $\pa D_0$ and $\pa\Om_0$.
Next, we define the sets
$$
\Om_g=(\mathrm{Id}+\phi)(\Om_0) \ \mbox{ and } \ D_f=(\mathrm{Id}+\phi)(D_0).
$$ 
If $f$ and $g$ are sufficiently small, $D_f$ and $\Om_g$ satisfy $\ol{D_f}\subset\Om_g$. 
\par
Now, we consider the Banach spaces (equipped with their standard norms, that will be denoted by $\norm{\cdottone}$):
\begin{eqnarray}\nonumber
& \cF=\setbld{f\in C^{2+\al}(\pa D_0)}{\int_{\pa D_0} f\, dS =0},\quad 
\cG=\setbld{g\in C^{2+\al}(\pa\Om_0)}{\int_{\pa\Om_0} g\, dS =0}, \\
&\cH=\setbld{h\in C^{1+\al}(\pa\Om_0)}{\int_{\pa\Om_0} h\, dS =0}.\nonumber
\end{eqnarray}
In order to be able to use Theorem \ref{ift} on page \pageref{ift}, we introduce a mapping $\Psi: \cF\times \cG\to \cH$ by:
\begin{equation}\label{Psi li seme}
\Psi(f,g)=\left\{ \pp{\gr u_{f,g}\cdot n_g}\circ \pp{\id+g\,n} + d_{f,g}\right\} J_\tau({g}) \ \mbox{ for } \ (f,g)\in\cF\times\cG.
\end{equation}

Here, $u_{f,g}$ is the solution of \eqref{pbcava eq}--\eqref{pbcava dirichlet} with $\Om=\Om_g$ and $\sg=\sg_c\,\chi_{D_f}+\chi_{\Om_g\setminus D_f}$ and $d_{f,g}$ is computed via \eqref{whatisd}, with $\Om=\Om_g$ and $u=u_{f,g}$. 
Also, $n_g$ is the outward unit normal to $\pa\Om_g$ (hence we will agree that $n_g=n$ for $g\equiv0$).
Finally, the term $J_\tau(g)>0$ is the tangential Jacobian associated to the transformation $x\mapsto x+g(x)\,n(x)$ (see \eqref{tg Jacobian}): this term ensures that the image $\Psi(f,g)$ has zero integral over $\pa\Om_0$ for all $(f,g)\in \cF\times \cG$, as an integration of \eqref{pbcava neumann} on $\pa\Om_g$ requires, when $d=d_{f,g}$.
\par
Thus, by definition, we have $\Psi(f,g)=0$ if and only if the pair $(D_f,\Om_g)$ solves \eqref{pbcava eq}--\eqref{pbcava neumann}. Moreover, we know that the mapping $\Psi$ vanishes at $(f_0, g_0)=(0,0)$. 

\section{Computing the derivative of $\Psi$}
The first step will consist in proving the Fr\'echet differentiability of $\Psi$. 
\begin{lemma}
The map $\Psi:\cF\times\cG\to\cH$, defined in \eqref{Psi li seme} is Fr\'echet differentiable in a neighborhood of $(0,0)\in\cF\times\cG$. 
\end{lemma}
\begin{proof}
In order to show the differentiability of $\Psi$, we will resort to the machinery developed in Chapter \ref{ch shape derivatives}. As the elements of $\cF$ and $\cG$ are only defined on the surface of spheres we first need to ``extend'' them to suitable perturbations in the whole $\rn$ in order to proceed. To this end consider $\phi\in\C^{2+\al}(\rn,\rn)$. We can rewrite an analogous formulation of \eqref{Psi li seme} for perturbations of the whole $\rn$:
\[
\widehat{\Psi}(\phi):=\restr{\Big\{ \pp{\gr u_\phi\cdot n_\phi} \circ \pp{\id+\phi} +d_\phi \Big\} J_\tau (\phi)}{\pa\Om_0}, 
\] 
where the $\phi$ subscript is used in the natural way, i.e. as in \eqref{Psi li seme} according to the notation introduced in \eqref{phi f g}. Moreover, notice that, under \eqref{phi f g} we have 
\begin{equation}\label{wt nakutemo ii}
\widehat{\Psi}(\phi)=\Psi(f,g).
\end{equation}
It is enough to inspection the differentiability of each ``piece'' of $\widehat{\Psi}$ and then conclude by composition. Put $v_\phi:=u_\phi\circ\pp{\id+\phi}$, we have 
\[
\gr u_\phi\circ \pp{\id+\phi}=\pp{I+D\phi}^T\gr v_\phi,
\] 
which is differentiable in a neighborhood of $0$ by Theorem \ref{thm diff state func 1}. The map $\phi\mapsto n_\phi\circ\pp{\id+\phi}$ is differentiable by Proposition \ref{der of normal}. The function $d_\phi$, defined as in \eqref{whatisd} with the obvious modifications, is also differentiable (its derivative can be computed by the Hadamard formula, see Example \ref{ex l1} for the details about $\per(\cdottone)$ and $\vol(\cdottone)$). Finally, since $J_\tau$ is also differentiable by Lemma \ref{J tau diff}, the proof of the differentiability of $\widehat{\Psi}$ (and thus that of $\Psi$) is complete.
\end{proof}
We will now proceed to the actual computation of $\pa_f \Psi (0,0)$. Since $\Psi$ is Fr\'echet differentiable, $\pa_f \Psi (0,0)$ can be computed as a G\^ateaux derivative:
$$
\pa_f\Psi(0,0)(f)= \lim_{t\to 0} \frac{\Psi(t f,0)-\Psi(0,0)}{t} \ \mbox{ for } \ f\in\cF.
$$
\par
From now on, we fix $f\in\cF$, set $g=0$ and, to simplify notations, we will write $D_t, u_t, d(t)$ in place of $D_{tf}, u_{tf,0}, d_{tf,0}$. As done previously, we will still write $u$ for $u_0$. 
The following characterization of the shape derivative of $u_t$ is a direct consequence of Theorem \ref{u' general thm}. 
\begin{lemma}
For every $f\in\cF$, the shape derivative $u'$ of $u_t$ solves the following:
\begin{eqnarray}
\label{pbu' eq}& \sg \De u'=\beta u' &\quad\mbox{ in }\ D_0\cup (\Om_0\setminus \overline{D_0}), \\
\label{pbu' flux}& [\sg \pa_n u']=0 &\quad \mbox{ on } \pa D_0,\\
\label{pbu' jump}& [u']=-[\pa_n u]f & \quad \mbox{ on }\pa D_0,\\ 
\label{pbu' bdary}&u'=0 &\quad\mbox{ on }\ \partial\Om_0. 
\end{eqnarray}
\end{lemma}

\begin{lemma}
\label{lem:shape-derivatives}
For all $f\in\cF$ we have 
$d'(0)=0$.
\end{lemma}

\begin{proof}
We rewrite \eqref{whatisd} as
$$
d(t)|\pa \Om_0|-\ga|\Om_0|=-\beta \int_{\Omega_0} u_t \, dS,
$$
then differentiate and evaluate at $t=0$. The derivative of the left-hand side equals 
$d'(0)\,|\pa\Om_0|$.
Thus, we are left to prove that the derivative of the function defined by
$$
I(t)=\int_{\Omega_0} u_t\, dx
$$
vanishes at $t=0$.
\par
To this aim, since $u_t$ solves \eqref{pbcava eq} for $D=D_t$, we multiply both sides of this for $u_t$ and integrate to obtain that
\begin{equation}\nonumber
\ga\,I(t)=\gamma \int_{\Omega_0} u_t\, dx = \beta \int_{\Om_0} u_t^2\, dx+ \sg_c\int_{D_t} \, \abs{\gr u_t \,}^2 dx+\int_{\Omega_0\setminus\overline{D_t}} \abs{\gr u_t \,}^2 dx,
\end{equation}
after an integration by parts.
Thus, the desired derivative can be computed by using the Hadamard formula (Proposition \ref{hadam form})
\begin{equation*}
\label{gamma int u prime}
\begin{aligned}
\ga\,I'(0)&=2\beta \int_{\Omega_0} u u' 
+2\int_{\Om_0} \sg \gr u \cdot \gr u' 
+ \int_{\partial D_0}[\sg |\partial_n u|^2]f 
\\
&=
2\beta \int_{\Omega_0} u u' 
+ 2 \int _{\Omega_0} \sg \gr u \cdot \gr u' 
=0.
\end{aligned}
\end{equation*}
Here, in the second equality we used that the jump of $\sg |\dn u|^2$ is constant on $\pa D_0$ and that $f\in\cF$, while, the third equality ensues by integrating \eqref{pbu' eq} against $u$.
\end{proof}

\begin{theorem}
\label{prop psi'}
The Fr\'echet derivative $\pa_f \Psi (0,0)(\cdottone,0):\cF \to \cH$ is defined by the formula
$$
\pa_f\Psi (0,0)(f)=\pa_n u',
$$
where $u'$ is the solution of the boundary value problem \eqref{pbu' eq}--\eqref{pbu' bdary}.
\end{theorem}
\begin{proof}
Since $\Psi$ is Fr\'echet differentiable, we can compute $\pa_f \Psi$ as a G\^ateaux derivative as follows:
$$
\partial_f\Psi(0,0)(f)=\restr{\frac{d}{dt}}{t=0}\Psi(t f,0)=\restr{\frac{d}{dt}}{t=0}\left\{ \gr u_t(x)\cdot n(x) + d(t) \right\}J_\tau (0).
$$
Since $J_\tau (0)=1$, the thesis is a direct consequence of Lemma \ref{lem:shape-derivatives} and definition \eqref{def shape der of state function}.
Finally, the fact that this mapping is well-defined (i.e. $\partial_n u'$ actually belongs to $\mathcal{H}$ for all $f\in\mathcal{F}$) follows from the calculation 
$$
\int_{\partial\Omega_0} \partial_n u' = \int_{\Omega_0} \dv(\sg \gr u')=\beta \int_{\Omega_0} u'=
\be\, I'(0)=0,
$$ 
where we also used \eqref{pbu' eq}--\eqref{pbu' bdary}.
\end{proof}

\section{Applying the implicit function theorem}

Here we give the main result of this Chapter, which clearly implies Theorem \ref{thm3}.

\begin{theorem}\label{mainthm cava}
There exists $\varepsilon>0$ such that, for all $g\in\mathcal{G}$ with $\norm{g}<\varepsilon$ there exists a unique $f(g)\in\mathcal{F}$ such that the pair $(D_{f(g)},\Omega_{g})$ is a solution of the overdetermined problem \eqref{pbcava eq}--\eqref{pbcava neumann}.
\end{theorem}
\begin{proof}
This theorem consists of a direct application of Theorem \ref{ift} on page \pageref{ift}.
We know that the mapping $(f,g)\mapsto\Psi(f,g)$ is Fr\'echet differentiable and we computed its Fr\'echet derivative with respect to the variable $f$ in Theorem \ref{prop psi'}. We are left to prove that the mapping $\pa_f\Psi(0,0):\cF\to\cH$, given in Theorem \ref{prop psi'}, is a bounded and invertible linear transformation.
\par
Linearity and boundedness of $\pa_f\Psi(0,0)$ ensue from the properties of problem \eqref{pbu' eq}--\eqref{pbu' bdary}. 

We are now going to prove the invertibility of $\pa_f \Psi(0,0)$. To this end we study the relationship between the spherical harmonic expansions of the functions $f$ and $u'$ (see Appendix \ref{app sphar} for notations and properties of the harmonic functions). Suppose that, for some real coefficients $\al_{k,i}$ the following holds
\begin{equation}\label{f in sphar}
f(R\theta) = \sum_{k=1}^\infty \sum_{i=1}^{d_k} \al_{k,i} Y_{k,i}(\theta), \quad \mbox{ for } \theta\in\mathbb{S}^{N-1}. 
\end{equation} 
Under the assumption \eqref{f in sphar}, we can apply the method of separation of variables to get
\begin{equation}\label{u' in sphar}
u'(r\theta)= \sum_{k=1}^\infty\sum_{i=1}^{d_k}\al_{k,i}s_k(r)Y_{k,i}(\theta), \quad \mbox{ for }r\in (0,R)\cup (R,1) \mbox{ and } \theta \in\mathbb{S}^{N-1}.
\end{equation}
Here $s_k$ denotes the solution of the following problem:
\begin{eqnarray}
&&\sg\left\{ \pa_{rr} s_k+ \frac{N-1}{r}\,\pa_r s_k - \frac{k(k+N-2)}{r^2}s_k \right\} = \be s_k \,\mbox{ in } (0,R)\cup (R,1), \label{the 2nd order ODE}\\
&& s_k(R^+)-s_k(R^-)=\pa_r u(R^-)-\pa_r u(R^+), \quad \sg_c\, \pa_r s_k(R^-)= \pa_r s_k(R^+), \nonumber\\ 
&& s_k(1)=0,\qquad s_k(0)=0, \nonumber
\end{eqnarray}
where, by a slight abuse of notation, the letters $\sg$ and $u$ denote the radial functions
$\sg(|{x}|)$ and $u(|{x}|)$ respectively. Notice that the condition $s_k(0)=0$ derives from the fact that $s_k$ is non-singular at $r=0$. Indeed, this ensues by multiplying \eqref{the 2nd order ODE} by $r^2$ and letting $r \to 0$. 
By \eqref{u' in sphar} we see that $\pa_f \Psi(0,0)$ preserves the eigenspaces of the Laplace--Beltrami operator, and, in particular, $\pa_f\Psi(0,0)$ is invertible if and only if $\pa_r s_k(1)\ne0$ for all $k\in\{1,2,\dots\}$. Let us show the latter. Suppose by contradiction that $\pa_r s_k(1)=0$ for some $k\in\{1,2,\dots\}$. Then, since $s_k(1)=0$, by the unique solvability of the Cauchy problem for the ordinary differential equation \eqref{the 2nd order ODE}, 
$s_k \equiv 0$ on the interval $[R,1]$. Therefore $\pa_r s_k(R^+)=0$ and thus also $\pa_r s_k(R^-)=0$. Therefore, in view of \eqref{the 2nd order ODE}, we see that $s_k$ achieves neither its positive maximum nor its negative minimum on the interval $[0,R]$. Thus $s_k \equiv 0$ also on $[0,R]$.
On the other hand, since $\sg_c\ne1$, we see that $[\dn u]\ne 0$ on $\pa D_0$ and hence $s_k(R^-)\not=0$, which is a contradiction. \end{proof}

Lastly, we remark that the volumes of the domains $D_{f(g)}$ and $\Om_g$, found by Theorem \ref{mainthm cava}, do not necessarily coincide with those of $D_0$ and $\Om_0$ (and the same goes for surface areas). This is because only volume preserving conditions {\bf at first order} were prescribed in the definitions of $\cF$ and $\cG$. Nevertheless, the arguments of Theorem \ref{mainthm cava} can be refined to gain the control on the domains' volume (or surface area, for the matter). 
\begin{corollary}
There exists $\varepsilon>0$ such that, for all $\wt{g}\in\mathcal{C}^{2+\al}(\pa\Om_0)$ with $\norm{\wt{g}}<\varepsilon$ and such that $\vol(\Om_{\wt{g}})=\vol(\Om_0)$, there exists a unique $\wt{f}=\wt{f}(\wt{g})\in\mathcal{C}^{2+\al}(\pa D_0)$ with $\vol(D_{\wt{f}})=\vol(D_0)$ such that the pair $(D_{\wt{f}},\Omega_{\wt{g}})$ is a solution of the overdetermined problem \eqref{pbcava eq}--\eqref{pbcava neumann}. An analogous result holds true when every occurrence of $\vol(\cdottone)$ is replaced by $\per(\cdottone)$ in the statement above. 
\end{corollary}
\begin{proof}
First of all, for any $g\in\cG$ small enough we will construct a domain $\wt{\Om_g}$ such that $\vol(\wt{\Om_g})=\vol(\Om_0)$. As done in Proposition \ref{extension of pert}, we set $\wt{\Om_g}:=t\Om_g$ for $t=\sqrt[\leftroot{-1}\uproot{5}\scriptstyle N]{\frac{\Om_0}{\vol{\Om_g}}}$. If $g$ is small enough, then $\wt{\Om_g}=\Om_{\wt{g}}$ for some $\wt{g}\in\C^{2+\al}(\Om_0)$. The map $g\mapsto \wt{g}$ is continuous in a neighborhood of $g=0$. Moreover, we claim that, for $g$ small enough, the map $g\mapsto \wt{g}$ is also invertible. Indeed, by definition $\pa \Om_{\wt{g}}=t\pa \Om_g$ for some $t$ to be determined. We have
\begin{equation}\label{find t}
x+\wt{g}(x)n(x)=t\pp{x+g(x)n(x)} \quad \text{ for all }x\in\pa\Om_0.
\end{equation}
There is only one value of $t$ such that $g$ in \eqref{find t} has vanishing integral over $\pa \Om_0$. Namely, since $n(x)=x$ on $\pa\Om_0$, we obtain:
\[
t=\frac{\per(\Om_0)+\int_{\pa\Om_0}\wt{g}}{\per(\Om_0)}.
\]
Notice that we can make $t$ arbitrarily close to $1$ by controlling the size of $\wt{g}$.
Of course, the same arguments work for $f\in\cF$ as well.
We define an auxiliary function $\wt{\Psi}(f,g):=\Psi(\wt{f},\wt{g})$, where, by a slight abuse of notation, we used the letter $\Psi$ to denote the extension of \eqref{Psi li seme} to $\C^{2+\al}(\pa D_0)\times\C^{2+\al}(\pa\Om_0)$. As remarked in Proposition \ref{extension of pert}, we have
\[
\dato \wt{tf}=f,
\]
in other words the perturbations $tf$ and $\wt{tf}$ are indistinguishable at first order. Therefore, the statement of Theorem \ref{mainthm cava} holds for the functional $\wt{\Psi}$ as well. Hence there exists some $\eps>0$ such that for all $g\in\cG$ with $\norm{g}<\eps$, there exists a unique $f=f(g)\in\cF$ such that $\wt{\Psi}(f,g)=\Psi(\wt{f},\wt{g})$.
Up to choosing a smaller $\eps>0$, we can conclude by the invertibility of the map $\wt{\cdottone}$.
\end{proof}

\chapter*{Acknowledgements}
\addcontentsline{toc}{chapter}{Acknowledgements}
First of all, I would like to thank my supervisor Professor Shigeru Sakaguchi for his continuous support and stimulating discussions on a wide range of topics. I admire his way of teaching even the most abstract mathematical concepts by always revealing the geometrical essence behind it.
I am grateful to Professor Rolando Magnanini: not only for his insightful comments and challenging mathematical problems but also for the precious help he gave me when I decided to go study in Japan.
I would like to express my gratitude to Ms. Chisato Karino and Ms. Sumie Narasaka for their support. I cannot avoid mentioning the financial support received by the Ministry of Education, Culture, Sports, Science and Technology (MEXT) and the Japan Society for the Promotion of Science (JSPS), without which all this would not have been possible. 
I would also like to thank all my lab mates, especially Toshiaki Yachimura for the innumerable hours of passionate mathematical discussions, which undoubtedly fueled my love for the subject (and made my Japanese improve). Special thanks go to my girlfriend Lin Zhu, one of the greatest sources of happiness and motivation in my life right now. 
Finally, I feel in need to thank my parents for their love, support and patience. I think that letting your son be free to go to the other side of the globe, where you cannot protect him directly, is an act of deep trust that goes against the natural instinct of a parent (especially an Italian one, I would say).

\begin{appendices}
\chapter{Elements of tangential calculus}\label{tang calc}
\renewcommand{\theequation}{\thechapter.\arabic{equation}}
\setcounter{theorem}{0}
\renewcommand{\thetheorem}{\Alph{chapter}\arabic{theorem}}
Let $\om$ be a bounded open set of class $\C^1$. For every $g\in\C^1(\pa\om)$ we define its \emph{tangential gradient} as 
\begin{equation}\label{tg grad}
\grt g :=\gr \widetilde{g}-(\gr \widetilde{g}\cdot n)n \quad \text{ on }\pa\om,
\end{equation}
where $\widetilde{g}$ is an extension of class $\C^1$ of $g$ to a neighborhood of $\pa\om$. Notice that, by density, the tangential gradient can be defined in the natural way for all functions $g\in W^{1,1}(\pa\om)$.
It is easy to see that this definition does not depend on the choice of the extension. Indeed, this is equivalent to showing that $\gr \widetilde{g}=(\gr\widetilde{g}\cdot n)n$ on $\pa\om$ for all $\widetilde{g}$ of class $\C^1$ on a neighborhood of $\pa\om$ with $\widetilde{g}\equiv 0$ on $\pa\om$.
To this end, fix a point $x_0\in\pa\om$ and take a smooth path $\ga:[0,1]\to\pa\om$ with $\ga(0)=x_0$. Since, by assumption, $\widetilde{g}(\ga(t))=0$ for all $t$, we have $\gr\widetilde{g}(x_0)\cdot \ga'(0)=0$. By the arbitrariness of $x_0$ and $\ga$ we conclude that $\gr \widetilde{g}$ is parallel to $n$ at each point of $\pa\om$, which was our claim. One obvious property of the tangential gradient is the following:
\begin{equation}\label{tg.nor=0}
\grt g\cdot n =0 \quad \text{ for all }g\in\C^1(\pa\om).
\end{equation}
Let $w\in\C^1(\pa\om,\rn)$. The \emph{tangential divergence} of $w$ is defined as 
\begin{equation}\label{tg div}
\dv_\tau w:=\dv \wt{w} - n\cdot \pp{D\wt{w}\,n},
\end{equation} 
where $\wt{w}$ is a $\C^1$ extension of $w$ to a neighborhood of $\pa\om$. This definition can by extended by density to vector fields $w\in W^{1,1}(\pa\om,\rn)$. 
Just like the tangential gradient, the definition of tangential divergence is independent of the extension chosen. Indeed one can verify that
\begin{equation}\label{trace appears}
\dv\wt{w}-n\cdot\pp{D\wt{w}n}=\tr(D_\tau w),
\end{equation}
where 
\begin{equation}\label{tg jacobian}
D_\tau w \text { is the matrix whose $i$-th row is given by }\grt w_i.
\end{equation} 
The following tangential versions of the Leibniz rule hold true: for all $f,g\in\C^1(\pa\om)$ and $w\in\C^1(\pa\om,\rn):$ 
\begin{equation}\label{L.R}
\grt(fg)=f\grt g+ g\grt f, \qquad \dv_\tau(gw)=g\,\dv_\tau w+w\cdot \grt g. 
\end{equation}
The first identity above follows directly from the definition of tangential gradient \eqref{tg grad}, while the second identity can be proved by applying the first one to each row of $D_\tau (g w)$ and then taking the trace (recall \eqref{trace appears}).
Tangential divergence is used to define the (additive) mean curvature $H$ (i.e. the sum of the principal curvatures) of a surface by means of the unit normal $n$:
\begin{equation}\label{H}
H:=\dv_\tau n. 
\end{equation} 
Actually, if $\om$ is an open set of class $\C^2$, then $\dv\, \wt{n}=H$ on $\pa\om$ for all unitary extensions $\wt{n}$ of class $\C^1$ of the outward unit normal $n$. Indeed $(D\wt{n})\wt{n}=0$ on $\pa\om$ because the norm of $\wt{n}$ is constant in a neighborhood of $\pa\om$. Therefore $\dv \wt{n}=\dv_\tau n=H$. 

Now, let $w_\tau$ denote the tangential part of a vector field $w$ on $\pa\om$, that is 
\begin{equation}\label{tg part}
w_\tau:=w-(w\cdot n)n \quad \text{ on }\pa\om.
\end{equation}
Let $\om$ of class $\C^2$ and $w\in W^{1,1}(\pa\om,\rn)$. By combining \eqref{tg part}, \eqref{L.R}, \eqref{tg.nor=0} and \eqref{H}, we get the following decomposition result for the tangential divergence. 
\begin{equation}\label{decomp tg div}
\dv_\tau w= \dv_\tau w_\tau + H w\cdot n \quad \text{ on }\pa\om.
\end{equation}
\begin{lemma}[Tangential Stokes formula]\label{tg stokes}
Let $\om$ be a bounded open set of class $\C^2$ and $w\in W^{1,1}(\pa\om,\rn)$. Then
\[
\int_{\pa\om} \dv_\tau w =\int_{\pa\om} Hw\cdot n.
\]
\end{lemma}
\begin{proof}
We would like to follow along the same lines as \cite[Chapter 8, Subsection 5.5, page 367]{SG}, where an elegant proof is given using shape derivatives. By density we might assume, without loss of generality, that $w\in \C^1(\pa\om,\rn)$. Moreover, in what follows the same notation $w$ will denote a $\C^1$ extension of $w$ to the whole $\rn$. 
Take now an Hadamard perturbation $\Phi(t)=t\xi n$ on $\pa\om$. By the divergence theorem applied to the perturbed domain $\om_t$, we have 
\begin{equation}\label{dec tg div}
\int_{\om_t} \dv w = \int_{\pa\om_t} w \cdot n_t, \quad \text{ for } t\ge 0 \text{ small}, 
\end{equation}
where $n_t$ is taken to be unitary. Differentiating both sides with the aid of the usual and surface Hadamard formulas (Proposition \ref{hadam form} and Corollary \ref{surface hadam form}) and Proposition \ref{der of normal} yields
\[
\int_{\pa\om} \dv w\, \xi= \int_{\pa\om} -w\cdot \grt\xi + \dn(w\cdot n)\xi + H w\cdot n \xi. 
\]
Suppose $\xi\equiv1$ on $\pa\om$. We get
\[
\int_{\pa\om} \dv w=\int_{\pa\om} \dn(w\cdot n)+ \int_{\pa\om} H w\cdot n. 
\]
Since $\dn(w\cdot n)=n\cdot (Dw\,n)$ the thesis follows by the definition of tangential divergence \eqref{tg div}. 
\end{proof}

Combining Lemma \ref{tg stokes} and the second identity of \eqref{L.R} yields 
\begin{equation}
\int_{\pa\om}w \cdot \grt g= -\int_{\pa\om}g\, \dv_\tau w +\int_{\pa\om}H g w\cdot n.
\end{equation} 
We will now introduce the last tangential differential operator of this appendix: the Laplace--Beltrami operator. For $\om$ of class $\C^2$, the Laplace--Beltrami operator, denoted by $\De_\tau$, is defined as
\begin{equation}\label{def lapt}
\De_\tau u=\dv_\tau (\grt u) \quad \text{ for }u\in W^{2,1}(\pa\om).
\end{equation}
\begin{proposition}[Decomposition of the Laplace operator]\label{decomp lapl}
Assume that $\om$ is an open set of class $\C^2$ and let $u\in\C^2(\ol{\om})$, then
\begin{equation}\label{dec lap}
\De u= \pa_{nn} u+H\dn u+\De_\tau u\quad\ton\pa\om,
\end{equation}
where $\pa_{nn}u:=n\cdot(D^2 u \,n)$. Notice that, by density, \eqref{dec lap} can remains true for functions $u\in H^3(\om)$.
\end{proposition}
\begin{proof}
By definition of tangential divergence we have
\[
\De u=\dv(\gr u)=\dv_\tau(\gr u)+n\cdot (D(\gr u)n)\quad\ton\pa\om.
\]
We conclude noticing that, by \eqref{decomp tg div}, 
\[
\dv_\tau(\gr u)=\dv_\tau(\grt u)+H\dn u= \De_\tau u + H\dn u.
\]
\end{proof}

We conclude by stating a corollary of Lemma \ref{tg stokes}:
\begin{proposition}[Tangential integration by parts]\label{tg int by parts}
Assume that $\om$ is a bounded open set of class $\C^2$. For $f\in H^2(\om)$ and $g\in H^3(\om)$ the following holds
\[
\int_{\pa\om}\grt f\cdot \grt g= -\int_{\pa\om} f\De_\tau g.
\]
\end{proposition}
Notice that the formula above bears a striking resemblance to the usual integration by parts formula on open sets, and the absence of the ``boundary term'' is due to the fact that $\pa\om$ has no boundary. 

\chapter{Spherical harmonics}
\label{app sphar}

For integer $N\ge2$ and $k\ge0$, let $\PP_k(\RR^N)$ denote the set of all polynomial functions $\rn\to\RR$ whose degree is at most $k$. Moreover, let $\HH_k(\rn)$ denote the set of harmonic polynomials in $\PP_k(\rn)$. Lastly, let $\YY_k(\rn)$ denote the subset of polynomials in $\HH_k(\rn)$ that are also harmonic. 
$\YY_k(\rn)$ is a vector space over the reals; its dimension is finite and will be denoted by $d_k$. A combinatoric argument shows that 
\begin{equation}
d_0=1,\quad\quad d_k=\frac{(2k+N-2)(k+N-3)}{k!(N-2)!} \quad\tfor k\ge 1.
\end{equation}

We will now introduce the so-called \emph{harmonic decomposition} of a polynomial, it will be a key ingredient in proving Theorem \ref{cos}. We refer to \cite[Theorem 2.1, Chapter IV]{sw} for a proof. 
\begin{lemma}[Harmonic decomposition]\label{harm dec}
Every polynomial $p\in\PP_k(\rn)$ can be uniquely written in the form
\[
p=h_k+|x|^2h_{k-2}+\dots+ |x|^{2m} h_{k-2m}, 
\]
where $m=\left[k/2\right]$ and $h_i\in\HH_i(\rn)$ for each $i$.
\end{lemma}

Let $\YY_k(\SS^{N-1}):=\setbld{\restr{h}{\SS^{N-1}}}{h\in\YY_k(\rn)}$. Elements of $\YY_{k-1}(\SS^{N-1})$ are usually called \emph{spherical harmonics of degree} $k$ in the literature. Notice that every homogeneous polynomial $p$ of degree $k$ is uniquely determined by its restriction to $\SS^{N-1}$ by means of the relation
\begin{equation}\label{ext by hom}
p(x)=|x|^k \restr{p}{\SS^{N-1}}(x/|x|)\quad \tfor x\ne0.
\end{equation}
Therefore, we have \[\dim \YY_k(\SS^{N-1})=\dim \YY_k(\rn)=d_k.\]
Let now $\{Y_{k,i}\}_{i=1}^{d_k}$ denote an orthonormal basis of $\YY_k(\SS^{N-1})$. 
Another simple consequence of \eqref{ext by hom} is the following.
\begin{proposition}\label{prop sphar}
Spherical harmonics $Y_k\in\YY_k(\SS^{N-1})$ solve to the following eigenvalue problem on the unit sphere:
\begin{equation}\label{eq sphar}
-\De_\tau Y_k = \lambda_k Y_k \;\ton\SS^{N-1}, \quad \text{ where }\;\la_k=k(k+N-2).
\end{equation}
In particular, spherical harmonics of distinct degree are mutually orthogonal in $L^2(\SS^{N-1})$. 
\end{proposition}
\begin{proof}
Take an arbitrary $Y_k\in\YY_k(\SS^{N-1})$. By \eqref{ext by hom} we know that the extension $H_k(x):=|x|^k Y_k(x/|x|)$ is a harmonic function. Therefore, by Proposition \ref{decomp lapl} we can write
\[
0=\De H_k = k(k-1) Y_k + (N-1)k Y_k + \De_\tau Y_k\quad \ton \SS^{N-1}. 
\]
Rearranging the terms yields $-\De_\tau Y_k= k(k+N-2)Y_k$ on $\sss$.
Orthogonality will be proved in a classical way. Let $Y_j\in\YY_j(\SS^{N-1})$ and $Y_k\in\YY_k(\SS^{N-1})$ be two spherical harmonics corresponding to different indices $j\ne k$. By tangential integration by parts (Proposition \ref{tg int by parts}), we have
\[
\int_{\SS^{N-1}}\grt Y_j\cdot \grt Y_k = -\int_\sss Y_k \De_\tau Y_j = \la_j \int_\sss Y_j Y_k.
\]
Inverting the roles of $Y_j$ and $Y_k$ we get $\int_\sss \grt Y_j \cdot \grt Y_k= \la_k \int_\sss Y_j Y_k$. 
Since, by assumption, $\la_j\ne\la_k$, then $\int_\sss Y_j Y_k$ must vanish.
\end{proof}
We have the following result.
\begin{theorem}\label{cos}

The spherical harmonics $\{Y_{k,i}\}$ (where $k\in\{0,1,\dots\}$ and $i\in\{1,\dots,d_k\}$ for every $k$) form a complete orthonormal system in $L^2(\SS^{N-1})$.
\end{theorem}
\begin{proof}
Orthonormality is clear. We will now prove completeness. By invoking the density of $\C(\SS^{N-1})$ in $L^2(\SS^{N-1})$, it will be enough to show that
\[\C(\SS^{N-1})=\bigoplus_{k\ge0}\YY_k(\SS^{N-1}).\]
Now, by Stone-Weierstrass approximation theorem (see \cite[Theorem 2.2, page 36]{bressan}), we know that for every compact set $K\subset\rn$, continuous functions on $K$ can be approximated by polynomials in the max-norm with arbitrary precision. Finally, since every harmonic polynomial can be written as the sum of homogeneous harmonic polynomials, we conclude by Lemma \ref{harm dec}. 
\end{proof}
\begin{remark}\emph{
In particular, Theorem {\rm\ref{cos}} ensures that every solution of \eqref{eq sphar} can be written as the restriction to the unit sphere of some homogeneous harmonic polynomial. 
}\end{remark}

Theorem \ref{cos} can be applied in the computations regarding perturbations of the ball. 
For instance, let $\Phi=t h$ be an Hadamard perturbation acting on the unit ball $B_1$. Then, spherical harmonics form the ``right'' basis to work with. The function $h\cdot n:\sss\to \RR$ can be decomposed as $\sum_{k,i}\al_{k,i}Y_{k,i}$ for $\{\al_{k,i}\}_{k,i}\subset \RR$. This allows us to study the Hadamard perturbations generated by each spherical harmonic in the basis one by one and then recover the general case by (bi)linearity.
\begin{figure}[h]
\centering
\includegraphics[width=0.7\linewidth,center]{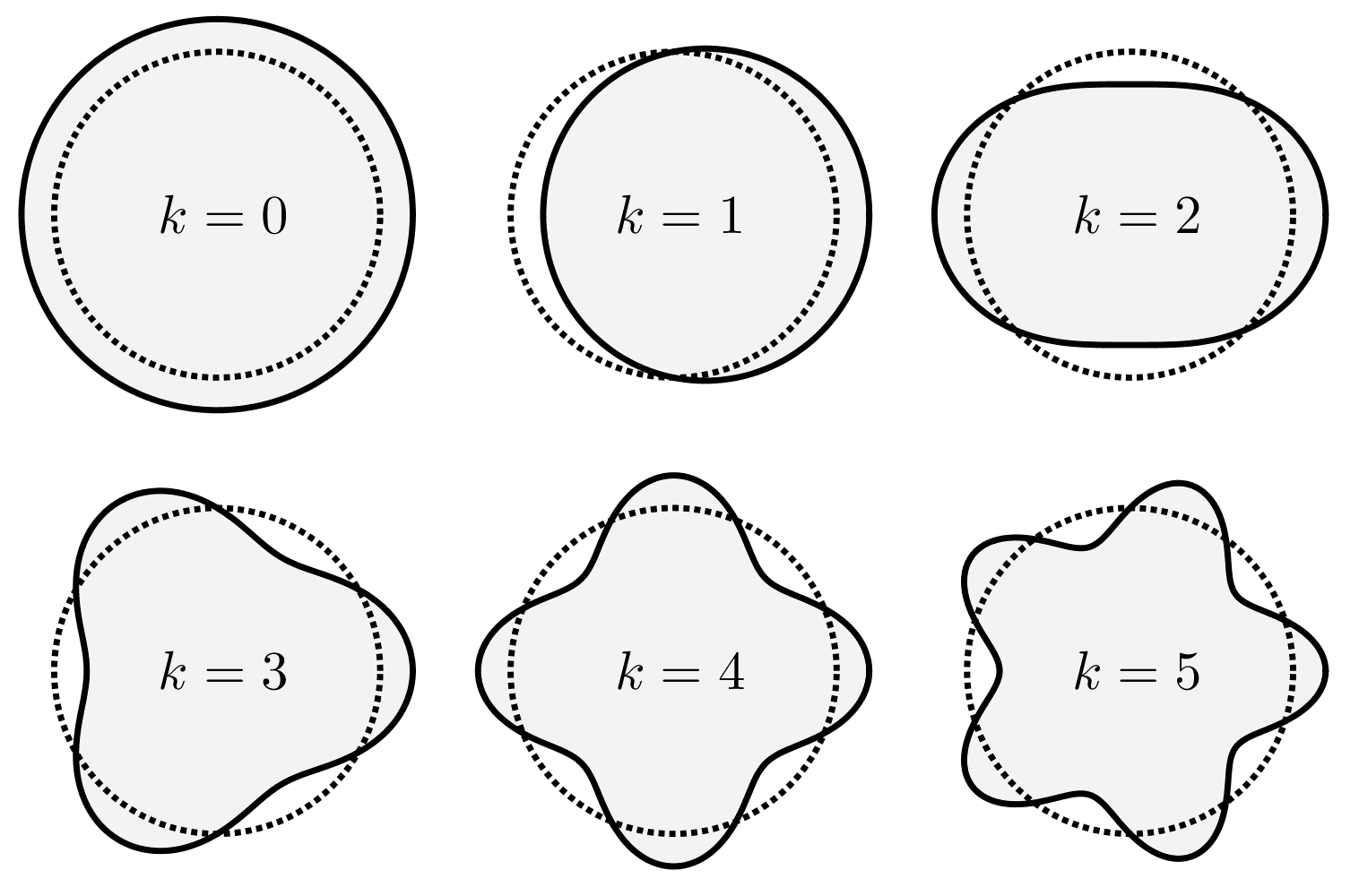}
\caption{How perturbations of the unit ball generated by spherical harmonics of various degree look (in two dimensions).
} 
\label{six}
\end{figure}
\begin{remark}\emph{\label{lastrmk}
A particular advantage brought by this approach is the following: geometrical constraints often assume an elegant form when rephrased using spherical harmonics. As a matter of fact (see Figure {\rm\ref{six}}), Hadamard perturbations generated by spherical harmonics of degree $k\ge0$ satisfy the first order volume preserving condition \eqref{1st} for all $k\ne 0$, while the first order barycenter preserving condition \eqref{bar 1st} is satisfied for all $k\ne 1$. The statement above follows immediately from the orthogonality relations among spherical harmonics proved in Proposition {\rm\ref{prop sphar}}. 
}\end{remark}
\end{appendices}

\end{document}